\documentclass[11pt,letterpaper]{amsart}

\usepackage{txfonts, amsmath,amstext,amsthm,amscd,amsopn,verbatim,amssymb, amsfonts}

\usepackage{hyperref}

\usepackage{xcolor}
\usepackage{diagbox}
\usepackage{array}
\newcommand{\PreserveBackslash}[1]{\let\temp=\\#1\let\\=\temp}
\newcolumntype{C}[1]{>{\PreserveBackslash\centering}p{#1}}
\newcolumntype{R}[1]{>{\PreserveBackslash\raggedleft}p{#1}}
\newcolumntype{L}[1]{>{\PreserveBackslash\raggedright}p{#1}}

\usepackage{subdepth}
\usepackage{tikz}
\usetikzlibrary{math}
\usetikzlibrary{positioning}

\usepackage{adjustbox}

\usepackage{todonotes}
\usepackage{multicol}

\usepackage[bbgreekl]{mathbbol}
\usepackage{enumerate}

\usepackage{float}
\restylefloat{table}

\usepackage{color, colortbl}
\usepackage{array}
\usepackage{varwidth} 

\definecolor{Gray}{gray}{0.9}

\newcolumntype{g}{>{\columncolor{Gray}}c}
\newcolumntype{M}{V{3cm}} 

\usepackage{tikz}
\usepackage{tikz-cd}
\usetikzlibrary{matrix}
\usetikzlibrary{shapes}
\usetikzlibrary{arrows, automata}
\usetikzlibrary{calc,3d}
\usetikzlibrary{decorations,decorations.pathmorphing,decorations.pathreplacing,decorations.markings}
\usetikzlibrary{through}

\tikzset{snakeit/.style={decorate, decoration={snake, amplitude=.2mm,segment length=1mm}}}

\tikzset{ext/.style={circle, draw,inner sep=1pt}, int/.style={circle,draw,fill,inner sep=2pt},nil/.style={inner sep=1pt}}
\tikzset{cy/.style={circle,draw,fill,inner sep=2pt},scy/.style={circle,draw,inner sep=2pt},scyx/.style={draw,cross out,inner sep=2pt},scyt/.style={draw,regular polygon,regular polygon sides=3,inner sep=0.95pt}}
\tikzset{exte/.style={circle, draw,inner sep=3pt},inte/.style={circle,draw,fill,inner sep=3pt}}
\tikzset{diagram/.style={matrix of math nodes, row sep=3em, column sep=2.5em, text height=1.5ex, text depth=0.25ex}}
\tikzset{diagram2/.style={matrix of math nodes, row sep=0.5em, column sep=0.5em, text height=1.5ex, text depth=0.25ex}}
\tikzset{rowcolsep/.style={column sep=.2cm, row sep=.1cm}}

\tikzset{
  crossed/.style={
    decoration={markings,mark=at position .5 with {\arrow{|}}},
    postaction={decorate},
    shorten >=0.4pt}}

\tikzset{every picture/.style={baseline=-.65ex} }
\tikzset{every loop/.style={draw}}

  \tikzset{->-/.style={decoration={
    markings,
    mark=at position .5 with {\arrow{>}}},postaction={decorate}}}

\theoremstyle{plain}
  \newtheorem{thm}{Theorem}
  \newtheorem{defi}[thm]{Definition}
  \newtheorem{prop}[thm]{Proposition}
  
  \newtheorem{cor}[thm]{Corollary}
  
  \newtheorem{lemma}[thm]{Lemma}
\theoremstyle{definition}
\newtheorem{ex}[thm]{Example}
\newtheorem{rem}[thm]{Remark}

\numberwithin{thm}{section}

\newcommand{\Hom}{\mathrm{Hom}}

\newcommand{\R}{{\mathbb{R}}}

\newcommand{\Z}{{\mathbb{Z}}}

\newcommand{\mF}{\mathcal{F}}

\DeclareMathOperator{\id}{id}

\newcommand{\op}{\mathcal}

\newcommand{\Lie}{\mathsf{Lie}}

\DeclareMathOperator{\coker}{coker}

\newcommand{\fGC}{\mathsf{fGC}}

\newcommand{\Ass}{\mathsf{Assoc}}

\newcommand{\bpm}{\begin{pmatrix}}
\newcommand{\epm}{\end{pmatrix}}

\newcommand{\GC}{\mathsf{GC}}

\newcommand{\HGC}{\mathsf{HGC}}

\newcommand{\bbS}{\mathbb{S}}

\DeclareMathOperator{\Aut}{Aut}

\DeclareMathOperator{\sgn}{sgn}
\DeclareMathOperator{\gr}{gr}

\newcommand{\MM}{{\mathcal M}}

\newcommand{\GK}{\mathsf{GK}}

\newcommand{\Q}{\mathbb{Q}}

\newcommand{\bMM}{\overline{\MM}}

\newcommand{\bGK}{\overline{\GK}{}}

\newcommand{\tGK}{\widetilde{\GK}{}}

\newcommand{\POp}{{\op P}}

\newcommand{\Pois}{{\mathsf{Pois}}}

\newcommand{\Embbar}{\overline{\mathrm{Emb}}}

\DeclareMathOperator{\im}{\mathrm{im}}

\newcommand{\XPhi}{\Xi}

\newcommand{\Xast}{X^{\star}}
\newcommand{\Jast}{J^{\star}}
\newcommand{\tJast}{\tilde J^{\star}}

\author{Sam Payne}
\address{Department of Mathematics \\ 2515 Speedway, PMA 8.100 \\ Austin, TX 78722 USA}
\email{sampayne@utexas.edu}
\author{Thomas Willwacher}
\address{Department of Mathematics\\ ETH Zurich\\ R\"amistrasse 101 \\ 8092 Zurich, Switzerland}
\email{thomas.willwacher@math.ethz.ch}

\begin{document}
\title{Weight two compactly supported cohomology of moduli spaces of curves}

\begin{abstract}
  We study the weight 2 graded piece of the compactly supported rational cohomology of the moduli spaces of curves $\MM_{g,n}$ and show that this can be computed as the cohomology of a graph complex that is closely related to graph complexes arising in the study of embedding spaces. For $n = 0$, we express this cohomology in terms of $W_0H_c^\bullet(\MM_{g',n'})$ for $g' \leq g$ and $n'\leq 2$, and thereby produce several new infinite families of nonvanishing unstable cohomology groups on $\MM_g$, including the first such families in odd degrees. In particular, we show that $ \dim H^{4g-k}(\MM_g)$ grows at least exponentially with $g$, for $k \in \{8, 9, 11, 12, 14, 15, 16, 18, 19\}$. 
\end{abstract}

 \thanks{
SP has been supported by NSF grants DMS--2001502 and DMS--2053261. Portions of this work were carried out while he was in residence at ICERM for the special thematic semester program on Combinatorial Algebraic Geometry, with support from NSF grant DMS--1439786 and the Simons Foundation Institute Grant Award ID 507536. TW has been supported by the ERC starting grant 678156 GRAPHCPX, and the NCCR SwissMAP, funded by the Swiss National Science Foundation.}

\maketitle

\setcounter{tocdepth}{1}
   \tableofcontents

\section{Introduction}

The weight zero compactly supported cohomology $W_0H_c^\bullet(\MM_{g,n})$ is naturally isomorphic to the cohomology of the genus $g$ part of the standard commutative graph complex with $n$ marked external vertices \cite{CGP1, CGP2}. Similar graph complexes arise in the study of rational homology and homotopy groups of higher dimensional long links \cite{TT1}.  In this paper, we use commutative graph complexes with additional decorations on the external vertices, similar to those appearing in the embedding calculus \cite{FTW2} to study the next nontrivial weight-graded piece of the compactly supported cohomology of $\MM_{g,n}$.  Throughout, all cohomology groups are taken with rational coefficients. 
 
The weight $k$ graded piece of $H^\bullet_c (\MM_{g,n})$ is $$\gr_k H_c^\bullet (\MM_{g,n}) := W_k H^\bullet_c(\MM_{g,n}) / W_{k-1} H^\bullet_c(\MM_{g,n}).$$  Note that $\gr_1 H^\bullet_c (\MM_{g,n})$ vanishes for all $g$ and $n$; this follows from the vanishing of $H^1(\overline \MM_{g,n})$ for all $g$ and $n$, via Deligne's weight spectral sequence. The next nontrivial weight graded piece is $\gr_2 H_c^\bullet(\MM_{g,n})$, which we study by combining graph complex techniques with the presentations and pullback formulas for $H^2 (\overline \MM_{g,n})$ from \cite{ArbarelloCornalba98, Keel92}.

Our first main result expresses $\gr_2 H^\bullet_c (\MM_{g,n})$ as the cohomology of a graph complex $X_{g,n}$  generated by possibly disconnected simple graphs without loops or multiple edges, with features and decorations as follows.  Say the $1$-valent vertices are {\em external} and all other vertices are {\em internal}.  \begin{itemize}
\item All internal vertices are at least trivalent; 
\item Each external vertex has a label from $\{1,\dots,n,\epsilon,\omega \}$; 
\item Each label $1,\dots,n$ appears exactly once, and the label $\omega$ appears exactly twice; 
\item The graph obtained by joining all external vertices marked $\epsilon$ or $\omega$ is connected and has genus $g$.
\item There is no connected component consisting of a single edge connecting two external vertices with markings $\{\epsilon, \omega\}$ or $\{\omega, \omega\}$.
\end{itemize}
The edges that do not contain an external vertex with marking from $\{1, \dots, n\}$ are called \emph{structural}. Each generator is equipped with a total ordering of its structural edges, and reordering the structural edges induces multiplication by the sign of the permutation.  Note that some connected components may not have any internal vertices.  The \emph{degree} of a generator for $X_{g,n}$ is the number of structural edges plus one.
In figures, we replace each external vertex with its label and omit the edge ordering.  For example, the following is a generator of $X_{2,2}$ in degree $6$.
\begin{equation*}
\begin{tikzpicture}
\node[int] (v1) at (0:.5) {};
\node[int] (v2) at (120:.5) {};
\node[int] (v3) at (-120:.5) {};
\node (w1) at (0:1) {$1$};
\node (w2) at (120:1) {$\epsilon$};
\node (w3) at (-120:1) {$\omega$};
\node (u1) at (-2.5,0) {$\omega$};
\node (u2) at (-1.5,0) {$2$};
\draw (v1) edge (v2) edge (v3) edge (w1)
     (v2) edge (v3) edge (w2) (v3) edge (w3);
\draw (u1) edge (u2);
\end{tikzpicture}
\end{equation*}

The differential $\delta$ on $X_{g,n}$ is a sum of two parts 
$
\delta = \delta_{split} + \delta_{join}.
$ 
Here $\delta_{split}$ is a sum over all vertices and over all ways of splitting the vertex into two vertices joined by an edge, so that the new vertices are at least trivalent.
\begin{align}\label{equ:deltasplit}
  \delta_{split} \Gamma &= \sum_{v \text{ vertex} }  
  \Gamma\text{ split $v$} 
  &
  \begin{tikzpicture}[baseline=-.65ex]
  \node[int] (v) at (0,0) {};
  \draw (v) edge +(-.3,-.3)  edge +(-.3,0) edge +(-.3,.3) edge +(.3,-.3)  edge +(.3,0) edge +(.3,.3);
  \end{tikzpicture}
  &\mapsto
  \sum
  \begin{tikzpicture}[baseline=-.65ex]
  \node[int] (v) at (0,0) {};
  \node[int] (w) at (0.5,0) {};
  \draw (v) edge (w) (v) edge +(-.3,-.3)  edge +(-.3,0) edge +(-.3,.3)
   (w) edge +(.3,-.3)  edge +(.3,0) edge +(.3,.3);
  \end{tikzpicture}
  \end{align}
The part $\delta_{join}$ glues together a subset $S$ of the $\epsilon$- and $\omega$-decorated external vertices, where
\begin{itemize}
\item $|S| \geq 2$, and $S$ contains either 0 or 1 of the $\omega$-decorated vertices
\end{itemize}
and then attaches an edge to a new external vertex. The new edge is decorated by $\omega$ if $S$ contains an $\omega$-decorated vertex, and by $\epsilon$ otherwise:  
\begin{align*}
  \delta_{join} 
 \begin{tikzpicture}[baseline=-.8ex]
 \node[draw,circle] (v) at (0,.3) {$\Gamma$};
 \node (w1) at (-.7,-.5) {};
 \node (w2) at (-.25,-.5) {};
 \node (w3) at (.25,-.5) {};
 \node (w4) at (.7,-.5) {};
 \draw (v) edge (w1) edge (w2) edge (w3) edge (w4);
 \end{tikzpicture} 
 = 
 \sum_{S} 
 \begin{tikzpicture}[baseline=-.8ex]
 \node[draw,circle] (v) at (0,.3) {$\Gamma$};
 \node (w1) at (-.7,-.5) {};
 \node (w2) at (-.25,-.5) {};
 \node[int] (i) at (.4,-.5) {};
 \node (w4) at (.4,-1.3) {$\epsilon$ or $\omega$};
 \draw (v) edge (w1) edge (w2) edge[bend left] (i) edge (i) edge[bend right] (i) (w4) edge (i);
 \end{tikzpicture} \, .
 \end{align*}
Each graph produced by the differential has one new structural edge, which is taken first in the ordering; the relative ordering of the remaining edges is preserved.

\begin{thm}\label{thm:main WHGC}
For all $g$ and $n$ such that $2g + n \geq 3$ and $(g,n)\neq (1,1)$ there is an $\bbS_n$-equivariant isomorphism 
\[
 \gr_2 H^\bullet_c(\MM_{g,n}) \cong H^\bullet(X_{g,n},\delta).
\]
\end{thm}

\noindent This isomorphism at the level of cohomology is achieved through a zig-zag of quasi-isomorphisms of complexes, relating the row of Deligne's weight spectral sequence that computes $\gr_2 H^\bullet_c(\MM_{g,n})$ to the graph complex $X_{g,n}$.

The cohomology of such commutative graph complexes is far from fully understood.  Nevertheless, a partial understanding of this graph cohomology is enough to yield new results about the unstable cohomology of $\MM_{g,n}$ in weight 2, just as in weight 0. For $n = 0$, we express the weight 2 compactly supported cohomology of $\MM_g$ in terms of the weight 0 compactly supported cohomology of $\MM_{g',1}$ and $\MM_{g',2}$ for $g' \leq g$, as follows. Let $$\mathbb{V}_{\bullet, \bullet} := \textstyle{\bigwedge^2} \Big(\textstyle{\bigoplus_{g,k}} W_0 H^k_{c}(\MM_{g,1})\Big),$$ bigraded so that an element of $W_0 H^k_c(\MM_{g,1}) \wedge W_0 H^{k'}_c (\MM_{g',1})$ is in degree $(g+g', k + k')$. Let  $V_{as}$ denote the antisymmetric part of a $\bbS_2$-representation $V$. By convention, we set $H^\bullet_c(\MM_{0,2}) := 0$.

\begin{thm}\label{cor:n0}
For $g \geq 2$, there is an isomorphism of $\Q$-vector spaces
\[
 \gr_2 H_c^{k} (\MM_{g})
\ \cong \ 
 \mathbb{V}_{g,k-3} 
  \oplus W_0 H^{k-3}_c(\MM_{g-1,2})_{as}
   \oplus 
 \mathbb{V}_{g-1, k- 4}
 \oplus W_0 H^{k-4}_c(\MM_{g-2,2})_{as}.
\]
\end{thm}

\begin{rem}
Note that the Hodge structure on $\gr_2 H_c^{k} (\MM_{g})$ is pure Tate; this follows from Deligne's weight spectral sequence associated to $\MM_g \subset \bMM_g$, which abuts to $\gr_\bullet H_c^k(\MM_g)$, together with the fact that $H^2(\bMM_{g',n'})$ is pure Tate for all $g'$ and $n'$; see \S\ref{sec:wss}, below. The Hodge structure on $\gr_2 H_c^{k} (\MM_{g})$ is therefore obtained from the $\Q$-vector space description above by tensoring with $\Q(-1)$.  For any variety or Deligne-Mumford stack $X$ for which $\gr_1 H^*_c(X) = 0$, such as $\MM_g$, the subspace $W_2H_c^*(X)$ is an extension of $\gr_2H_c^*(X)$ by $W_0H_c^*(X)$, in the category of mixed Hodge structures.  For $X = \MM_g$, one might naturally expect that this extension is trivial. To prove this, it would be enough to show that $W_2H_c^*(\MM_g)$ is the mixed Hodge realization of a mixed Tate motive over $\Z$, by \cite[(1.6.11)]{DeligneGoncharov05}.
\end{rem}

\noindent 
To understand the geometric meaning of the first two summands in this theorem, let $D_{g}^\circ \subset \bMM_{g}$ be the locally closed substack parametrizing $1$-nodal curves.  We have a natural identification
\[
W_0H^k_c (D^\circ_{g}) \cong \mathbb{V}_{g,k} \oplus W_0 H^k_c (\MM_{g-1,2})_{as} .
\]
Cup product with the first Chern class of the normal bundle followed by the push-forward for the open inclusion $\iota$ of $D^\circ_{g}$ in the boundary $\partial \MM_{g} := \bMM_{g} \smallsetminus \MM_{g}$ and the coboundary map $\delta$ for the excision sequence give
\[
W_0H^\bullet_c (D^\circ_{g}) \xrightarrow{c_1 \smile} \gr_2 H^{\bullet + 2}_c(D^\circ_{g}) \xrightarrow{\iota_*} \gr_2 H^{\bullet + 2}_c \partial \MM_{g} \xrightarrow{\delta} \gr_2 H^{\bullet + 3}_c (\MM_{g}).
\]
We expect that the geometric content of Theorem~\ref{cor:n0} is that this composition is injective and the image accounts for the first two summands in Theorem~\ref{cor:n0}. As a first step toward seeing this, note that $c_1$ is represented by a smooth differential form that extends to the boundary. This pulls back to the normalization of $\partial \MM_g$ and hence acts on every page of the weight spectral sequence abutting to $\gr H^\bullet_c(\MM_g)$ (cf. the proof of Theorem~\ref{cor:Psi injection} in Section~\ref{sec:J1n}). This action can be interpreted graphically by following our zig-zag of quasi-isomorphisms. We do not have an analogous expected geometric interpretation for the other two summands.  For a graphical illustration of the simplest non-trivial cohomology classes produced by Theorem~\ref{cor:n0}, see  Section~\ref{sec:cor n0 example}.

For applications, we will use only subspaces of $\mathbb{V}_{\bullet, \bullet}$ that arise from the images of known classes under pullback to the universal curve $W_0 H^\bullet_c(\MM_g) \hookrightarrow W_0H^\bullet_c(\MM_{g,1})$. 
Recall that  $W_0H^{2g}_c(\MM_g)$ is nonvanishing for $g = 3$, $g = 5$, and $g \geq 7$, and its dimension grows at least exponentially with $g$ \cite[Theorem~1.1]{CGP1}; the proof relies on connections with Grothendieck-Teichm\"uller theory \cite{grt} and a deep result of Francis Brown \cite{Brown}. Here, we show furthermore that $W_0H^{2g+3}_c(\MM_g)$ is nonvanishing for $g =6$ and $g \geq 8$, and its dimension again grows at least exponentially with $g$; see Proposition~\ref{prop:degree3}.  It follows that the dimensions of $\mathbb{V}_{g,2g+k}$ and $\mathbb{V}_{g, 2g + k + 3}$ grow at least exponentially with $g$ whenever $W_0H^{2g' + k}_c(\MM_{g'}) \neq 0$ for some $g'$.  Combining Theorem~\ref{cor:n0} with what is currently known about $\bigoplus_g W_0H^\bullet_c(\MM_{g})$ and applying Poincar\'e duality, we have the following:

\begin{cor} \label{cor:odd}
The following unstable rational cohomology groups are nonzero: \vspace{-8pt}
\begin{multicols}{2}
\begin{enumerate}
\item $H^{4g-8}(\MM_g)$ for $g = 9$ and $g \geq 11;$
\item $H^{4g-9}(\MM_g)$ for $g = 6$ and $g \geq 8;$  
\item $H^{4g-11} (\MM_g)$ for $g = 10$ and  $g \geq 12$;
\item $H^{4g-12} (\MM_g)$ for $g = 9$ and $g \geq 11$;
\item $H^{4g-14}(\MM_g)$ for $g = 15$ and $g \geq 17$;
\item $H^{4g-15} (\MM_g)$ for $g = 14$ and $g \geq 16$;
\item $H^{4g-16} (\MM_g)$ for $g = 13$, $g = 15$, and $g \geq 17;$
\item $H^{4g-18} (\MM_g)$ for $g =17$ and $g \geq 19;$
\item $H^{4g-19} (\MM_g)$ for $g =16$ and $g \geq 18$.
\end{enumerate}
\end{multicols}
\vspace{-8 pt}
\noindent In each case, the dimension of $\gr_2 H^{4g-k}(\MM_g)$ grows at least exponentially with $g$.
\end{cor}

\noindent In particular, this gives the first infinite collections of nonvanishing odd cohomology groups on $\MM_g$.  Only a few sporadic examples were known previously, such as $H^{5}(\MM_4)$ \cite{Tommasi05}, and $H^{15}(\MM_6)$, $H^{23}(\MM_8),$ and $H^{27}(\MM_{10})$ \cite{CGP1}.  For the proof of Corollary~\ref{cor:odd}, see Section~\ref{sec:support}. 

\begin{rem}
For context, let us recall that $H^\bullet(\MM_g)$ is fully understood in the stable range, consisting of degrees up to $2\lfloor g/3 \rfloor$, where it agrees with the polynomial ring on the $\kappa$-classes \cite{MadsenWeiss07, Wahl13}. The Euler characteristic computations of Harer and Zagier show that the dimension of the unstable cohomology of $\MM_g$ grows super-exponentially with $g$ \cite{HarerZagier86}. This is far larger than the tautological subring generated by the $\kappa$-classes, which is nonzero precisely in the even degrees up to $2g-4$ \cite{Faber99, Looijenga95}. The problem of understanding the rest of the cohomology of $\MM_g$, starting with the degrees in which it is zero and nonzero, is highlighted in \cite[Section~9]{Margalit19}.

Until recently, there were only two instances of cohomology groups that were known to be nonzero due to the presence of non-tautological classes: $H^6(\MM_3)$ \cite{Looijenga93} and $H^5(\MM_4)$, each of which is 1-dimensional. It was previously conjectured by Kontsevich and independently by Church, Farb, and Putman that $H^{4g-5-k}(\MM_g)$ vanishes for $g \gg 0$, when $k$ is fixed \cite{CFP2, K3}.  This is true for $k < 0$ for virtual cohomological dimension reasons and for $k = 0$ by \cite{CFP1, MSS}. The main result of \cite{CGP1} disproved the conjecture for $k =1$. Corollary~\ref{cor:odd} shows that the conjecture is false in nine more cases, namely for $k \in \{3,4, 6,7,9, 10, 11, 13, 14\}$. See Table~\ref{tab:nonvanishing}, at the end of the introduction, for a summary of what is now known about the vanishing and non-vanishing of $H^{4g-5-k}(\MM_g)$.

Corollary~\ref{cor:odd} uses only the subspace of $\mathbb{V}_{\bullet, \bullet}$ that comes from $W_0 H^\bullet_c(\MM_g) \hookrightarrow W_0 H^\bullet_c(\MM_{g,1})$. Explicit computations show that this inclusion is an isomorphism for $g < 7$. For $g = 8$, the Euler characteristic formulas in \cite{CFGP} show that it is not surjective, but we do not know in which degrees the cokernel is supported.  Similarly, $W_0H^\bullet_c(\MM_{g,2})_{as}$ vanishes for $g < 7$, and for $g = 8$, its Euler characteristic is nonzero but we do not know in which degrees it is supported.  By Theorem~\ref{thm:main WHGC}, any future improvement in our understanding of $W_0H^\bullet_c(\MM_g)$, $W_0H^\bullet_c(\MM_{g,1})$ and $W_0H^\bullet_c(\MM_{g,2})_{as}$ will yield a corresponding improvement in our understanding of $\gr_2 H^\bullet_c(\MM_g)$. For an inspiring sense of what one might reasonably hope or expect to be true about $W_0H^\bullet_c(\MM_g)$, see \cite[Conjecture~1 and Table~2]{Brown21}.  
\end{rem}

For $n\geq 1$ we do not have an analogue of Theorem~\ref{cor:n0} expressing $\gr_2 H^\bullet_c(\MM_{g,n})$ in terms of $W_0 H^\bullet_c(\MM_{g',n'})$. Nevertheless, we produce injections that prove nonvanishing and lower bounds on dimensions in many cases. We begin with $n = 1$. Recall that the pullback $\pi^* \colon \gr_k H^\bullet (\MM_g) \to \gr_k H^\bullet_c (\MM_{g,1})$ is injective, as is the composition
\[
  \gr_k H^\bullet_c(\MM_{g})
  \xrightarrow{\pi^*}
  \gr_k H^\bullet_c(\MM_{g,1})
  \xrightarrow{\psi\wedge} \gr_{k+2} H^{\bullet+2}_c(\MM_{g,1}),
\]
because further composition with the Gysin push-forward is multiplication by $2g-2$. For weight $k = 0$, we strengthen this by showing that $\psi \wedge$ is itself injective: 
\begin{thm}\label{cor:Psi injection}
Let $g\geq 2$. Multiplication by the $\psi$-class at the marking yields an injection 
\begin{equation}\label{equ:Psi injection}
  \psi \wedge \colon W_0 H^\bullet_c(\MM_{g,1}) \to \gr_2 H^{\bullet+2}_c(\MM_{g,1}).
\end{equation}
Furthermore, the image of the injection 
\begin{equation}\label{equ:gr2 injection n1}
  \pi^*\colon \gr_2 H^\bullet_c(\MM_{g})
  \to \gr_2 H^{\bullet}_c(\MM_{g,1}),
\end{equation}
intersects trivially with that of \eqref{equ:Psi injection}.
\end{thm}

For $n\geq 2$, we show that a quasi-isomorphic subcomplex of $X_{g,n}$ can be split into several direct summands, as detailed in Section~\ref{sec:classes}. For some of these summands one may evaluate the cohomology explicitly in terms of known data or weight 0 cohomology. The summands whose cohomology can be computed correspond, in a sense to be made precise below, to graphs of the form 
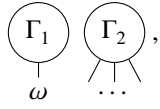
\[
\begin{tikzpicture}
  \node[circle, draw] (v) at (1,0) {$\Gamma_2$};
  \node[circle, draw] (w) at (0,0) {$\Gamma_1$};
  \node (we) at (0,-.8) {$\omega$};
  \node (v1) at (1,-.8) {$\dots$};
  \draw (w) edge (we) (v) edge (v1.north west) edge (v1) edge (v1.north east);
\end{tikzpicture} \ ,
\]
where $\Gamma_1$ is a connected component with one $\omega$-decoration and no other external vertices. The genus $g$ of the overall graph is the sum of the genera $g_1$ and $g_2$ of $\Gamma_1$ and $\Gamma_2$, respectively. The cohomology of the summands with $g_2=0$ and $g_2=1$ are explicitly understood, and the complementary part $\Gamma_1$ contributes a tensor factor that is identified with $W_0 H^\bullet_c(\MM_{g_1,1})$. In this way, one arrives at the following:

\begin{thm}\label{cor:n2}
For $g\geq 2$ and $n\geq 2$ one has an injection 
\[
  H_c^{n-3}(\MM_{0,n})
  \otimes W_0 H_c^{\bullet-n}(\MM_{g,1})
\   \oplus \ 
  \mathbb{B}_n \otimes W_0 H_c^{\bullet-n-3}(\MM_{g-1,1})
  \to 
  \gr_2 H_c^\bullet(\MM_{g,n}),
\]
where
$\mathbb{B}_n$ is a vector space defined in \eqref{equ:An Bn def} below, which has dimension at least $(n-2)!$ for $n\geq 3$.
\end{thm}


\begin{table} [h!] \label{tab:nonvanishing}
\scalebox{.65}{
    \centering
    \renewcommand{\arraystretch}{1.3}
    \definecolor{light-grey}{gray}{.67}
    \definecolor{dark-grey}{gray}{.33}
    \begin{tabular}{| c| C{8 pt} C{8 pt} C{8 pt} C{8 pt} C{8 pt} C{8 pt} C{8 pt} C{8 pt} C{8 pt} C{8 pt} C{8 pt} C{8 pt} C{8 pt} C{8 pt} C{8 pt} C{8 pt} C{8 pt} C{8 pt} C{8 pt} C{8 pt} C{8 pt} C{8 pt} C{8 pt} C{8 pt} C{8 pt} C{8 pt} C{8 pt} C{8 pt}| }
    \hline
$g,k$ & 0 & 1 & 2 & 3 & 4 & 5 & 6 & 7 & 8 & 9 & 10 & 11 & 12 & 13 & 14 & 15 & 16 & 17 & 18 &  19 & 20 & 21 & 22 & 23 & 24 & 25 & 26 & 27\\ [0.5ex]
     \hline
    2 & &  & &  \multicolumn{1}{l|}{\cellcolor{black}} & & & & & & & & & & & & & & & & & & & & & & & &  \\ \cline{6-9}
    3 & & \cellcolor{black} & & & & \cellcolor{black} & & \cellcolor{black} & & & & & & & & & & & & & & & & & & & & \\ \cline{9-13}
    4 & & & & & & & \cellcolor{black} & \cellcolor{black} & & \cellcolor{black} & & \cellcolor{black} & & & & & & & & & & & & & & & &  \\ \cline{13-17}
    5 & & \cellcolor{black}  & \cellcolor{light-grey} & \cellcolor{light-grey} & \cellcolor{light-grey}& \cellcolor{light-grey}& \cellcolor{light-grey} & \cellcolor{light-grey} & \cellcolor{light-grey} & \cellcolor{black} & \cellcolor{light-grey} & \cellcolor{black} & \cellcolor{light-grey} & \cellcolor{black} & & \cellcolor{black} & & & & & & & & & & & & \\ \cline{17-21}
    6 &  & \cellcolor{light-grey} & \cellcolor{light-grey} & \cellcolor{light-grey} & \cellcolor{black} & \cellcolor{light-grey} & \cellcolor{light-grey} & \cellcolor{light-grey} & \cellcolor{light-grey} & \cellcolor{light-grey} & \cellcolor{light-grey} & \cellcolor{black} &\cellcolor{light-grey}  & \cellcolor{black} & \cellcolor{light-grey} & \cellcolor{black} & & \cellcolor{black} & & \cellcolor{black} & & & & & & & & \\ \cline{21-25}
    7 & & \cellcolor{black} & \cellcolor{light-grey}& \cellcolor{light-grey}& \cellcolor{light-grey}& \cellcolor{light-grey}& \cellcolor{light-grey}& \cellcolor{light-grey} & \cellcolor{light-grey} & \cellcolor{light-grey} & \cellcolor{light-grey} & \cellcolor{light-grey} & \cellcolor{light-grey}& \cellcolor{black}& \cellcolor{light-grey} &\cellcolor{black} & \cellcolor{light-grey} &\cellcolor{black} & \cellcolor{light-grey}  & \cellcolor{black}& & \cellcolor{black} & & \cellcolor{black} & & & & \\ \cline{26-29}
    8 & & \cellcolor{black} & \cellcolor{light-grey} & \cellcolor{light-grey} & \cellcolor{black} & \cellcolor{light-grey} & \cellcolor{light-grey} & \cellcolor{light-grey} & \cellcolor{light-grey} & \cellcolor{light-grey} & \cellcolor{light-grey} & \cellcolor{light-grey} &  \cellcolor{light-grey}& \cellcolor{light-grey} & \cellcolor{light-grey} &  \cellcolor{black} &\cellcolor{light-grey}  & \cellcolor{black} & \cellcolor{light-grey}& \cellcolor{black} & \cellcolor{light-grey} & \cellcolor{black} & \cellcolor{light-grey} & \cellcolor{black} & & \cellcolor{black} & & \cellcolor{black}\\
    9 & & \cellcolor{black} & \cellcolor{light-grey} & \cellcolor{dark-grey} & \cellcolor{black} &  \cellcolor{light-grey}& \cellcolor{light-grey} & \cellcolor{dark-grey}&  \cellcolor{light-grey} & \cellcolor{light-grey}& \cellcolor{light-grey}& \cellcolor{light-grey}& \cellcolor{light-grey}& \cellcolor{light-grey}& \cellcolor{light-grey}& \cellcolor{light-grey}& \cellcolor{light-grey}& \cellcolor{black}& \cellcolor{light-grey}& \cellcolor{black}& \cellcolor{light-grey}& \cellcolor{black}& \cellcolor{light-grey}& \cellcolor{black} & \cellcolor{light-grey}& \cellcolor{black} & & \cellcolor{black} \\ 
    10 & & \cellcolor{black} & \cellcolor{light-grey} & \cellcolor{light-grey} & \cellcolor{black} &  \cellcolor{light-grey}& \cellcolor{dark-grey} & \cellcolor{light-grey} & \cellcolor{black}& \cellcolor{light-grey}& \cellcolor{light-grey}& \cellcolor{light-grey}& \cellcolor{light-grey}& \cellcolor{light-grey}& \cellcolor{light-grey}& \cellcolor{light-grey}& \cellcolor{light-grey}& \cellcolor{light-grey}& \cellcolor{light-grey}& \cellcolor{black}& \cellcolor{light-grey}& \cellcolor{black}& \cellcolor{light-grey}& \cellcolor{black} & \cellcolor{light-grey}& \cellcolor{black} & \cellcolor{light-grey} & \cellcolor{black} \\
    11 & & \cellcolor{black} & \cellcolor{light-grey} & \cellcolor{dark-grey} &  \cellcolor{dark-grey}& \cellcolor{light-grey} &  \cellcolor{light-grey} & \cellcolor{dark-grey} & \cellcolor{light-grey} & \cellcolor{light-grey}& \cellcolor{light-grey}& \cellcolor{light-grey}& \cellcolor{light-grey}& \cellcolor{light-grey}& \cellcolor{light-grey}& \cellcolor{light-grey}& \cellcolor{light-grey}& \cellcolor{light-grey}& \cellcolor{light-grey}& \cellcolor{light-grey}& \cellcolor{light-grey}& \cellcolor{black}& \cellcolor{light-grey}& \cellcolor{black} & \cellcolor{light-grey}& \cellcolor{black} & \cellcolor{light-grey} & \cellcolor{black} \\
    12 & & \cellcolor{black} & \cellcolor{light-grey} & \cellcolor{dark-grey} &  \cellcolor{dark-grey}& \cellcolor{light-grey}& \cellcolor{dark-grey} & \cellcolor{dark-grey} & \cellcolor{light-grey}& \cellcolor{light-grey}& \cellcolor{light-grey}& \cellcolor{light-grey}& \cellcolor{light-grey}& \cellcolor{light-grey}& \cellcolor{light-grey}& \cellcolor{light-grey}& \cellcolor{light-grey}& \cellcolor{light-grey}& \cellcolor{light-grey}& \cellcolor{light-grey}& \cellcolor{light-grey}& \cellcolor{light-grey}& \cellcolor{light-grey}& \cellcolor{black} & \cellcolor{light-grey}& \cellcolor{black} & \cellcolor{light-grey} & \cellcolor{black} \\
    13 & & \cellcolor{black} & \cellcolor{light-grey} & \cellcolor{dark-grey} &  \cellcolor{dark-grey}&  \cellcolor{light-grey} & \cellcolor{dark-grey} & \cellcolor{dark-grey}& \cellcolor{light-grey}& \cellcolor{light-grey}& \cellcolor{light-grey}& \cellcolor{dark-grey}& \cellcolor{light-grey}&\cellcolor{light-grey}& \cellcolor{light-grey}& \cellcolor{light-grey}& \cellcolor{light-grey}& \cellcolor{light-grey}& \cellcolor{light-grey}& \cellcolor{light-grey}& \cellcolor{light-grey}& \cellcolor{light-grey}& \cellcolor{light-grey}& \cellcolor{light-grey} & \cellcolor{light-grey}& \cellcolor{black} & \cellcolor{light-grey} & \cellcolor{black} \\
    14 & & \cellcolor{black} & \cellcolor{light-grey} & \cellcolor{dark-grey} &  \cellcolor{dark-grey}& \cellcolor{light-grey}&  \cellcolor{dark-grey}& \cellcolor{dark-grey}& \cellcolor{light-grey} & \cellcolor{light-grey} &\cellcolor{dark-grey}& \cellcolor{light-grey}& \cellcolor{light-grey}& \cellcolor{light-grey}& \cellcolor{light-grey}& \cellcolor{light-grey}& \cellcolor{light-grey}& \cellcolor{light-grey}& \cellcolor{light-grey}& \cellcolor{light-grey}& \cellcolor{light-grey}& \cellcolor{light-grey}& \cellcolor{light-grey}& \cellcolor{light-grey} & \cellcolor{light-grey}& \cellcolor{light-grey} & \cellcolor{light-grey} & \cellcolor{black} \\
    15 & & \cellcolor{black} & \cellcolor{light-grey} & \cellcolor{dark-grey} &  \cellcolor{dark-grey}& \cellcolor{light-grey}&  \cellcolor{dark-grey}& \cellcolor{dark-grey}& \cellcolor{light-grey} & \cellcolor{dark-grey}& \cellcolor{light-grey}& \cellcolor{dark-grey}& \cellcolor{light-grey}& \cellcolor{light-grey}& \cellcolor{light-grey}& \cellcolor{light-grey}& \cellcolor{light-grey}& \cellcolor{light-grey}& \cellcolor{light-grey}& \cellcolor{light-grey}& \cellcolor{light-grey}& \cellcolor{light-grey}& \cellcolor{light-grey}& \cellcolor{light-grey} & \cellcolor{light-grey}& \cellcolor{light-grey} & \cellcolor{light-grey} & \cellcolor{light-grey} \\
    16 & & \cellcolor{black} & \cellcolor{light-grey} & \cellcolor{dark-grey} &  \cellcolor{dark-grey}& \cellcolor{light-grey}&  \cellcolor{dark-grey}& \cellcolor{dark-grey}& \cellcolor{light-grey} & \cellcolor{light-grey}& \cellcolor{dark-grey}& \cellcolor{light-grey}& \cellcolor{light-grey}& \cellcolor{light-grey}& \cellcolor{dark-grey}& \cellcolor{light-grey}& \cellcolor{light-grey}& \cellcolor{light-grey}& \cellcolor{light-grey}& \cellcolor{light-grey}& \cellcolor{light-grey}& \cellcolor{light-grey}& \cellcolor{light-grey}& \cellcolor{light-grey} & \cellcolor{light-grey}& \cellcolor{light-grey} & \cellcolor{light-grey} & \cellcolor{light-grey} \\
    17 & & \cellcolor{black} & \cellcolor{light-grey} & \cellcolor{dark-grey} &  \cellcolor{dark-grey}& \cellcolor{light-grey}&  \cellcolor{dark-grey}& \cellcolor{dark-grey}& \cellcolor{light-grey} & \cellcolor{dark-grey}& \cellcolor{dark-grey}& \cellcolor{dark-grey}& \cellcolor{light-grey}& \cellcolor{dark-grey}& \cellcolor{light-grey}& \cellcolor{light-grey}& \cellcolor{light-grey}& \cellcolor{light-grey}& \cellcolor{light-grey}& \cellcolor{light-grey}& \cellcolor{light-grey}& \cellcolor{light-grey}& \cellcolor{light-grey}& \cellcolor{light-grey} & \cellcolor{light-grey}& \cellcolor{light-grey} & \cellcolor{light-grey} & \cellcolor{light-grey} \\
    18 & & \cellcolor{black} & \cellcolor{light-grey} & \cellcolor{dark-grey} &  \cellcolor{dark-grey}& \cellcolor{light-grey}&  \cellcolor{dark-grey}& \cellcolor{dark-grey}& \cellcolor{light-grey} & \cellcolor{dark-grey}& \cellcolor{dark-grey}& \cellcolor{dark-grey}& \cellcolor{light-grey}& \cellcolor{light-grey}& \cellcolor{dark-grey}& \cellcolor{light-grey}& \cellcolor{light-grey}& \cellcolor{light-grey}& \cellcolor{light-grey}& \cellcolor{light-grey}& \cellcolor{light-grey}& \cellcolor{light-grey}& \cellcolor{light-grey}& \cellcolor{light-grey} & \cellcolor{light-grey}& \cellcolor{light-grey} & \cellcolor{light-grey} & \cellcolor{light-grey} \\    
    19& & \cellcolor{black} & \cellcolor{light-grey} & \cellcolor{dark-grey} &  \cellcolor{dark-grey}& \cellcolor{light-grey}&  \cellcolor{dark-grey}& \cellcolor{dark-grey}& \cellcolor{light-grey} & \cellcolor{dark-grey}& \cellcolor{dark-grey}& \cellcolor{dark-grey}& \cellcolor{light-grey}& \cellcolor{dark-grey}& \cellcolor{dark-grey}& \cellcolor{light-grey}& \cellcolor{light-grey}& \cellcolor{light-grey}& \cellcolor{light-grey}& \cellcolor{light-grey}& \cellcolor{light-grey}& \cellcolor{light-grey}& \cellcolor{light-grey}& \cellcolor{light-grey} & \cellcolor{light-grey}& \cellcolor{light-grey} & \cellcolor{light-grey} & \cellcolor{light-grey} \\
    20 & & \cellcolor{black} & \cellcolor{light-grey} & \cellcolor{dark-grey} &  \cellcolor{dark-grey}& \cellcolor{light-grey}&  \cellcolor{dark-grey}& \cellcolor{dark-grey}& \cellcolor{light-grey} & \cellcolor{dark-grey}& \cellcolor{dark-grey}& \cellcolor{dark-grey}& \cellcolor{light-grey}& \cellcolor{dark-grey}& \cellcolor{dark-grey}& \cellcolor{light-grey}& \cellcolor{light-grey}& \cellcolor{light-grey}& \cellcolor{light-grey}& \cellcolor{light-grey}& \cellcolor{light-grey}& \cellcolor{light-grey}& \cellcolor{light-grey}& \cellcolor{light-grey} & \cellcolor{light-grey}& \cellcolor{light-grey} & \cellcolor{light-grey} & \cellcolor{light-grey} \\        
    21 & & \cellcolor{black} & \cellcolor{light-grey} & \cellcolor{dark-grey} &  \cellcolor{dark-grey}& \cellcolor{light-grey}&  \cellcolor{dark-grey}& \cellcolor{dark-grey}& \cellcolor{light-grey} & \cellcolor{dark-grey}& \cellcolor{dark-grey}& \cellcolor{dark-grey}& \cellcolor{light-grey}& \cellcolor{dark-grey}& \cellcolor{dark-grey}& \cellcolor{light-grey}& \cellcolor{light-grey}& \cellcolor{light-grey}& \cellcolor{light-grey}& \cellcolor{light-grey}& \cellcolor{light-grey}& \cellcolor{light-grey}& \cellcolor{light-grey}& \cellcolor{light-grey} & \cellcolor{light-grey}& \cellcolor{light-grey} & \cellcolor{light-grey} & \cellcolor{light-grey} \\ 
    \hline
    \end{tabular}
    }
    \caption{A summary of vanishing and nonvanishing results for the rational cohomology groups $H^{4g-5-k}(\MM_g)$.  Black and dark gray boxes denote previously known and new nonvanishing groups, respectively. The white boxes denote groups that are known to vanish, and the light gray boxes denote those that are as yet unknown.}
    \end{table}

\section{Preliminaries}

\subsection{Notation and conventions}
\label{sec:notation}
We work over the rational numbers $\Q$. All vector spaces are understood to be $\Q$-vector spaces, and likewise all homology and cohomology groups are taken with $\Q$-coefficients. If $V$ is a finite-dimensional vector space with the action of a finite group $G$, we identify the invariant subspace $V^G$ with the coinvariant space $V_G$, by averaging over the group action.

For a graded vector space $V$ and an integer $k$, let $V[k]$ be the graded vector space obtained by shifting all degrees by $k$, i.e., the degree $j$ part of $V$ is the degree $j-k$ part of $V[k]$.

When working with graded vector spaces we use the standard Koszul sign convention.
That is, for $V$ and $W$ graded vector spaces the isomorphism $V\otimes W\cong W\otimes V$ sends a tensor product $v\otimes w$ of homogeneous elements $v\in V$, $w\in W$ of degrees $|v|, |w|$ to the element $(-1)^{|v||w|} w\otimes v$. Furthermore, let $f \in \Hom(V,V')$ and $g \in \Hom(W,W')$ be two linear maps between graded vector spaces, of homogeneous degree $|f|$ and $|g|$, respectively.  Then we define the linear map $f \otimes g \in \Hom(V \otimes V', W \otimes W')$, such that, for homogeneous $v \in V$ and $w \in W$,
\begin{equation*} 
(f \otimes g)(v \otimes w) = (-1)^{|g||v|} f(v) \otimes g(w).
\end{equation*}
The phrase ``differential graded" is abbreviated dg. We follow cohomological grading conventions, so all differentials have degree $+1$ unless stated otherwise. The Koszul sign rule will be particularly important when we consider the tensor product of dg vectors space $(V, d_V)$ and $(W, d_W)$. The differential on $V \otimes W$ is then 
\[
d_{V \otimes W} := d_V \otimes \id_W + \id_V \otimes d_W.
\]
When no confusion seems possible, we denote this by $d_V + d_W$, with the Koszul sign convention understood.

\medskip

We will consider many instances of spaces that depend on a genus $g$ and a number of marked points $n$, among them the moduli spaces $\bMM_{g,n}$ and the dg vector spaces $X_{g,n}$.  When $n = 0$, we follow the usual convention of omitting this from the notation, e.g., we write $\bMM_{g} := \bMM_{g,0}$ and $X_{g} := X_{g,0}$.

\subsection{Quasi-isomorphisms and acyclicity}

A morphism of dg vector spaces is a quasi-isomorphism if it induces an isomorphism on cohomology, and a dg vector space $(V,d)$ is acyclic if its cohomology vanishes. We will make repeated use of the following elementary sufficient criteria for morphisms of dg vector spaces to be quasi-isomorphisms, and for dg vector spaces to be acyclic. 

\begin{lemma}\label{lem:acyclicity}
Let $(V,d)$ be a dg vector space with $V_1\subset V$ a dg subspace. Suppose there is a decomposition of graded vector spaces $V=V_0\oplus V_1$. Let $\pi_1$ be the projection onto $V_1$, and let $D=\pi_1\circ d: V_0 \to V_1$.  Then $\ker D \subset V$ and 
  $V_0\oplus \im D\subset V$ are dg subvector spaces.  Moreover, 
\begin{enumerate}
  \item \label{it:surj} If $D$ is surjective, then $\ker D\to V$ is a quasi-isomorphism, and
  \item If $D$ is injective then  
  $V\to V/(V_0\oplus \im D)=\coker D$ is a quasi-isomorphism.
\end{enumerate}
In particular, if $D$ is an isomorphism, then $(V,d)$ is acyclic.
\end{lemma}
\begin{proof} The differential $d$ on $V$ is a sum of three pieces $d_0$, $D$, and $d_1$: 
\[
\begin{tikzcd}
V_0 \ar[bend left]{r}[above]{D} \ar[loop above]{}[above]{d_0} & V_1 \ar[loop above]{}[above]{d_1}
\end{tikzcd}\,.
\]
The fact that $d^2=0$ translates into the three equations 
\begin{align*}
d_0^2&=0 
&
d_1^2&=0
&
d_1D+Dd_0 &=0.
\end{align*}
It follows that $\ker D$ and $V_0 \oplus \im D$ are dg subspaces of $V$.

Next, we show that if $D$ is an isomorphism then $(V,d)$ is acyclic. Let $x \in V$ be a cocycle, and write $x = x_0 + x_1$, with $x_i \in V_i$. Since $x$ is a cocycle, we have $d_0x_0=0$ and $Dx_0+d_1 x_1=0$. Thus
\[
Dx_0 = -d_1 x_1 = -d_1 D D^{-1} x_1 = Dd_0 D^{-1} x_1.
\]
Applying $D^{-1}$ to both sides shows that $x_0 = d_0 D^{-1} x_1$, and hence $x = d D^{-1} x_1$ is a coboundary, as required.

Statements (1) and (2) follow by applying this criterion for acyclicity to $V/\ker D$
and $V_0\oplus \im D$, respectively.
\end{proof}

\subsection{Deligne's weight spectral sequence} \label{sec:wss} The natural starting point for understanding the weight-graded compactly supported cohomology of a smooth variety or Deligne-Mumford stack with a given normal crossings compactification is the Poincar\'e dual of Deligne's weight spectral sequence \cite[\S3.2]{Deligne71}, cf. \cite[Example~3.5]{Petersen17}.  We briefly recall this construction in the special case of the Deligne-Mumford compactification of $\MM_{g,n}$ by the space of stable $n$-marked curves $\overline \MM_{g,n}$. It is naturally stratified by the topological type, which is encoded combinatorially in the dual graph of the curve. Each stratum thus corresponds to a stable $n$-marked graph $\Gamma$ of genus $g$, i.e., a graph with $n$ legs labeled $1, \ldots, n$ in which each vertex $v$ is labeled by a non-negative integer $g_v$ such that $2g_v + n_v \geq 3$,
where $n_v$ denotes the valence of the vertex $v$, and such that $g = h^1(\Gamma) + \sum_v g_v$.

For each such stable $n$-marked graph $\Gamma$ of genus $g$, the locus of curves with dual graph $\Gamma$ is the image of the natural gluing map
\[
\xi_\Gamma \colon \MM_\Gamma \to \overline \MM_{g,n},
\]
where $\MM_\Gamma := \prod_{v \in V(\Gamma)} \MM_{g_v,n_v}.$  This gluing extends naturally to $\overline \MM_\Gamma := \prod_v \overline \MM_{g_v,n_v}$.  Let $D_\Gamma := \xi_\Gamma (\prod_v \overline \MM_{g_v,n_v})$ be the image of this extension; it is the closure of the locus of curves with dual graph $\Gamma$. The normalization of $D_\Gamma$ is the smooth and proper Deligne-Mumford stack
\[
\widetilde D_\Gamma = \overline \MM_\Gamma / \Aut(\Gamma).
\]
\noindent Note that the codimension of $D_\Gamma$ in $\overline \MM_{g,n}$ is the number of edges in $\Gamma$.  Let $$\widetilde D^j := \bigsqcup_{|E(\Gamma)| = j} \widetilde D_\Gamma.$$  The image of a small neighborhood of a point in $\widetilde D^j$ is contained in precisely $j$ analytic branches of the boundary divisor $\partial \MM_{g,n} := \bMM_{g,n} \smallsetminus \MM_{g,n}$. The monodromy action on these branches defines a local system of rank $j$ on $\widetilde D^j$ whose determinant is denoted $\epsilon^j$.
The $E_1$-page of the weight spectral sequence for $H^\bullet_c (\MM_{g,n})$ is expressed naturally in terms of cohomology with coefficients in this local system:
\[
E_1^{j,k} \cong H^{k}(\widetilde D^j, \epsilon^j).
\]
Note that the branches of $D$ that contain the image of a small neighborhood of a point in $\widetilde D_\Gamma$ are naturally identified with the edges of $\Gamma$, and $\epsilon^j$ is trivialized by pullback to $\overline \MM_\Gamma$.  In this way, the local system $\epsilon^j$ is identified with the determinant of the permutation action of $\Aut(\Gamma)$ on $E(\Gamma)$.  Thus, we have
\[
E_1^{j,k} \cong \bigoplus_{|E(\Gamma)| = j} \big( H^{k}(\overline \MM_\Gamma) \otimes \det E(\Gamma) \big)^{\Aut(\Gamma)}.
\]
The spectral sequence degenerates at $E_2$. Hence $\gr_k H^{j+k}_c(\MM_{g,n})$ is canonically identified with $H^j$ of the complex
\begin{equation} \label{eq:weightk}
\resizebox{.91\hsize}{!}{
$\displaystyle{\cdots \to \bigoplus_{|\Gamma| = j-1} \big( H^{k}(\overline \MM_\Gamma) \otimes \det E(\Gamma) \big)^{\Aut(\Gamma)} \xrightarrow{d_{j-1}} \bigoplus_{|\Gamma| = j}\big( H^{k}(\overline \MM_\Gamma) \otimes \det E(\Gamma) \big)^{\Aut(\Gamma)} \xrightarrow{d_j} \bigoplus_{|E(\Gamma)| = j + 1} \big( H^{k}(\overline \MM_\Gamma) \otimes \det E(\Gamma) \big)^{\Aut(\Gamma)} \to \cdots}$
}
\end{equation}
All of the data in this spectral sequence is neatly encoded in the language of modular operads \cite{GK}, cf. \cite[Example~3.9]{Petersen17}. More precisely, the cohomology groups $H^\bullet(\overline \MM_{g,n})$ naturally form a modular cooperad, whose Feynman transform evaluated at $(g,n)$ is the direct sum, over all weights $k$, of the complex \eqref{eq:weightk}.  Using the K\"unneth decomposition, one may encode generators as linear combinations of graphs $\Gamma$ whose vertices $v$ are labeled with elements of $H^{k_v}(\overline \MM_{g_v,n_v})$ such that $\sum_v k_v = k$.  The result is now known as the Getzler-Kapranov graph complex and denoted $\GK_{g,n}^k$; see, e.g.,  \cite[Section~6.1]{AWZ} and \cite{Kalugin}.

\subsection{Modular cooperads and Feynman transform}
We briefly recall the notion of modular cooperads and their Feynman transforms, and refer the reader to \cite{GK} for details.  A stable $\bbS$-module is a collection of dg vector spaces $\POp=\{\POp(g,n)\}$, for $g,n\geq 0$ and  $2g+n\geq 3$, with each $\POp(g,n)$ equipped with an action of $\bbS_n$.  

Let $\POp$ be a stable $\bbS$-module. 
For a stable graph $\Gamma$ we may define the tensor product 
\[
\otimes_\Gamma \POp := \otimes_{v\in V(\Gamma)} \POp(g_v,n_v).
\]
Note that $\Aut(\Gamma)$ acts naturally on $\otimes_\Gamma \POp$.  A modular cooperad is a stable $\bbS$-module together with a morphism
\[
 \POp(g,n) \to \otimes_\Gamma \POp 
\]
for every stable graph of genus $g$ with $n$ legs, satisfying suitable compatibility relations.

To a modular cooperad $\POp$, one naturally associates the Feynman transform $\mF\POp$, which is a modular cooperad whose underlying stable $\bbS$-module is
\[
  \mF\POp(g,n) \cong \bigoplus_{[\Gamma]} \left(\bigotimes_\Gamma \POp \otimes \Q[-1]^{\otimes |E(\Gamma)|} \right)_{\Aut(\Gamma)}.
\]
Here, the sum is over isomorphism classes of stable graphs of genus $g$ with $n$ legs, $|E(\Gamma)|$ is the set of edges of $\Gamma$, and the action of the automorphism group is diagonal on $\bigotimes_\Gamma$ and $\Q[-1]^{\otimes |E(\Gamma)|}$ by permutation of factors. Elements of $\mF\POp(g,n)$ can be understood as isomorphism classes of graphs whose vertices are decorated by suitable elements of $\POp$. 
In cases where $\POp$ has no differential, the differential on $\mF\POp(g,n)$ has two pieces, $$\delta=\delta_{split}+\delta_{loop}.$$ 
The piece $\delta_{split}$ splits vertices similar to \eqref{equ:deltasplit}, using the cooperadic composition of $\POp$ on decorations.
Similarly, $\delta_{loop}$ creates a loop at a vertex
\begin{align*}
  \delta_{loop}:
  \begin{tikzpicture}
  \node[ext] (v) at (0,0) {};
  \draw (v) edge +(-.5,-.5) edge +(0,-.5) edge +(.5,-.5);
  \end{tikzpicture}
  &\mapsto 
  \begin{tikzpicture}
    \node[ext] (v) at (0,0) {};
    \draw (v) edge +(-.5,-.5) edge +(0,-.5) edge +(.5,-.5) (v) edge[loop] (v);
    \end{tikzpicture}
  \end{align*}
using the modular cooperadic map
$
\POp(g, n) \to \POp(g -1,n+2)
$
on the decoration. 

\subsection{The Getzler-Kapranov graph complex}

The cohomology groups $H^\bullet (\bMM_{g,n})$ of the Deligne-Mumford compactified moduli spaces $\bMM_{g,n}$ assemble into a modular cooperad that we shall denote $H(\bMM)$. We consider the Feynman-transform $\mF H(\bMM)$ of this modular cooperad.

An element of $\mF H(\bMM)(g,n)$ is a linear combination of stable $n$-marked graphs of genus $g$ whose vertices $v$ are decorated by elements of $H(\bMM_{g_v,n_v})$. As recalled in the preceding section, the differential has two parts
\begin{equation*}
\delta = \delta_{split} + \delta_{loop}. 
\end{equation*}
The dg vector spaces $\mF H(\bMM)(g,n)$ inherit the weight grading from $H(\bMM)$, where the weight of a graph is the sum of the degrees of its vertex decorations. 

\begin{defi}
The weight $k$ Getzler-Kapranov graph complex is
\[
\GK^k_{g,n} := \gr_k \mF H(\bMM)(g,n).
\]
\end{defi}

\noindent Note that $\GK^k_{g,n}$ is also naturally identified with \eqref{eq:weightk}, the weight $k$ row in on the $E_1$-page of the weight spectral sequence for $\MM_{g,n} \subset \bMM_{g,n}$, and hence there is a canonical isomorphism
\[
 H^\bullet \GK^k_{g,n} \cong \gr_k H^\bullet_c(\MM_{g,n}).
\]

\newcommand{\GG}{\fGC_{2,\mathrm{conn}}}
\subsection{A nonvanishing result in weight 0} 

Here we state and prove Proposition~\ref{prop:degree3}, a nonvanishing and exponential growth result for $W_0 H^{2g+3}_c (\MM_g)$. This proposition, combined with Theorem~\ref{cor:n0} and the results on $W_0H^\bullet_c(\MM_g)$ in \cite{CGP1}, gives the nonvanishing statements in Corollary~\ref{cor:odd}.

Let $\GG$ be the full commutative graph complex studied in \cite{grt}. The elements of $\GG$ are (possibly infinite) formal $\Q$-linear combinations of isomorphism classes of connected graphs with vertices of any valence. It is a dg Lie algebra with a combinatorially defined bracket, and the differential $\delta$ is the bracket with a single edge. The subspace $\GC_2$ spanned by stable graphs is a direct summand, and the restriction of $\delta$ is exactly $\delta_{split}$.  The cohomology of $\GG$ is
\begin{equation} \label{eq:loopcohom}
H(\GG) = H(\GC_2) \oplus \bigoplus_{k \geq 1} \Q [L_{4k+1}], 
\end{equation}
where $L_n$ denotes a loop of $n$ edges, in which every vertex has valence 2 \cite[Proposition~3.4]{grt}.

Recall that the cohomological degree of a graph in $\GG$ of genus $g$ with $e$ edges is $e -2g$. There is an extra differential $\nabla$ on $\GG$, given by bracket with a loop edge, studied in \cite{KWZ}.   

\begin{lemma} \label{lem:nablaclosed}
Every class in $H^0(\GC_2)$ is represented by a $\delta$-cocycle $F \in \GG$ such that $\nabla F = 0$. 
\end{lemma}

\begin{proof}
Let $F \in \GG$ be a $\delta$-cocycle representing a nontrivial class in $H(\GC_2)$. We must show that there is some $F' \in \GG$ such that $F' - F$ is $\delta$-exact and $\nabla F' = 0$.  First, note that
$
\delta \nabla F = - \nabla \delta F = 0.
$
Since $H(\GG)$ vanishes in negative degrees \cite[Theorem~1.1]{grt}, we see that $\nabla F$ is $\delta$-exact, i.e., there is $F_1 \in \GG$ such that $\delta F_1 = \nabla F$. Moreover, $\delta \nabla F_1 = - \nabla \delta F_1 = -\nabla^2 F = 0$.  Continuing in this way, the number of edges remains fixed and the genus increases by 1 at each step, so eventually we arrive at a sequence $(F = F_0, \ldots, F_n)$
\begin{equation}\label{eq:waterfall}
\delta F_n = \nabla F_{n-1} \quad \quad \mbox{ and } \nabla(F_n) = 0.
\end{equation}
Then, since $\GG$ is $\nabla$-acyclic in genus greater than 1 \cite[Corollary 3]{KWZ}, there is some $G_n$ such that $\nabla G_n = F_n$.  If $n > 1$, then we can replace $F_n$ with $F_{n-1} + \delta G_n$ to get a shorter sequence that satisfies \eqref{eq:waterfall}. Repeating the process, we arrive at such a sequence with $n = 1$.  Choose $G$ such that $\nabla G = F_1$, and set $F' = F + \delta G$.  
\end{proof}

Recall from \cite{CGP1} that $W_0 H^\bullet_c (\MM_g)$ is identified with the genus $g$ part of the cohomology of the complex of commutative stable graphs $\GC_2$ and the dimension of $W_0H^{2g}_c(\MM_g)$ grows at least exponentially with $g$.

\begin{prop} \label{prop:degree3}
The weight zero cohomology group $W_0 H^{2g+3}_c (\MM_g)$ is nonzero for $g = 6$ and $g \geq 8$.  Moreover, its dimension grows at least exponentially with $g$.
\end{prop}

\begin{proof}
Let us denote the genus $g$ part of $\GC_2$ by $\GC_2^{(g)}$. The main result of \cite{CGP1} shows that 
$W_0 H^{2g+k}_c(\MM_g) \cong H^k (\GC_2^{(g)})$.  We recall that $H^0(\GC_2)$ is naturally identified with the Grothendieck-Teichm\"uller Lie algebra \cite{grt}. Using this identification, it follows from the main results of \cite{Brown} that $H^0(\GC_2)$ contains a free Lie subalgebra generated by classes $\sigma_g \in H^0(\GC_2^{(g)})$ for odd $g \geq 3$; see \cite[(1.3)]{Brown21}.   

We consider two linear maps induced by the Lie bracket:
\[
[\sigma_3, \cdot ] \colon H^0 (\GC_2) \to H^0(\GC_2) \quad \quad \mbox{ and } \quad \quad [L_5, \cdot ] \colon H^0 (\GC_2) \to H^3(\GC_2).
\]
The natural target of $[L_5, \cdot ]$ is $H^3 (\GG)$, but the image of any nontrivial class has genus greater than 1; by \eqref{eq:loopcohom}, it must be in the summand $H^3 (\GC_2)$.
We claim that the kernel of $[L_5, \cdot]$ is contained in the kernel of $[\sigma_3, \cdot]$.  The proposition follows from the claim, since the free Lie algebra generated by $\sigma_5$, $\sigma_7$, \dots has trivial intersection with the kernel of $[\sigma_3, \cdot ]$ and hence also with the kernel of $[L_5, \cdot ]$. It remains to prove the claim.

Let $F \in \GG$ be a cocyle of genus $g$ representing a class in $H^0(\GC_2^{(g)})$. Suppose $[L_5, F]$ is exact, and write it as $\delta(\lambda)$, with $\lambda\in \GG$ of genus $g+1$. We must show that $[\sigma_3, F]$ is exact.  By Lemma~\ref{lem:nablaclosed}, we may assume $\nabla F = 0$.

Recall that there is a graph $\gamma$ of genus $2$ such that $(\delta + \nabla)(L_5 + \gamma) = \sigma_3$ (as in \cite[Figure 2]{KWZ}).  Let $\zeta := [\gamma,F] - \nabla \lambda$. Then
$
\zeta = [L_5 + \gamma, F] - (\delta + \nabla) \lambda,
$
and, using the assumption that $\nabla F = 0$, it follows that 
\begin{equation}
  \label{equ:delta nabla zeta}
  (\delta + \nabla) \zeta = [\sigma_3, F].
\end{equation}
Note that $\zeta$ is of genus $g+2$ and $[\sigma_3,F]$ is of genus $g+3$, and $\delta$ leaves the genus invariant, while $\nabla$ increases it by one.
Hence, comparing terms of like genus in \eqref{equ:delta nabla zeta}, we find that separately $\delta \zeta = 0$ and $\nabla \zeta = [\sigma_3, F]$.  Let $D$ denote the vertex deletion operation introduced by \v{Z}ivkovi\'c in \cite[Section~3]{Zivkovic}. Then $\nabla = \delta D - D \delta$ (\cite[Proposition~3.3]{Zivkovic}) and hence 
\begin{equation}\label{equ:sigma3F}
[\sigma_3, F]= \nabla \zeta = (\delta D - D \delta)\zeta = \delta(D \zeta).
\end{equation}
This shows that $[\sigma_3, F]$ is exact, as required, and proves the proposition.
We note that \cite[Proposition~3.3]{Zivkovic} is stated in a version of the graph complex that also contains disconnected graphs. However, since the connected graphs span a direct summand with respect to $\delta$, given that \eqref{equ:sigma3F} holds in the larger complex, it must also hold for the connected part.
\end{proof}

\subsection{The support of $H^\bullet (X_{g,n})$} \label{sec:support}

Here we deduce bounds on the support of $H^\bullet (X_{g,n})$ from Theorem~\ref{thm:main WHGC}, using mixed Hodge theory and known vanishing statements for $H^\bullet_c (\MM_{g,n})$.  We also explain how Corollary~\ref{cor:odd} follows from Theorem~\ref{cor:n0}.

\begin{cor} 
The graph cohomology $H^\bullet(X_{g,n})$ is supported in degrees between $\max\{2g, 2g-2+n\}$ and $3g-2+n$.
\end{cor}

\begin{proof}
By Theorem~\ref{thm:main WHGC}, we must show that $\gr_2 H^\bullet_c (\MM_{g,n})$ is supported in the indicated range of degrees.  By Poincar\'e duality, 
we have 
\begin{equation} \label{eq:PD}
\gr_2 H^\bullet_c (\MM_{g,n}) \cong \gr_{6g-8+2n} H^{6g-6+2n-\bullet}(\MM_{g,n})^\vee.
\end{equation}
To see the upper bound, recall that $\gr_\ast H^k$ is supported in weights $\ast \leq 2k$ \cite{Deligne71}.  Hence the right hand side of \eqref{eq:PD} vanishes when $\bullet > 3g-2+n$. For $n \geq 2$, the lower bound also follows from \eqref{eq:PD}, since the virtual cohomological dimension (vcd) of $\MM_{g,n}$ is $4g-4+n$.  The lower bounds for $n = 0$ and $1$ are similar; one uses that the vcds of $\MM_{g,1}$ and $\MM_{g}$ are $4g-3$ and $4g-5$,  respectively, and that $H^{4g-3}(\MM_{g,1})$ and $H^{4g-5}(\MM_{g})$ both vanish \cite{CFP1, MSS}.
\end{proof}

Finally, we explain how Corollary~\ref{cor:odd} follows from Theorem~\ref{cor:n0} and Proposition~\ref{prop:degree3}.

\begin{proof}[Proof of Corollary~\ref{cor:odd}]
We start by recalling that
\begin{itemize}
\item  $W_0H^{2g}_c(\MM_g)$ is nonvanishing for $g =3$, $g = 5$, and $g \geq 7$, by \cite{CGP1}, and 
\item $W_0 H^{2g+3}_c(\MM_g)$ is nonvanishing for $g = 6$ and $g \geq 8$, by Proposition~\ref{prop:degree3}. 
\end{itemize}
In both cases the dimension grows exponentially with $g$. Furthermore, we know that $W_0 H^{27}_c(\MM_{10}) \neq 0;$ the corresponding graph cohomology computation is due to Bar-Natan and McKay \cite{BNM}. For this proof only, let us set
\[
\mathbb{W}_0 := \bigoplus_g W_0H^{2g}_c(\MM_g); \quad \quad \mathbb{W}_3 := \bigoplus_g W_0 H^{2g+3}_c(\MM_g); \quad \quad \mathbb{W}_{7} := W_0H^{27}_c \MM_{10}.
\]
\vskip -6 pt
Then $\bigwedge^2 \mathbb{W}_0$ contributes to $\mathbb{V}_{g,2g}$ for $g = 8$ and $g \geq 10$, and the dimension of this subspace grows exponentially with $g$.  Considering the third summand in Theorem~\ref{cor:n0}, this shows that $\gr_2 H^{2g+2}_c(\MM_g)$ is nonzero for $g = 9$ and $g \geq 11$, and grows exponentially with $g$. The corresponding statement for $H^{4g-8}(\MM_g)$ follows by Poincar\'e duality, since $H^{4g-8}(\MM_g) \cong H^{2g+2}_c(\MM_g)^\vee$. This proves Corollary~\ref{cor:odd} for $H^{4g-8}(\MM_g)$.

Corollary~\ref{cor:odd} for $H^{4g-9}(\MM_g)$ is an immediate consequence of Proposition~\ref{prop:degree3}, since $H^{4g-9}(\MM_g) \cong H^{2g+3}_c(\MM_g)^\vee$.

The remaining cases of Corollary~\ref{cor:odd} are similar to the case of $H^{4g-8}(\MM_g)$.  The subspace $\mathbb{W_0} \wedge \mathbb{W}_3$ contributes to $\mathbb{V}_{g,2g+3}$ for $g = 9$ and $g \geq 11$, and the dimension of this contribution grows exponentially with $g$.  Considering the third summand in Theorem~\ref{cor:n0} and applying Poincar\'e duality shows that $H^{4g-11}(\MM_g)$ is nonzero for $g = 10$ and $g \geq 12$, and its dimension grows exponentially with $g$.  The corresponding statements for $H^{4g-12}(\MM_g)$ are proved similarly, using the first summand in Theorem~\ref{cor:n0}

The next three cases, for $H^{4g-14}(\MM_g)$, $H^{4g-15}(\MM_g)$ and $H^{4g-16}(\MM_g)$ use the contributions of $\bigwedge^2 \mathbb{W}_3$ and $\mathbb{W}_0 \wedge \mathbb{W}_7$ to $\mathbb{V}_{g, 2g+6}$ and $\mathbb{V}_{g, 2g+7}$, respectively.  The final two cases, for $H^{4g-18}(\MM_g)$ and $H^{4g-19}(\MM_g)$ are deduced similarly from the contribution of $\mathbb{W}_3 \wedge \mathbb{W}_7$ to $\mathbb{V}_{g, 2g + 10}$.
\end{proof}

\subsection{A vanishing result in weight 0.}
Our main motivation is to use what is known in weight 0 to prove new results about the weight 2 cohomology $\gr_2 H^\bullet_c(\MM_{g,n})$.  However, it is worth noting that the information flows meaningfully in both directions, via Theorem~\ref{cor:n0}.  For instance, applying basic results from mixed Hodge theory to the weight 2 cohomology, we deduce the following vanishing result in weight 0.

\begin{prop} 
The cohomology $W_0H^{3g-1}_c(\MM_{g,2})_{as}$ vanishes for all $g \geq 1$. Also,
\[
\dim \bigoplus_g W_0 H^{3g-2}_c(\MM_{g,1}) \leq 1.
\]
\end{prop}
\noindent In other words, the graph complex $G^{(g,2)}$ computing $W_0 H_c(\MM_{g,2})$ (see section \ref{sec:svert weight0}) has no antisymmetric cohomology in the top degree, corresponding to trivalent graphs with 2 marked points. Furthermore the top degree cohomology of $\bigoplus_g G^{(g,1)}$, corresponding to trivalent graphs, is at most one-dimensional.  

\begin{proof}
The antisymmetric part $W_0H^{3g-1}_c(\MM_{g},2)_{as}$ injects into $\gr_2 H^{3g+2}_c (\MM_{g+1})$, via the second summand in Theorem~\ref{cor:n0}, and the Poincar\'e dual of $\gr_2 H^{3g+2}_c (\MM_{g+1})$ is $\gr_{6g-2}H^{3g-2}(\MM_{g+1})$, which vanishes because $\gr_* H^k$ is supported in weights $* \leq 2k$ \cite{Deligne71}.  Similarly, the first summand in Theorem~\ref{cor:n0} gives an injection from $\bigwedge^2 \bigoplus_g W_0 H^{3g-2}_c(\MM_{g,1})$ into $\bigoplus_{g'} \gr_2 H^{3g'-1}_c (\MM_{g'})$, which vanishes for the same weight reasons.  
\end{proof}

\noindent It is not known whether $W_0 H^{3g-2}(\MM_{g,1})$ vanishes, or equivalently, whether every trivalent graph with one marked point is a coboundary in $G^{(g,1)}$, for all $g$.

\section{The weight 2 Getzler-Kapranov graph complex}\label{sec:explicit wt 2}

\subsection{Generators and relations} 
The decorated graphs that generate $\GK^2_{g,n}$ have a particularly simple description. Each generator has an underlying stable graph of genus $g$, and comes with a total ordering of the structural edges, i.e., the edges that are not incident to the external vertices labeled $1, \dots, n$, subject to the relation that reordering the structural edges is multiplication by the sign of the permutation. In addition, each vertex $v$ is decorated by an element of $H^{k_v}(\bMM_{g_v,n_v})$, where $g_v$ and $n_v$ are the genus and valence of $v$, and $\sum_v k_v = 2$.

Since $H^0(\bMM)=\Q$ and $H^1(\bMM)=0$ for each moduli space $\bMM$ attached to a vertex, any generator for $\GK^2_{g,n}$ has a unique vertex $v$ with decoration in $H^2(\bMM_{g_v,n_v})$, which we call the \emph{special vertex}.  After rescaling, we may assume all other vertices are decorated with the unit $1 \in H^0(\bMM)$; for simplicity, we omit these trivial decorations. 
We can thus give a finite generating set for $\GK^2_{g,n}$ by specifying a finite generating set for $H^2(\bMM_{g',n'})$ for all $g'$ and $n'$.

Recall that $H^2(\bMM_{g,n})$ is generated by the {\em tautological classes}: $\kappa$, $\psi_1, \ldots, \psi_n$, $\delta_{irr}$, and $\delta_{a,A}=\delta_{g-a,A^c}$, for subsets $A \subset \{1,\dots,n\}$ and $0\leq a \leq g$ such that $2a+|S|\geq 2$ and $2(g-a)+|A^c|\geq 2$. Thus $\GK_{g,n}^2$ is generated by graphs in which  the special vertex $v$ is decorated by one of these tautological classes in $H^2(\bMM_{g_v,n_v})$. Here, the set $\{1, \ldots, n_v\}$ is implicitly identified with the set 
set of half-edges incident to $v$. 

\medskip 

For $g \geq 3$, the tautological classes form a basis $H^2(\bMM_{g,n})$; there are no further relations. However, for $g\leq 2$, the tautological classes satisfy relations as follows \cite[Theorem~2.2]{ArbarelloCornalba98}:
 
For $g=2$ there is one relation:
\[
5\kappa = 5\sum_{i=1}^n \psi_i + \delta_{irr} - 5 \sum_{A} \delta_{0,A} +7 \sum_{A} \delta_{1,A}.  
\]
For $g=1$, there are $n+1$ relations:
\[
 \kappa =  \sum_{i=1}^n \psi_i -\sum_{|A|\geq 2} \delta_{0,A}, \mbox{ \ \ \ \ \ \ and \ \ \ \ \ \ }
12\psi_i = \delta_{irr} + 12 \sum_{i \in A} \delta_{0,A}.
\]
For $g = 0$, the relations are generated by:
\[
  \kappa = 
   \sum_{A\notni x,y} (|A|-1) \, \delta_{0,A}, \mbox{ \ \ \ \ \ \ \ \ }
 \psi_i = \sum_{i \in A \notni x,y} \delta_{0,A}, \mbox{ \ \ \ \ \ \ and \ \ \ \ \ \ }
 \delta_{irr} =0\, .
\]
Here $i$ and $x\neq y$ run over elements of $\{1,\dots,n\}$. See also Proposition~\ref{prop:H2 resolution} for another presentation of $H^2(\bMM_{0,n})$.

\begin{ex}
The following is a generator in degree $9$ for $\GK^2_{8,1}$. The inscribed numbers in vertices $v$ representing the genera $g_v$.
\[
  \begin{tikzpicture}[scale=1]
    \node[ext,accepting, label=90:{$\scriptstyle \kappa$}] (v1) at (0,0){$\scriptstyle 3$};
    \node[ext] (v2) at (180:1){$\scriptstyle 1$};
    \node[ext] (v3) at (60:1){$\scriptstyle 0$};
    \node[ext] (v4) at (-60:1){$\scriptstyle 0$};
    \draw (v1) 
    edge(v2)  edge (v3) edge (v4) 
    (v2) edge[bend left] (v3) edge[bend right] (v4)  -- +(180:1.3) 
    (v3) edge (v4) edge[loop] (v3);
    \node (w) at (180:2.5) {$1$};
    \end{tikzpicture}
\]
The special vertex is drawn with a double cirle, and decorated by $\kappa\in H^2(\bMM_{3,3})$.
\end{ex}

\subsection{Graphical depiction of boundary and $\psi$-class decorations}\label{sec:graphical bdry psi}
We use suggestive graphical depictions of the decorations on the special vertex, indicating a decoration $\psi_i$ with an arrow on the corresponding half-edge.
\begin{align*}
  \begin{tikzpicture}[scale=1]
    \node[ext,accepting, label=90:{$\scriptstyle \psi_i$}] (v) at (0,0){$\scriptstyle h$};
    \node at (.45,.15) {$\scriptstyle i$};
    \draw (v) edge +(0:.5) edge +(60:.5) edge +(-60:.5) edge +(120:.5) edge +(180:.5) edge +(-120:.5);
  \end{tikzpicture}
  &=:
  \begin{tikzpicture}[scale=1]
    \node[ext,accepting] (v) at (0,0){$\scriptstyle h$};
    \draw (v) edge[->-] +(0:.5) edge +(60:.5) edge +(-60:.5) edge +(120:.5) edge +(180:.5) edge +(-120:.5);
  \end{tikzpicture}
\end{align*}
To indicate a decoration $\delta_{a,A}$, we replace the special vertex with two vertices connected by a marked edge.
\begin{align}\label{equ:delta two vert}
  \begin{tikzpicture}[scale=1]
    \node[ext,accepting, label=90:{$\scriptstyle \delta_{a,A}$}] (v) at (0,0){$\scriptstyle h$};
    \draw (v) edge +(0:.5) edge +(60:.5) edge +(-60:.5) edge +(120:.5) edge +(180:.5) edge +(-120:.5);
  \end{tikzpicture}
  &=:
  \begin{tikzpicture}[scale=1]
    \node[ext] (v) at (0,0){$\scriptstyle b$};
    \node[ext] (w) at (-.7,0){$\scriptstyle a$};
    \draw (v) edge[crossed] (w) edge +(0:.5) edge +(60:.5) edge +(-60:.5) 
    (w) edge +(120:.5) edge +(180:.5) edge +(-120:.5);
  \end{tikzpicture}
\end{align}
Here $b = h-a$ and $A$ is the subset of the half-edges at the special vertex that are connected to the vertex labeled $a$.  

\subsection{Graphical depiction of the differential}
The differential $\delta = \delta_{split} + \delta_{loop}$ on $\GK^k_{g,n}$ acts on each generator by taking a sum over the vertices. Roughly speaking, it amounts to splitting vertices and attaching loops at vertices in all allowable ways, and then pulling back cohomological decorations along the corresponding clutching and gluing maps between moduli spaces. When the decoration on a given vertex is trivial, the differential acts just like the differential on the standard commutative graph complex. 

Thus, we may depict the contribution ``$\Gamma \mathrm{ \ split \ } v$" from a non-special vertex $v$ with $g_v = h$ as  
\begin{equation}
\label{equ:deltasplitGK}
  \begin{tikzpicture}[baseline=-.65ex]
  \node[ext] (v) at (0,0) {$\scriptstyle h$};
  \draw (v) edge +(-.3,-.3)  edge +(-.3,0) edge +(-.3,.3) edge +(.3,-.3)  edge +(.3,0) edge +(.3,.3);
  \end{tikzpicture}
  \mapsto
  \sum_{a+b=h}\sum
  \begin{tikzpicture}[baseline=-.65ex]
  \node[ext] (v) at (0,0) {$\scriptstyle a$};
  \node[ext] (w) at (0.5,0) {$\scriptstyle b$};
  \draw (v) edge (w) (v) edge +(-.3,-.3)  edge +(-.3,0) edge +(-.3,.3)
   (w) edge +(.3,-.3)  edge +(.3,0) edge +(.3,.3);
  \end{tikzpicture}\, .
  \end{equation}
On the right-hand side, the second sum is taken over all ways of distributing the half-edges incident to the special vertex over the two new vertices of genus $a$ and $b$ such that the resulting graph is stable. 
Similarly, we may depict the contribution to $\delta_{loop}$ from a non-special vertex $v$ as
\begin{equation}  \label{equ:delta loop GK}
  \begin{tikzpicture}
  \node[ext] (v) at (0,0) {$\scriptstyle h$};
  \draw (v) edge +(-.5,-.5) edge +(0,-.5) edge +(.5,-.5);
  \end{tikzpicture}
  \mapsto 
  \begin{tikzpicture}
    \node[ext] (v) at (0,0) {$\scriptscriptstyle h-1$};
    \path[overlay] (v) edge[loop] coordinate[midway](X) (v) ;
    \path (X);
    \draw (v) edge +(-.5,-.5) edge +(0,-.5) edge +(.5,-.5) (v) ;
    \end{tikzpicture},
\end{equation}
with the right-hand side being understood as zero if $h=0$. In both cases, the edge orderings on the right-hand side are fixed so that the newly created edge is first and the relative order of the other edges is unchanged.

Next we turn to the special vertex. In what follows, it will be convenient to consider curves with marked points labeled by a finite set $S$ that is not necessarily identified with $\{1, \dots, n\}$.  We follow the usual notational convention (e.g., from \cite{ArbarelloCornalba98}), writing $\MM_{g,S}$ for the moduli space of smooth curves of genus $g$ with $|S|$ marked points labeled by a bijection to $S$, and $\bMM_{g,S}$ for its compactification by stable curves.

To describe the differential one needs to understand the pullback of the classes in $H^2(\bMM_{h,k})$ under the maps
\[
\xi \colon \bMM_{h-1,S\cup \{t,t'\} } \to \bMM_{h,S} \mbox{ \ \ \ \ \ \ and \ \ \ \ \ \ }
\theta \colon \bMM_{a,A\cup \{q\} } \to \bMM_{h,S},
\]
where the latter is obtained by attaching a fixed curve of genus $h-a$ with marked points labeled by $A^c \cup \{r\}$, where $A^c := S \smallsetminus A$.  (The map $\theta$ depends on the choice of this curve, but only up to homotopy, so the pullback map on cohomology is well-defined.)   
The relevant formulas are well-known; see \cite[Lemmas~3.2 and 3.3]{ArbarelloCornalba98}.

First, for the pullback $\theta^*$, which is related to $\delta_{split}$ at the special vertex, we have:
\[
\theta^*(\kappa) = \kappa, \ \ \ \ \ \ \ 
\theta^*(\delta_{irr}) = \delta_{irr}, \ \ \ \  \ \ \mbox{ and } \ \ \ \ \ \ 
\theta^*(\psi_i) = 
\begin{cases}
  \psi_i & \text{if $i\in S$} \\
  0 & \text{otherwise}
\end{cases}\, .
\]
  \noindent The resulting formulas for $\delta_{split}$ when applied to graphs in which the special vertex is decorated with $\kappa$, $\delta_{irr}$, or $\psi_i$, can then be depicted graphically in a way that naturally generalizes \eqref{equ:deltasplitGK}. For example, when the special vertex is decorated with $\psi_i$, we have
\begin{align} \label{eq:splitpsi}
\delta_{split}\colon
\, 
\begin{tikzpicture}[baseline=-.65ex]
  \node[ext] (v) at (0,0) {$\scriptstyle h$};
  \draw (v) edge[->-] +(-.3,-.3)  edge +(-.3,0) edge +(-.3,.3) edge +(.3,-.3)  edge +(.3,0) edge +(.3,.3);
  \end{tikzpicture}
  &\mapsto
  \sum_{a + b =h}\sum
  \begin{tikzpicture}[baseline=-.65ex]
  \node[ext] (v) at (0,0) {$\scriptstyle a$};
  \node[ext] (w) at (0.5,0) {$\scriptstyle b$};
  \draw (v) edge (w) (v) edge[->-] +(-.3,-.3)  edge +(-.3,0) edge +(-.3,.3)
   (w) edge +(.3,-.3)  edge +(.3,0) edge +(.3,.3);
  \end{tikzpicture}
\end{align}
Here, the second sum is again taken over all ways of distributing the half-edges incident to the special vertex over the two new vertices of genus $a$ and $b$ such that the resulting graph is stable.  

Next, we discuss the formulas for $\theta^*(\delta_{b,B})$. In the special case where $A = S$, we have:
\begin{align*}
  \theta^*(\delta_{b,B}) &= 
  \begin{cases}
    \delta_{2b-h,S\cup \{q\}} - \psi_q & \text{if $(b,B) = (a,S)$ or $(b,B) = (h-a, \emptyset)$;} \\
    \delta_{b,B}+\delta_{b+a-h,B\cup \{q\} } & \text{otherwise.}
  \end{cases}
  \end{align*}
  If $A\neq S$ then 
  \begin{align*}
    \theta^*(\delta_{b,B}) &= 
    \begin{cases}
      -\psi_q & \text{if $(b,B)=(a,A)$ or $(b,B)=(h-a,A^c)$}; \\
      \delta_{b,B} & \text{if $B\subset A$ and $(b,B)\neq (a,A)$}; \\
      \delta_{b+a-h, (B \smallsetminus A^c) \cup \{q\}} & \text{if $B\supset A^c$ and $(b,B)\neq (h-a,A^c)$}; \\
      0 & \text{otherwise}.
    \end{cases}
  \end{align*}

Again, the resulting action of $\delta_{split}$ on a special vertex with decoration $\delta_{b,B}$ has a  convenient graphic depiction: 
\begin{equation}\label{equ:delta split delta}
\begin{aligned}
\delta_{split}\colon
\, 
\begin{tikzpicture}[scale=.9]
  \node[ext] (v) at (0,0){$\scriptstyle b$};
  \node[ext] (w) at (-.7,0){$\scriptstyle a$};
  \draw (v) edge[crossed] (w) edge +(0:.5) edge +(60:.5) edge +(-60:.5) 
  (w) edge +(120:.5) edge +(180:.5) edge +(-120:.5);
\end{tikzpicture}
&\mapsto 
\sum_{b'+b''=b}\sum
\begin{tikzpicture}[scale=.9]
  \node[ext] (v0) at (0.7,0){$\scriptstyle b''$};
  \node[ext] (v) at (0,0){$\scriptstyle b'$};
  \node[ext] (w) at (-.7,0){$\scriptstyle a$};
  \draw (v) edge[crossed] (w) edge (v0) edge +(-60:.5) 
  (v0) edge +(0:.5) edge +(60:.5) 
  (w) edge +(120:.5) edge +(180:.5) edge +(-120:.5);
\end{tikzpicture}
+\sum_{a'+a''=a}\sum
\begin{tikzpicture}[scale=.9]
  \node[ext] (w0) at (-1.4,0){$\scriptstyle a'$};
  \node[ext] (v) at (0,0){$\scriptstyle b$};
  \node[ext] (w) at (-.7,0){$\scriptstyle a''$};
  \draw (v) edge[crossed] (w)  edge +(-60:.5) edge +(0:.5) edge +(60:.5) 
  (w) edge (w0) edge +(120:.5) 
  (w0) edge +(180:.5) edge +(-120:.5);
\end{tikzpicture}
\\&\quad -
\begin{tikzpicture}[scale=.9]
  \node[ext] (v) at (0,0){$\scriptstyle b$};
  \node[ext, accepting] (w) at (-.7,0){$\scriptstyle a$};
  \draw (v) edge +(0:.5) edge +(60:.5) edge +(-60:.5) 
  (w) edge[->-] (v) edge +(120:.5) edge +(180:.5) edge +(-120:.5);
\end{tikzpicture}
-
\begin{tikzpicture}[scale=1]
  \node[ext, accepting] (v) at (0,0){$\scriptstyle b$};
  \node[ext] (w) at (-.7,0){$\scriptstyle a$};
  \draw (v) edge[->-] (w) edge +(0:.5) edge +(60:.5) edge +(-60:.5) 
  (w) edge +(120:.5) edge +(180:.5) edge +(-120:.5);
\end{tikzpicture}
\end{aligned}
\end{equation}

Next, consider pullback $\xi^*$, which is related to $\delta_{loop}$ at the special vertex.  We have 
\[
\resizebox{.95\hsize}{!}{$
\xi^* \kappa = \kappa, \ \ \ \ \ \ \ \ \ 
   \xi^* \psi_j = \psi_j, \ \ \ \ \ \ \ \ \ \mbox { and } \ \ \ \ \ \ \ \ \ \ \
 \xi^* \delta_{a,A} = 
 \begin{cases}
  \delta_{a,A} & \text{if $h=2a$, $A=S=\emptyset$} \\
  \delta_{a,A}+\delta_{a-1,A\cup\{t,t'\} }
  &\text{otherwise}
 \end{cases}\, . 
 $}
\]
These algebraic rules correspond to the following pictorial definition of $\delta_{loop}$:
\begin{gather*}
  \begin{tikzpicture}[scale=1]
    \node[ext,accepting, label=-90:{$\kappa$}] (v) at (0,0){$\scriptstyle h$};
    \draw (v)  edge +(0:.5) edge +(60:.5) edge +(-60:.5) edge +(120:.5) edge +(180:.5) edge +(-120:.5);
  \end{tikzpicture}
  \mapsto 
  \begin{tikzpicture}[scale=1]
    \node[ext,accepting, label=-90:{$\kappa$}] (v) at (0,0){$\scriptscriptstyle h-1$};
    \draw (v) edge +(0:.5) edge +(60:.5) edge +(-60:.5) edge +(120:.5) edge +(180:.5) edge +(-120:.5) edge[loop above] (v);
  \end{tikzpicture}
  \quad \quad
  \begin{tikzpicture}[baseline=-.65ex, scale=1.0]
    \node[ext] (v) at (0,0) {$ \scriptstyle h$};
    \draw (v) edge[->-] +(-.5,-.5)  edge +(-.5,0) edge +(-.5,.5) edge +(.5,-.5)  edge +(.5,0) edge +(.5,.5);
    \end{tikzpicture}
    \mapsto
    \begin{tikzpicture}[baseline=-.65ex, scale=1.0]
      \node[ext] (v) at (0,0) {$\scriptscriptstyle h-1$};
      \draw (v) edge[->-] +(-.5,-.5)  edge +(-.5,0) edge +(-.5,.5) edge +(.5,-.5)  edge +(.5,0) edge +(.5,.5)edge[loop above] (v);
    \end{tikzpicture}
    \\
  \begin{tikzpicture}[scale=1.2]
    \node[ext] (v) at (0,0){$\scriptstyle b$};
    \node[ext] (w) at (-.7,0){$\scriptstyle a$};
    \draw (v) edge[crossed] (w) edge +(0:.5) edge +(60:.5) edge +(-60:.5) 
    (w) edge +(120:.5) edge +(180:.5) edge +(-120:.5);
  \end{tikzpicture}
  \mapsto 
  \begin{tikzpicture}[scale=1.2]
    \node[ext] (v) at (0,0){$\scriptstyle b$};
    \node[ext] (w) at (-.7,0){$\scriptscriptstyle a-1$};
    \draw (v) edge[crossed] (w) edge +(0:.5) edge +(60:.5) edge +(-60:.5) 
    (w) edge +(120:.5) edge +(180:.5) edge +(-120:.5)
    edge[loop above] (w);
  \end{tikzpicture}
  \, + \, 
  \begin{tikzpicture}[scale=1.2]
    \node[ext] (v) at (0,0){$\scriptscriptstyle b-1$};
    \node[ext] (w) at (-.7,0){$\scriptstyle a$};
    \draw (v) edge[crossed] (w) edge +(0:.5) edge +(60:.5) edge +(-60:.5) edge[loop above] (v)
    (w) edge +(120:.5) edge +(180:.5) edge +(-120:.5);
  \end{tikzpicture} \quad .
\end{gather*}
Note that these pictures are formally similar to \eqref{equ:delta loop GK}.

For the boundary class $\delta_{irr}$, we have
\[
 \xi^*\delta_{irr} =
 \delta_{irr}-\psi_t -\psi_{t'} + \sum_{a} 
 \sum_{ {t\in A,t'\notin A}} \delta_{a,A}.
\]
This means that the action of the differential $\delta_{loop}$ on a $\delta_{irr}$-decorated vertex may be depicted graphically as
\begin{equation*}
  \begin{aligned}
  \delta_{loop}\colon
  \, 
  \begin{tikzpicture}[scale=1]
    \node[ext,accepting, label=90:{$\scriptstyle \delta_{irr}$}] (v) at (0,0){$\scriptstyle h$};
    \draw (v) edge +(0:.5)  
    edge +(-60:.5) 
    edge +(180:.5) edge +(-120:.5);
  \end{tikzpicture}
  &\mapsto 
  \begin{tikzpicture}[scale=1, label distance=-1mm]
    \node[ext,accepting, label=90:{$\scriptstyle \delta_{irr}$}] (v) at (0,0){$\scriptscriptstyle h-1$};
    \draw (v) edge +(0:.5)  
    edge +(-60:.5) 
     edge +(180:.5) edge +(-120:.5)
    edge[loop] (v);
  \end{tikzpicture}
- 2
\begin{tikzpicture}[scale=1]
  \node[ext,accepting] (v) at (0,0){$\scriptscriptstyle h-1$};
  \draw (v) edge +(0:.5)  
  edge +(-60:.5) 
    edge +(180:.5) edge +(-120:.5)
  edge[loop, ->-] (v);
\end{tikzpicture}
  + \sum_{a + b=h-1}
  \begin{tikzpicture}[scale=1]
    \node[ext] (v) at (0,0){$\scriptstyle b$};
    \node[ext] (w) at (-.7,0){$\scriptstyle a$};
    \draw (v) edge[crossed] (w)  
     edge +(-60:.5) edge +(0:.5)  
    (w) edge (w0)  
    (w) edge +(180:.5) edge +(-120:.5)
    (v) edge[out=120, in=60] (w);
  \end{tikzpicture}
  \end{aligned}
  \end{equation*}
  
  \subsection{Example of a nontrivial cocycle in $\GK_{g,n}^2$}\label{sec:cor n0 example} 
Here we depict a representative in $\GK^2$ of the first nontrivial cohomology class produced by Theorem~\ref{cor:n0}. It lives in $\gr_2 H^{19}_c(\MM_8)\cong H^{19}(\GK_{8}^2)$, and arises from the subspace $W_0H_c^{6}(\MM_{3,1}) \wedge W_0 H_c^{10}(\MM_{5,1})$ of  $\mathbb V_{8,16}$.
A representative is given by
  \begin{align*}
\frac{1}{20} \sum
\left(
    \begin{tikzpicture}
  \node[ext] (v1) at (0:.5) {};
\node[ext] (v2) at (90:.5) {};
\node[ext] (v3) at (-90:.5) {};
\node[ext] (v4) at (180:.5) {};
\draw (v1) edge (v2) edge (v4) edge (v3) 
(v2) edge (v3) edge (v4)
(v3) edge (v4);
\end{tikzpicture}
\right)
\begin{tikzpicture}
\draw (0,0) edge[dashed] + (1,0);
\end{tikzpicture}
\left(
  \begin{tikzpicture}
    \node[ext] (v0) at (0,0) {};
    \node[ext] (v1) at (0:.5) {};
  \node[ext] (v2) at (72:.5) {};
  \node[ext] (v3) at (144:.5) {};
  \node[ext] (v4) at (-144:.5) {};
  \node[ext] (v5) at (-72:.5) {};
  \draw (v0) edge (v1) edge (v2) edge (v4) edge (v3) edge (v5)
  (v1) edge (v2) edge (v5)
  (v3) edge (v2) edge (v4)
  (v4) edge (v5);
  \end{tikzpicture}
+
\frac52
\begin{tikzpicture}
  \node[ext] (v1) at (-.9,.5) {};
  \node[ext] (v2) at (-.9,-.5) {};
\node[ext] (v3) at (-.4,0) {};
\node[ext] (v4) at (.9,.5) {};
\node[ext] (v5) at (.9,-.5) {};
\node[ext] (v6) at (.4,0) {};
\draw (v1) edge (v2) edge (v4) edge (v3) 
(v2) edge (v3) edge (v5)
(v3) edge (v5) edge (v6)
(v4) edge (v5) edge (v6)
(v5) edge (v6);
\end{tikzpicture}
\right),
\end{align*}
where the dashed edge should be attached to one vertex on the left and one on the right, and one decorates it with an arrow pointing left, minus one pointing right.
The result is a linear combination of 10 graphs with $\psi$-decorations: 
\begin{gather*}
      \begin{tikzpicture}[scale=0.7]
    \node[ext, accepting] (v1) at (0:.5) {};
  \node[ext] (v2) at (90:.5) {};
  \node[ext] (v3) at (-90:.5) {};
  \node[ext] (v4) at (180:.5) {};
  \draw (v1) edge (v2) edge (v4) edge (v3)
  (v2) edge (v3) edge (v4)
  (v3) edge (v4);
    \begin{scope}[xshift=2cm]
      \node[ext] (w0) at (0,0) {};
      \node[ext] (w1) at (0:.5) {};
    \node[ext] (w2) at (72:.5) {};
    \node[ext] (w3) at (144:.5) {};
    \node[ext] (w4) at (-144:.5) {};
    \node[ext] (w5) at (-72:.5) {};
    \draw (w0) edge (w1) edge (w2) edge (w4) edge (w3) edge (w5)
    (w1) edge (w2) edge (w5)
    (w3) edge (w2) edge (w4)
    (w4) edge (w5);
    \end{scope}
    \draw (v1) edge[->-] (w3);
  \end{tikzpicture}
  \ - \
  \begin{tikzpicture}[scale=0.7]
    \node[ext] (v1) at (0:.5) {};
  \node[ext] (v2) at (90:.5) {};
  \node[ext] (v3) at (-90:.5) {};
  \node[ext] (v4) at (180:.5) {};
  \draw (v1) edge (v2) edge (v4) edge (v3)
  (v2) edge (v3) edge (v4)
  (v3) edge (v4);
    \begin{scope}[xshift=2cm]
      \node[ext] (w0) at (0,0) {};
      \node[ext] (w1) at (0:.5) {};
    \node[ext] (w2) at (72:.5) {};
    \node[ext, accepting] (w3) at (144:.5) {};
    \node[ext] (w4) at (-144:.5) {};
    \node[ext] (w5) at (-72:.5) {};
    \draw (w0) edge (w1) edge (w2) edge (w4) edge (w3) edge (w5)
    (w1) edge (w2) edge (w5)
    (w3) edge (w2) edge (w4)
    (w4) edge (w5);
    \end{scope}
    \draw (w3) edge[->-] (v1);
  \end{tikzpicture}
  \ + \   \frac15 
  \begin{tikzpicture} [scale=0.7]
    \node[ext, accepting] (v1) at (0:.5) {};
  \node[ext] (v2) at (90:.5) {};
  \node[ext] (v3) at (-90:.5) {};
  \node[ext] (v4) at (180:.5) {};
  \draw (v1) edge (v2) edge (v4) edge (v3)
  (v2) edge (v3) edge (v4)
  (v3) edge (v4);
    \begin{scope}[xshift=2cm]
      \node[ext] (w0) at (0,0) {};
      \node[ext] (w1) at (0:.5) {};
    \node[ext] (w2) at (72:.5) {};
    \node[ext] (w3) at (144:.5) {};
    \node[ext] (w4) at (-144:.5) {};
    \node[ext] (w5) at (-72:.5) {};
    \draw (w0) edge (w1) edge (w2) edge (w4) edge (w3) edge (w5)
    (w1) edge (w2) edge (w5)
    (w3) edge (w2) edge (w4)
    (w4) edge (w5);
    \end{scope}
    \draw (v1) edge[->-] (w0);
  \end{tikzpicture}
  \ - \ \frac15 
  \begin{tikzpicture}[scale=0.7]
    \node[ext] (v1) at (0:.5) {};
  \node[ext] (v2) at (90:.5) {};
  \node[ext] (v3) at (-90:.5) {};
  \node[ext] (v4) at (180:.5) {};
  \draw (v1) edge (v2) edge (v4) edge (v3)
  (v2) edge (v3) edge (v4)
  (v3) edge (v4);
    \begin{scope}[xshift=2cm]
      \node[ext, accepting] (w0) at (0,0) {};
      \node[ext] (w1) at (0:.5) {};
    \node[ext] (w2) at (72:.5) {};
    \node[ext] (w3) at (144:.5) {};
    \node[ext] (w4) at (-144:.5) {};
    \node[ext] (w5) at (-72:.5) {};
    \draw (w0) edge (w1) edge (w2) edge (w4) edge (w3) edge (w5)
    (w1) edge (w2) edge (w5)
    (w3) edge (w2) edge (w4)
    (w4) edge (w5);
    \end{scope}
    \draw (w0) edge[->-] (v1);
  \end{tikzpicture}
 \ +
  \\ 
  \begin{tikzpicture}[scale=0.7]
    \node[ext, accepting] (v1) at (0:.5) {};
  \node[ext] (v2) at (90:.5) {};
  \node[ext] (v3) at (-90:.5) {};
  \node[ext] (v4) at (180:.5) {};
  \draw (v1) edge (v2) edge (v4) edge (v3)
  (v2) edge (v3) edge (v4)
  (v3) edge (v4);
  \begin{scope}[xshift=2cm]
    \node[ext] (w1) at (-.9,.5) {};
    \node[ext] (w2) at (-.9,-.5) {};
  \node[ext] (w3) at (-.4,0) {};
  \node[ext] (w4) at (.9,.5) {};
  \node[ext] (w5) at (.9,-.5) {};
  \node[ext] (w6) at (.4,0) {};
  \draw (w1) edge (w2) edge (w4) edge (w3) 
  (w2) edge (w3) edge (w5)
  (w3) edge (w5) edge (w6)
  (w4) edge (w5) edge (w6)
  (w5) edge (w6);
  \end{scope}
  \draw (v1) edge[->-] (w1);
  \end{tikzpicture}
 \ - \
  \begin{tikzpicture}[scale=0.7]
    \node[ext] (v1) at (0:.5) {};
  \node[ext] (v2) at (90:.5) {};
  \node[ext] (v3) at (-90:.5) {};
  \node[ext] (v4) at (180:.5) {};
  \draw (v1) edge (v2) edge (v4) edge (v3)
  (v2) edge (v3) edge (v4)
  (v3) edge (v4);
  \begin{scope}[xshift=2cm]
    \node[ext, accepting] (w1) at (-.9,.5) {};
    \node[ext] (w2) at (-.9,-.5) {};
  \node[ext] (w3) at (-.4,0) {};
  \node[ext] (w4) at (.9,.5) {};
  \node[ext] (w5) at (.9,-.5) {};
  \node[ext] (w6) at (.4,0) {};
  \draw (w1) edge (w2) edge (w4) edge (w3) 
  (w2) edge (w3) edge (w5)
  (w3) edge (w5) edge (w6)
  (w4) edge (w5) edge (w6)
  (w5) edge (w6);
  \end{scope}
  \draw (w1) edge[->-] (v1);
  \end{tikzpicture}
\ + \
  \begin{tikzpicture}[scale=0.7]
    \node[ext, accepting] (v1) at (0:.5) {};
  \node[ext] (v2) at (90:.5) {};
  \node[ext] (v3) at (-90:.5) {};
  \node[ext] (v4) at (180:.5) {};
  \draw (v1) edge (v2) edge (v4) edge (v3)
  (v2) edge (v3) edge (v4)
  (v3) edge (v4);
  \begin{scope}[xshift=2cm]
    \node[ext] (w1) at (-.9,.5) {};
    \node[ext] (w2) at (-.9,-.5) {};
  \node[ext] (w3) at (-.4,0) {};
  \node[ext] (w4) at (.9,.5) {};
  \node[ext] (w5) at (.9,-.5) {};
  \node[ext] (w6) at (.4,0) {};
  \draw (w1) edge (w2) edge (w4) edge (w3) 
  (w2) edge (w3) edge (w5)
  (w3) edge (w5) edge (w6)
  (w4) edge (w5) edge (w6)
  (w5) edge (w6);
  \end{scope}
  \draw (v1) edge[->-] (w2);
  \end{tikzpicture}
  \ - \
  \begin{tikzpicture}[scale=0.7]
    \node[ext] (v1) at (0:.5) {};
  \node[ext] (v2) at (90:.5) {};
  \node[ext] (v3) at (-90:.5) {};
  \node[ext] (v4) at (180:.5) {};
  \draw (v1) edge (v2) edge (v4) edge (v3)
  (v2) edge (v3) edge (v4)
  (v3) edge (v4);
  \begin{scope}[xshift=2cm]
    \node[ext] (w1) at (-.9,.5) {};
    \node[ext, accepting] (w2) at (-.9,-.5) {};
  \node[ext] (w3) at (-.4,0) {};
  \node[ext] (w4) at (.9,.5) {};
  \node[ext] (w5) at (.9,-.5) {};
  \node[ext] (w6) at (.4,0) {};
  \draw (w1) edge (w2) edge (w4) edge (w3) 
  (w2) edge (w3) edge (w5)
  (w3) edge (w5) edge (w6)
  (w4) edge (w5) edge (w6)
  (w5) edge (w6);
  \end{scope}
  \draw (w2) edge[->-] (v1);
  \end{tikzpicture}
  \\ + \ 
  \begin{tikzpicture}[scale=0.7]
    \node[ext, accepting] (v1) at (0:.5) {};
  \node[ext] (v2) at (90:.5) {};
  \node[ext] (v3) at (-90:.5) {};
  \node[ext] (v4) at (180:.5) {};
  \draw (v1) edge (v2) edge (v4) edge (v3)
  (v2) edge (v3) edge (v4)
  (v3) edge (v4);
  \begin{scope}[xshift=2.2cm]
    \node[ext] (w1) at (-.9,.5) {};
    \node[ext] (w2) at (-.9,-.5) {};
  \node[ext] (w3) at (-.4,0) {};
  \node[ext] (w4) at (.9,.5) {};
  \node[ext] (w5) at (.9,-.5) {};
  \node[ext] (w6) at (.4,0) {};
  \draw (w1) edge (w2) edge (w4) edge (w3) 
  (w2) edge (w3) edge (w5)
  (w3) edge (w5) edge (w6)
  (w4) edge (w5) edge (w6)
  (w5) edge (w6);
  \end{scope}
  \draw (v1) edge[->-] (w3);
  \end{tikzpicture}
  \ - \
  \begin{tikzpicture}[scale=0.7]
    \node[ext] (v1) at (0:.5) {};
  \node[ext] (v2) at (90:.5) {};
  \node[ext] (v3) at (-90:.5) {};
  \node[ext] (v4) at (180:.5) {};
  \draw (v1) edge (v2) edge (v4) edge (v3)
  (v2) edge (v3) edge (v4)
  (v3) edge (v4);
  \begin{scope}[xshift=2.2cm]
    \node[ext] (w1) at (-.9,.5) {};
    \node[ext] (w2) at (-.9,-.5) {};
  \node[ext, accepting] (w3) at (-.4,0) {};
  \node[ext] (w4) at (.9,.5) {};
  \node[ext] (w5) at (.9,-.5) {};
  \node[ext] (w6) at (.4,0) {};
  \draw (w1) edge (w2) edge (w4) edge (w3) 
  (w2) edge (w3) edge (w5)
  (w3) edge (w5) edge (w6)
  (w4) edge (w5) edge (w6)
  (w5) edge (w6);
  \end{scope}
  \draw (w3) edge[->-] (v1);
  \end{tikzpicture}\, .
  \end{gather*}

  \begin{rem}
    Note that Theorem \ref{cor:n0} also provides another embedding of $\mathbb V_{g,k}$ into $H_c(\MM_{g+1})$, and the element of $\mathbb V_{8,16}$ underlying the degree 19 cocycle of genus 8 depicted above gives rise to another non-trivial element of $H_c^{20}(\MM_{9})$.
  The corresponding cocycle in $\GK_9^2$ has approximately three times as many terms, and we do not draw it here. In the proof of Theorem \ref{cor:n0} and Appendix \ref{sec:direct map def} we do provide an explicit algorithm for obtaining the image of the corresponding cocycle in the quasi-isomorphic quotient $\bGK_9^2$ of $\GK_9^2$.
  \end{rem}

\subsection{The quotient $\bGK$}
Let $I_{g,n}\subset \GK_{g,n}$ be the graded subspace spanned by graphs that have at least one vertex of positive genus. 
Then define the quotient 
\[
\bGK_{g,n} = \GK_{g,n}/(I_{g,n}+\delta I_{g,n}).
\]
In other words, $\bGK_{g,n}$ is generated by graphs all of whose vertices are of genus zero. Among those graphs we impose relations of the form $\delta_{loop}\Gamma$, for $\Gamma$ running over graphs with exactly one vertex of genus 1. The quotient $\bGK_{g,n}$ inherits the weight grading from $\GK_{g,n}$, where the weight is the sum of the cohomological degrees of the vertex decorations. We denote the summand of weight $k$ by
\[
  \bGK_{g,n}^k \subset \bGK_{g,n}.
\]
In particular, the weight 0 piece $\bGK_{g,n}^{0}$ is generated by graphs without loop edges, all of whose vertices have genus zero. After rescaling, we may assume that all vertex decorations are trivial, i.e., each decoration is the identity $1 \in H^0(\bMM_{0,k})$. One main result of \cite{CGP1,CGP2} is that the quotient $\GK_{g,n}^0\to \bGK^0_{g,n}$ is a quasi-isomorphism for all $g,n$.

\medskip

In this paper, we prove an analogous result in weight 2.  The weight 2 piece $\bGK_{g,n}^{2}$ is generated by graphs with all vertices of genus zero, and one special vertex decorated by a tautological generator for $H^2(\bMM_{0,k})$. We quotient by graphs with a loop edge at a non-special vertex. Furthermore, if there is a loop edge at the special vertex, then we impose the additional relations $\delta_{0,A}=0$ for sets of half-edges $A$ that contain both half-edges of the loop.  (Since $\delta_{0,A} = \delta_{0,A^c}$ in $H^2(\bMM_{0,k})$, this also kills graphs labeled with $\delta_{0,A}$, where $A$ contains neither half-edge of the loop.)

\begin{thm}\label{thm:main quotient}
  The quotient map $\GK_{g,n}^2\to \bGK_{g,n}^2$ is a quasi-isomorphism for all $g, n$ except $g=n=1$. 
\end{thm}
\noindent For $g=n=1$ the right-hand complex is acyclic and the left-hand complex has one-dimensional cohomology, spanned by the class of a single-vertex graph with decoration $\delta_{irr}$.
\begin{equation*}
  \begin{tikzpicture}[scale=1]
    \node[ext,accepting, label=90:{$\scriptstyle \delta_{irr}$}] (v1) at (0,0){$\scriptstyle 1$};
    \node (w) at (1,0) {$1$};
    \draw (v1) edge (w);
  \end{tikzpicture}
\end{equation*}

\subsection{Resolutions of $H^2(\bMM_{0,n})$}

In the proof of Theorem~\ref{thm:main quotient}, we use the following presentation for $H^2(\bMM_{0,n})$.  We believe this is well-known to experts. However, lacking a suitable reference, we include a short proof.

\begin{lemma} \label{lem:H2 presentation}
The cohomology group $H^2(\bMM_{0,n})$ is generated by boundary classes $\delta_{0,A}$ for subsets $A \subset \{1, ..., n \}$ with $2 \leq |A| \leq n-2$, and the $\psi$-classes $\psi_1, \ldots, \psi_n$ modulo the relations
\begin{enumerate}
\item $\delta_{0,A} = \delta_{0,A^c}$;
\item \label{eq:psi relation} $\psi_i + \psi_j = \sum_{i \in A, j \in A^c} \delta_{0,A}$.
\end{enumerate}
\end{lemma}

\begin{proof}
These classes generate $H^2(\bMM_{0,n})$ and $\dim H^2 (\bMM_{0,n}) = 2^{n-1} - \binom{n}{2} - 1$ \cite{Keel92}. The relations $\delta_{0,A} = \delta_{0, A^c}$ are standard, and the quotient by these relations has dimension $2^{n-1} - 1$.  The remaining $\binom{n}{2}$ relations are independent, so it suffices to show that they hold in $H^2(\bMM_{0,n})$.  For $n =3$, this is trivial. We proceed by induction on $n$.  Assume that \eqref{eq:psi relation} holds in $H^2 (\bMM_{0,n})$.  Let $\pi \colon \bMM_{0,n+1} \to \bMM_{0,n}$ be the universal curve.  Then apply the pullback formulas $\pi^* \psi_i = \psi_i - \delta_{0, \{i,n+1\}}$ and $\pi^* \delta_{0,A} = \delta_{0,A} + \delta_{0,A \cup \{n+1\}}$ \cite[Lemma~3.1]{ArbarelloCornalba98} to deduce that \eqref{eq:psi relation} holds in $H^2(\bMM_{0,n+1})$.
\end{proof}

The following proposition is an immediate consequence of Lemma~\ref{lem:H2 presentation}.
\begin{prop}  \label{prop:H2 resolution}
There is an $\bbS_n$-equivariant short exact sequence
\begin{equation*}
  0\to \bigoplus_{1\leq i<j\leq n} \Q E_{ij}
  \to 
  \Big( \bigoplus_{1\leq i\leq n} \Q\psi_i \Big) \oplus  \bigoplus_{1\in A} \Q \delta_{0,A}   \to  H^2(\bMM_{0,n}) \to 0
\end{equation*}
in which the first arrow is given by $E_{ij} \mapsto \psi_i + \psi_j - \sum_{i \in A, j \in A^c} \delta_{0,A}$.  Note that here we identify $\delta_{0,A} = \delta_{0, A^c}$.
\end{prop}

\noindent The $\bbS_n$-action is given by permuting all labels and setting $E_{ij} = E_{ji}$ for $i > j$ and $\delta_{0,A} = \delta_{0,A^c}$ when $1 \not \in A$.

\section{Filtration by the number of vertices} 

The main idea of the proof of Theorem \ref{thm:main quotient} is to use the filtration on $\GK_{g,n}$ by the number of vertices, and look at the associated spectral sequence. The associated graded of this filtration $\gr \GK_{g,n}$ can be identified with $(\GK_{g,n}, \delta_{loop})$, i.e., we kill the part of the differential splitting vertices.
As discussed in Section~\ref{sec:thm quotient proof actual}, this complex splits into subcomplexes that are direct sums of tensor products of complexes associated to single vertices. We begin by studying these single vertex complexes
\begin{equation}\label{equ:Vgndef}
V_{g,n}
=
\bigoplus_{h} \big(H^\bullet(\bMM_{g-h,n+2h}) \otimes \Q[-1]^{\otimes h}\big)_{\bbS_2 \wr \bbS_h}
\end{equation}
where the wreath product $\bbS_2 \wr \bbS_h$ acts by permuting the last $2h$ marked points and the factors $\Q[-1]^{\otimes h}$, introducing signs via the Koszul sign rule, see section \ref{sec:notation}. The weight grading on $V_{g,n}$ is given by the cohomological degree in $H^\bullet(\bMM_{g-h,n+2h})$, and the complex decomposes into a direct sum of weight graded pieces. We denote the piece of weight $k$ by $V_{g,n}^k$; in this paper, we focus on the cases $k = 0$ and $k = 2$.

\subsection{The single vertex complex in weight zero} \label{sec:svert weight0}
In weight zero, we have 
\begin{equation} \label{eq:weight0}
  H(V_{g,n}^0)=
  \begin{cases}
    \Q & \text{for $g=0$, $n\geq 3$} \\
    0 & \text{otherwise}
  \end{cases}.
\end{equation}
To see this, note that 
\[
\big(H^0(\bMM_{g-h,n+2h}) \otimes \Q[-1]^{\otimes h}\big)_{\bbS_2 \wr \bbS_h} = \begin{cases} \Q & \text{for $h \in \{0,1\}$} \\
0 & \text{otherwise} \end{cases},
\]
and the differential on $V_{g,n}^0$ is nontrivial when $g$ is positive.

The complex $\overline{\mathsf{GK}}{}^{0}_{g,n}$ is generated by connected simple graphs of genus $g$ with $n$ marked legs.
Following \cite{CGP2}, we denote this
\[
G^{(g,n)} :=  \overline{\mathsf{GK}}{}^{0}_{g,n}. 
\]
One can deduce from \eqref{eq:weight0} that $\GK_{g,n}^0\to \bGK^0_{g,n}$ is a quasi-isomorphism. However, we do not give the argument now. Instead, we proceed directly to the weight two case and merely remark that the weight zero case can be handled by a simplified version of the argument in Section~\ref{sec:thm quotient proof actual}, filtering by the number of vertices and using \eqref{eq:weight0} in place of Proposition~\ref{prop:onevertmain}. Thus, the methods presented here give an alternative approach to the following result.

\begin{thm}[{\cite{CGP1, CGP2}}]\label{thm:weight0}
The graph cohomology $H(G^{(g,n)})=H(\bGK_{g,n}^0)$ is identified with the weight zero part $W_0H_c^\bullet(\MM_{g,n})$ of the compactly supported cohomology of the open moduli space of curves.
\end{thm}


\subsection{The single vertex complex in weight 2}
Our proof of Theorem~\ref{thm:main quotient} will combine \eqref{eq:weight0} with the following analogous statement in weight 2.
\begin{prop}\label{prop:onevertmain}
  For $g\geq 2$ and all $n$ one has 
  \[
    H(V_{g,n}^2) = 0.
  \]
  For $g=1$ one has
  \[
    H(V_{1,n}^2) \cong \Q \delta_{irr} \oplus \Big( H^2(\bMM_{0,n+2})_{\bbS_2}/ \xi^*H^2(\bMM_{1,n})\Big). 
  \]
  For $g = 0$, the differential on $V_{0,n}^2$ vanishes.
\end{prop}

We write $V^2_{g,n}$ in the form:
\begin{equation} \label{eq:V2gn}
 \resizebox{0.91\hsize}{!}{$
 0 \to H^2(\bMM_{g,n}) \to  H^2(\bMM_{g-1,n+2})_{\bbS_2} \to H^2(\bMM_{g-2,n+4})_{\bbS_2 \wr \bbS_2}  \to H^2 (\bMM_{g-3,n})_{\bbS_2 \wr \bbS_3} \to  \cdots .
 $}
\end{equation}
Here, the action of $\bbS_2 \wr \bbS_h$ is understood to include a sign that accounts for the permutation of the tensor factors $\Q[-1]^{\otimes h}$ in \eqref{equ:Vgndef}. To make this concrete, recall that $H^2(\bMM_{g,n})$ is generated by the tautological classes $\kappa$, $\psi_i$, $\delta_{irr}$ and $\delta_{a,A}=\delta_{g-a,A^c}$, with relations for $g\leq 2$ as discussed in Section \ref{sec:explicit wt 2}. Then the action of $\bbS_2 \wr \bbS_h$ is induced by a signed permutation action on this set of tautological generators, as follows.  One may think of each generator as a graph with a single internal vertex of genus $g - h$ incident to $h$ loop edges and $n$ marked legs, equipped with a decoration by a tautological generator for $H^2(\bMM_{g-h,n+2h})$. Then each copy of $\bbS_2$ acts by exchanging the half-edges in a loop (without sign). The quotient map to $\bbS_h$ is given by the permutation of the loop edges, and the induced permutation on generators is twisted by the sign of this $\bbS_h$-action.  

\begin{lemma}
As a vector space $V^2_{g,n} = \bigoplus_h H^2(\bMM_{g-h, n+2h})_{\bbS_2 \wr \bbS_h}$ is generated by the following classes
\begin{enumerate}
\item $\kappa$ or $\delta_{irr}$, for $h\in\{0,1\}$,
\item $\psi_i$, with $i \in \{1, \ldots, n\}$, for  $h \in \{0,1\}$,
\item $\psi_j$ with $j \not \in \{1, \ldots, n\}$, for $h\in\{1,2\}$
\item $\delta_{a,A}$, for $h \in \{0,1,2,3\}$.
\end{enumerate}
In particular, the complex \eqref{eq:V2gn} vanishes beyond $h = 3$.
\end{lemma}

\begin{proof}
Any other tautological generator for $H^2(\bMM_{g-h, n+2h})$ is preserved by an element of $\bbS_2 \wr \bbS_h$ that acts by a transposition on the set of loop edges.
\end{proof}

\begin{proof}[Proof of Proposition \ref{prop:onevertmain} for $g\geq 4$]
For $g\geq 4$, there are no nontrivial relations among the above generators for $V^2_{g,n}$, i.e., all relations are induced by isomorphisms of marked and decorated graphs. In this case, we can see that the complex is acyclic by constructing a null-homotopy as follows. Depict boundary classes and $\psi$-classes by marked edges and half-edges, respectively, as in Section~\ref{sec:explicit wt 2}.  Thus the generators of $V^2_{g,n}$ are depicted by graphs that have either one or two vertices.  Say that a vertex $v$ is \emph{active} if either $g_v > 0$ or if there is a loop attached to $v$.  Note that each vertex has at most one loop attached, since otherwise there is an odd automorphism interchanging two loops.

For each loop edge, there is a natural contraction, in which the loop edge is deleted and the genus of the attached vertex is increased by 1. The orientation on the resulting graph is determined as follows: after reordering, we may assume that the contracted loop was first in the loop order, and then the remaining edges keep their induced ordering. The null homotopy takes the sum of all such contractions, divided by the number of active vertices. 
\end{proof}

When $g = 0$ there is nothing to prove. When $g = 1$, $2$, or $3$, additional arguments are required. For simplicity, we give complete proofs for $n \geq 1$; the cases where $n = 0$ are similar. Throughout the remainder of the proof, we use $S$ to denote an arbitrary subset of $\{1, \ldots, n\}$, subject to  restrictions as specifically indicated. For instance, ``$\delta_{0,S}, \ 1 \in S$" refers to the $2^{n-1}$ elements of the form $\delta_{0,S}$ where $S$ is a subset of $\{1, \ldots, n\}$ that contains $1$.

\begin{proof}[Proof of Proposition \ref{prop:onevertmain} for $g= 3$]
The complex $V^2_{3,n}$ has the form
\begin{equation}\label{equ:V3 explicit}
  H^2(\bMM_{3,n}) \to H^2(\bMM_{2,n+2})_{\bbS_2} \to H^2(\bMM_{1,n+4})_{\bbS_2\wr \bbS_2} \to H^2(\bMM_{0,n+6})_{\bbS_2 \wr \bbS_3} \, .
\end{equation}
By \cite[Theorem~2.10]{ArbarelloCornalba98} we know that $ H^2(\bMM_{3,n}) \to H^2(\bMM_{2,n+2})_{\bbS_2}$ is injective.  We now check exactness at the remaining places, working from right to left. Let $s, s'$, $t,t'$, and $u,u'$ be the pairs of markings added to $\{1, ..., n\}$ at each step. We claim that $H^0(\bMM_{0,n+6})_{\bbS_2 \wr \bbS_3}$ has a basis consisting of the classes $$\delta_{0,S \cup \{s,s',t\}}, \ 1 \in S.$$ This can be seen by taking  
coinvariants of the resolution in Proposition~\ref{prop:H2 resolution}; each $E_{ij}$ is stabilized by an element of $\bbS_2 \wr \bbS_3$ that acts by a simple transposition on the set of 3 loops, and hence the coinvariants of $\bigoplus_{i,j} \Q E_{ij}$ vanish. 

Next, recall that $H^2(\bMM_{1,n+4})$ is freely generated by boundary classes, which are permuted by $\bbS_2 \wr \bbS_2$. Taking coinvariants, we see that $H^2(\bMM_{1,n+4})_{\bbS_2\wr \bbS_2}$ has a basis consisting of:
  \[
 \delta_{0,S\cup \{s\}}, \ |S|\geq 1; \quad \quad \quad
 \delta_{0,S\cup \{s,s'\}}; \quad \quad \quad
 \delta_{0,S\cup \{s,s',t\}}. 
 \]
 
Working with respect to these bases, the same chain homotopy argument used for $g \geq 4$ shows that \eqref{equ:V3 explicit} is exact at $H^2(\bMM_{0,n+6})_{\bbS_2 \wr \bbS_3}$ and at $H^2(\bMM_{1,n+4})_{\bbS_2 \wr \bbS_2}$.   It remains to check that it is exact at $H^2(\bMM_{2,n+2})_{\bbS_2}$.  
We prove this when $n > 0$ by counting dimensions and computing the Euler characteristic; the argument for $n = 0$ is similar. 
      
Assume $n > 0$. Then $H^2(\bMM_{2,n+2})_{\bbS_2}$ has a basis given by $\psi_1, \ldots, \psi_n, \psi_s$, and $\delta_{irr}$ together with: 
\[
 \resizebox{0.91\hsize}{!}{$
\delta_{0,S}, \ |S|\geq 2; \quad \quad \quad
\delta_{1,S}; \quad \quad \quad
\delta_{2,S}; \quad \quad \quad
\delta_{0,S\cup\{s\}}, \ |S|\geq 1; \quad \quad \quad
\delta_{1,S\cup\{s\}}, \ 1\in S.
$}
\]
The dimension of $H^2(\bMM_{3,n})$ is $2 \cdot 2^{n} + 1$. Counting the bases above shows that 
\[
\chi(V^2_{3,n}) = (2 \cdot 2^n + 1) -(4 \cdot 2^n + 2^{n-1}) + (3 \cdot 2^n - 1) -2^{n-1},
\]
which is $0$, as required.
\end{proof}  
  
 \begin{proof}[Proof of Proposition \ref{prop:onevertmain} for $g = 2$]
The complex $V_{2,n}^2$ has the form 
\begin{equation*}
  H^2(\bMM_{2,n}) \to  H^2(\bMM_{1,n+2})_{\bbS_2} \to H^2(\bMM_{0,n+4})_{\bbS_2 \wr \bbS_2} .
\end{equation*}
By \cite[Theorem~2.10]{ArbarelloCornalba98}, we know that the first arrow is injective. We proceed to prove exactness at the remaining two places, working from right to left.

Say $s, s'$ and $t, t'$ are the two pairs of markings added at each step. First, we claim that $H^2(\bMM_{0,n+4})_{\bbS_2 \wr \bbS_2}$ has a basis consisting of:
\begin{equation} \label{eq:basisS2S2}
\delta_{0,S \cup \{s, s'\}}, \ 1 \in S; \quad \quad \quad \delta_{0, S \cup \{s\}}, \ |S| \geq 2.
\end{equation}
To see this, consider the $\bbS_2 \wr \bbS_2$-coinvariants of the resolution of  $H^2(\bMM_{0,n+4})$ from Proposition~\ref{prop:H2 resolution}. The coinvariant space of $\bigoplus_{i,j} \Q E_{ij}$ has a basis consisting of $E_{ss'}$ together with $E_{1s}, \ldots, E_{ns}$.  Similarly, the coinvariant space of  $\Big( \bigoplus_i \Q\psi_i \Big) \oplus  \Big(\bigoplus_S \Q \delta_{0,S} \Big)$ has a basis consisting of:
\[
\psi_{s}; \quad \quad \quad \delta_{0,S \cup \{s, s'\}}, \ 1 \in S; \quad \quad \quad \delta_{0, S \cup \{s\}}, \ |S| \geq 1.
\]
By examining the differential, one checks that $E_{s,s'}$ can be used to eliminate $\psi_s$, while $E_{is}$ can be used to eliminate $\delta_{0, \{i, s\}}$ for $i \in \{1, \ldots, n \}$.  Hence the images of the remaining basis elements listed in \eqref{eq:basisS2S2} form a basis for  $H^2(\bMM_{0,n+4})_{\bbS_2\wr \bbS_2}$, as claimed.

Next, recall that $H^2(\bMM_{1,n+2})$ is freely generated by boundary classes, which are permuted by the $\bbS_2$-action. One then sees that the coinvariant space $H^2(\bMM_{1,n+2})_{\bbS_2}$ has a basis consisting of: 
\begin{equation} \label{eq:basisH1n+2}
\delta_{irr}; \quad \quad \quad \delta_{0,S}, \ |S| \geq 2; \quad \quad \quad \delta_{0, S \cup \{s \}}, \ |S| \geq 1; \quad \quad \quad \delta_{1,S}.
\end{equation}
Working with these bases, the same chain homotopy argument that works for $g \geq 4$ gives a splitting of the map to $H^2(\bMM_{0,n+4})_{\bbS_2 \wr \bbS_2}$. It remains to check that $V^2_{2,n}$ is exact at $H^2(\bMM_{1,n+2})_{\bbS_2}$.  Just as for $g = 3$, we do this by counting dimensions and computing the Euler characteristic. For simplicity, we assume $n \geq 1$; the $n = 0$ case is similar.

The dimension of $H^2(\bMM_{2,n})$ is $2^n + 2^{n-1}$. Counting the bases \eqref{eq:basisS2S2} and \eqref{eq:basisH1n+2} then shows
\[
\chi(V^2_{2,n}) = (2^n + 2^{n-1}) - (3 \cdot 2^n - n - 1) + (2^n + 2^{n-1} - n - 1),
\]
which is 0, as required.
\end{proof}

We use the following lemma in the proof of Proposition~\ref{prop:onevertmain} for $g = 1$, and also in the proof of Proposition~\ref{prop:tGKqiso}.

\begin{lemma} \label{lem:S2independent}
The classes $\delta_{0,S}$ in $H^2(\bMM_{1,n})$ for $S \subset \{1, \ldots, n \}$, $|S| \geq 2$ have linearly independent image in $H^2(\bMM_{0,n+2})_{\bbS_2}$, where $\bbS_2$ acts by transposing the last two marked points.
\end{lemma}

\begin{proof}
 These classes are $\bbS_2$-invariant, so it suffices to show that they are linearly independent in $H^2(\bMM_{0,n+2})$.  We claim that there is a curve $C$ in $\bMM_{0,n+2}$ that intersects $\delta_{0,S}$ at finitely many points and is disjoint from $\delta_{0,S'}$ for all other subsets $S' \subset \{1, \ldots, n\}$. Fix $n+2$ general points in $\mathbb{P}^1$. Let $\mathbb{G}_m$ act so that the coordinates of the points in $S$ are multiplied by $z$ and the rest are multiplied by $z^{-1}$, and let $C$ be the closure of this $\mathbb{G}_m$-orbit.  The image of $\mathbb{G}_m$ is contained in the open moduli space $\MM_{0,n+2}$, and the remaining two points of $C$ are general in $\delta_{0,S}$. This proves the claim, and the lemma follows.
 \end{proof}

\begin{proof}[Proof of Proposition \ref{prop:onevertmain} for $g= 1$]
We now consider the two-term complex $V^2_{1,n}$,
\[
 H^2(\bMM_{1,n}) \to H^2(\bMM_{0,n+2})_{\bbS_2}.
\]
Recall that $H^2(\bMM_{1,n})$ is freely generated by the boundary classes $\delta_{irr}$ and $\delta_{0,S},  \ |S| \geq 2$.  Also, $H^2(\bMM_{0,n+2})_{\bbS_2}$ is generated by the classes
\[
\delta_{0,S}, \ |S| \geq 2; \quad \quad \quad \delta_{0, S \cup \{s\}}, \ |S| \geq 1, \ |S^c| \geq 1.
\]
The differential is then given by $\delta_{irr} \mapsto 0$ and $\delta_{0,S} \mapsto \delta_{0,S}$. The statement of Proposition~\ref{prop:onevertmain} for $g = 1$ reduces to saying that the kernel of the differential is generated by $\delta_{irr}$. This follows from  Lemma~\ref{lem:S2independent}, since the classes $\delta_{0,S}, \ |S| \geq 2$ have linearly independent images in $H^2(\bMM_{0,n+2})_{\bbS_2}$. This completes the proof of Proposition~\ref{prop:onevertmain}. 
\end{proof}

\newcommand{\WW}{W}
\subsection{The filtered map of filtered complexes} \label{sec:thm quotient proof actual}
We now prove Theorem \ref{thm:main quotient}, showing that the projection
\begin{equation}\label{equ:GKbGK proj}
\GK^2_{g,n} \to\bGK^2_{g,n}   
\end{equation}
is a quasi-isomorphism for any $(g,n)\neq (1,1)$. We do this by filtering both sides by the number of vertices and showing that the induced map between spectral sequences is an isomorphism at $E_2$. From this it follows that \eqref{equ:GKbGK proj} is a quasi-isomorphism by the spectral sequence comparison theorem  \cite[Theorem 5.2.12]{Weibel}.

\begin{proof}[Proof of Theorem~\ref{thm:main quotient}]

Let $\mF^p\GK^2_{g,n}\subset \GK^2_{g,n}$ and $\mF^p\bGK^2_{g,n}\subset \bGK^2_{g,n}$ be the dg sub-vector spaces of $\GK^2_{g,n}$ and $\bGK^2_{g,n}$, respectively, spanned by graphs with at least $p$ vertices. So $\mF^\bullet \GK^2_{g,n}$ and $\mF^\bullet \bGK^2_{g,n}$ are decreasing filtrations of dg vector spaces, and \eqref{equ:GKbGK proj} respects the filtrations. We now consider the spectral sequences associated to these filtrations and the induced map between them.

On $E_0$, we have the associated graded dg vector spaces
\[
  \gr \GK^2_{g,n} \cong (\GK^2_{g,n}, \delta_{loop}) 
\quad \quad \quad \mbox{ and } \quad \quad \quad 
  \gr \bGK^2_{g,n} \cong (\bGK^2_{g,n}, 0).
\]

Say that the \emph{loopless core} of a stable graph is the stable graph obtained by contracting all loop edges and increasing the genera of the adjacent vertices accordingly. Note that $\delta_{loop}$ does not change the loopless core, and hence $(\GK^2_{g,n}, \delta_{loop})$ decomposes as a direct sum
\[
  (\GK^2_{g,n}, \delta_{loop})
  \cong 
  \bigoplus_{\Gamma} (\GK_\Gamma,\delta_{loop}),
\]
where $\Gamma$ ranges over isomorphism classes of $n$-marked stable graphs of genus $g$ without loop edges and $\GK_\Gamma$ is the dg subspace spanned by graphs with loopless core $\Gamma$.

Since $\delta_{loop}$ acts independently on all vertices we furthermore have the isomorphism of dg vector spaces 
\[
  (\GK_\Gamma,\delta_{loop})
  \cong 
  \left( \bigotimes_{v\in \mathrm{Vert}(\Gamma)}  V^{k_v}_{g_v,n_v} \otimes \Q[-1]^{\otimes E(\Gamma)} \right)_{\Aut(\Gamma)},
\]
where $E(\Gamma)$ is the set of structural edges of $\Gamma$.
Since taking coinvariants with respect to finite group actions commutes with taking cohomology we find that 
\[
  H(\GK_\Gamma,\delta_{loop})
  \cong 
  \left( \bigotimes_{v\in \mathrm{Vert}(\Gamma)}  H(V^{k_v}_{g_v,n_v}) \otimes \Q[-1]^{\otimes E(\Gamma)} \right)_{\Aut(\Gamma)}.
\]
Recall that $k_v=0$ for all vertices except one special vertex $w$ which has $k_w = 2$.  By \eqref{eq:weight0}, the right hand side vanishes if $\Gamma$ has a non-special vertex of positive genus. Otherwise, the right hand side is 
\[
\left( H(V^2_{g_w, n_w}) \otimes \Q[-1]^{\otimes E(\Gamma)} \right)_{\Aut(\Gamma)}.
\]
In other words, $H(\GK^2_{g,n}, \delta_{loop})$ can be identified with a graded vector space of stable graphs without loop edges where all vertices have genus 0 except possibly one special vertex $w$ that is decorated by $H(V^{2}_{g_w,n_w})$, oriented by a total ordering of the structural edges.

The differential on the $E_1$-page of our spectral sequence is then induced by the part $\delta_{split}$ of the original differential. Our goal is to show that the projection
\begin{equation}\label{equ:GKbGK E1 proj}
  (H(\GK^2_{g,n}, \delta_{loop}), \delta_{split})
  \to 
  (\bGK^2_{g,n}, \delta_{split}),
\end{equation}
is a quasi-isomorphism. The right-hand side is also a graded vector space of stable graphs. Now all vertices have genus 0, but loop edges are allowed, and one special vertex $w$ is decorated by $H^2(\bMM_{0,n_w})$, modulo the image of $\delta_{loop}$. Applying Proposition~\ref{prop:onevertmain}, we see that \eqref{equ:GKbGK E1 proj} is surjective, with kernel $K_{g,n}$ generated by stable graphs in which all vertices have genus 0 except one special vertex of genus 1 that is decorated with $\delta_{irr}$. Equivalently, $K_{g,n}$ is a commutative graph complex generated by graphs with one special vertex in which all vertices have valence at least 3 except the special vertex, which may also have valence 1 or 2.  We prove the acyclicity of $K_{g,n}$ as follows.  

  Let $K_1\subset K_{g,n}$ be the dg subspace spanned by graphs in which the special vertex has valence  $1$, and $K_{\geq 2}\subset K_{g,n}$ the complementary subspace spanned by graphs having a special vertex of valence $\geq 2$. The differential $\delta_{split}$ decomposes accordingly into the following pieces 
  \begin{equation*}
    \begin{tikzcd}
      K_{\geq 2} \ar[loop above]{}{}  \ar{r}{s} &  K_1 \ar[loop above]{} 
    \end{tikzcd}.
  \end{equation*}
The part $s$ splits off the special vertex as shown,
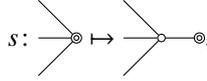
\[
  s\colon
  \begin{tikzpicture}
    \node[ext, accepting] (v) at (0,0) {};
    \draw (v) edge +(-.5,.5) edge +(-.5,0) edge +(-.5,-.5);
  \end{tikzpicture}
\mapsto 
\begin{tikzpicture}
  \node[ext, accepting] (v) at (0.5,0) {};
  \node[ext] (w) at (0,0) {};
  \draw (w) edge +(-.5,.5) edge +(-.5,0) edge +(-.5,-.5) edge (v);
\end{tikzpicture},
\]
with the special vertex marked by double circles. It is clear that this map $s\colon K_{\geq 2}\to K_1$ is injective. It is also surjective if $(g,n)\neq (1,1)$.
Applying Lemma \ref{lem:acyclicity} it follows that $K_{g,n}$ is acyclic, as required.
\end{proof}




\section{A zig-zag of quasi-isomorphisms of graph complexes}

In this section, we discuss the graph complex $X_{g,n}$ in more detail. We then describe a zig-zag of quasi-isomorphisms relating $\bGK^2_{g,n}$ to $X_{g,n}$ and thereby prove Theorem~\ref{thm:main WHGC}.

\smallskip

Recall that we consider $X_{g,n}$ only for non-negative integers $g$ and $n$ such that $2g+n \geq 3$.

\subsection{Generators for $X_{g,n}$} 
The generators for $X_{g,n}$ are simple graphs without loops or multiple edges, in which no vertices have valence 2.  The vertices of valence at least 3 are \emph{internal} and those of valence 1 are \emph{external}.  Each external vertex is decorated with an element from the set $\{\epsilon, \omega, 1, \ldots, n \}$, such that:
\begin{itemize}
\item Each label $1, \ldots, n$ appears exactly once and the label $\omega$ appears exactly twice;
\item The graph obtained by joining all external vertices labeled $\epsilon$ or $\omega$ is connected and has genus $g$;
\end{itemize}
Say that an edge with two external vertices labeled $a$ and $b$ is an \emph{$(a,b)$-edge}. We further require that
\begin{itemize}
\item No connected component is an $(\epsilon,\omega)$ or $(\omega,\omega)$ edge; 
\end{itemize}

\noindent An edge is \emph{structural} if it does not contain an external vertex with label from $\{1, \ldots, n\}$. Note that there are exactly $n$ non-structural edges. The degree of a graph is the number of structural edges plus one.  Each generator comes with a total ordering of the structural edges, and we impose the relation that permuting the structural edges is multiplication by the sign of the permutation.

The differential $\delta$ on $X_{g,n}$ is a sum of two parts 
$
\delta = \delta_{split} + \delta_{join}.
$ 
Here $\delta_{split}$ is a sum over all vertices and over all ways of splitting the vertex into two vertices joined by an edge:
\begin{align*}
  \delta_{split} \Gamma &= \sum_{v \text{ vertex} } 
  \Gamma\text{ split $v$} 
  &
  \begin{tikzpicture}[baseline=-.65ex]
  \node[int] (v) at (0,0) {};
  \draw (v) edge +(-.3,-.3)  edge +(-.3,0) edge +(-.3,.3) edge +(.3,-.3)  edge +(.3,0) edge +(.3,.3);
  \end{tikzpicture}
  &\mapsto
  \sum
  \begin{tikzpicture}[baseline=-.65ex]
  \node[int] (v) at (0,0) {};
  \node[int] (w) at (0.5,0) {};
  \draw (v) edge (w) (v) edge +(-.3,-.3)  edge +(-.3,0) edge +(-.3,.3)
   (w) edge +(.3,-.3)  edge +(.3,0) edge +(.3,.3);
  \end{tikzpicture}
  \end{align*}
  
The part $\delta_{join}$ glues together a subset $S$ of the $\epsilon$- and $\omega$-decorated external vertices, such that $|S| \geq 2$, and $S$ contains either 0 or 1 of the $\omega$-decorated vertices. It then attaches an edge to a new external vertex decorated by $\epsilon$ or $\omega$, respectively:
\begin{align*}
  \delta_{join} 
 \begin{tikzpicture}[baseline=-.8ex]
 \node[draw,circle] (v) at (0,.3) {$\Gamma$};
 \node (w1) at (-.7,-.5) {};
 \node (w2) at (-.25,-.5) {};
 \node (w3) at (.25,-.5) {};
 \node (w4) at (.7,-.5) {};
 \draw (v) edge (w1) edge (w2) edge (w3) edge (w4);
 \end{tikzpicture} 
 = 
 \ \sum_{S}  
 \begin{tikzpicture}[baseline=-.8ex]
 \node[draw,circle] (v) at (0,.3) {$\Gamma$};
 \node (w1) at (-.7,-.5) {};
 \node (w2) at (-.25,-.5) {};
 \node[int] (i) at (.4,-.5) {};
 \node (w4) at (.4,-1.3) {$\epsilon$ or $\omega$};
 \draw (v) edge (w1) edge (w2) edge[bend left] (i) edge (i) edge[bend right] (i) (w4) edge (i);
 \end{tikzpicture} \, .
 \end{align*}
In both parts, there is precisely one new structural edge, and the ordering is chosen so that the new edge comes first and the relative ordering of the old edges is preserved.  It is straightforward to verify that applying $\delta_{split}$ and $\delta_{join}$ to generators for $X_{g,n}$ produces a sum of generators for $X_{g,n}$, and that $\delta$ squares to zero (as do $\delta_{split}$ and $\delta_{join}$).

In pictures, we draw graphs in $X_{g,n}$ with filled black vertices, to distinguish them from graphs in $\GK_{g,n}$.

\begin{rem} 
The graph complex $X_{g,n}$ is closely related to the graph complexes studied by Fresse, Turchin, and the second author in the context of the embedding calculus \cite{FTW2}.
For suitable submanifolds $M\subset \R^m$ they showed that the rational homotopy groups of the (long) embedding spaces modulo immersions $\Embbar_\partial(M,\R^N)$ can be expressed as the cohomology of a graph complex $\HGC_{\bar A,N}$, in which external vertices are labeled by elements of the augmentation ideal $\bar A$ of a dg commutative algebra model of the one-point-compactification $M\cup\{\infty\}$. 
Furthermore, the cohomology groups of $\Embbar_\partial(M,\R^N)$ are computed by the Chevalley complex of $\HGC_{\bar A,N}$, and this latter complex can be identified with a complex of (possibly) disconnected graphs.

In particular, if $M$ is a union of $n$ copies of $\R^k$ and $S^{2\ell}$, then the corresponding model $\bar A$ has a basis consisting of $\epsilon$ in degree zero, $\omega$ in degree $2\ell$, and  $\{\omega_1, \ldots, \omega_n\}$ in degree $k$, with product given by $\epsilon^2 = \epsilon$, $\epsilon\omega =\omega,$ and 
  $\epsilon \omega_j = \omega^2= \omega_i \omega = \omega_i\omega_j = 0$.

Then, up to unimportant degree shifts, $X_{g,n}$ may be identified with a subquotient of the Chevalley complex of $\HGC_{\bar A,N}$, for $N$ even, corresponding to genus $g$ graphs without loop edges that contain each decoration $\omega_j$ exactly once, and that are connected after fusing the $\epsilon$- and $\omega$-decorated legs.
We shall not need this connection to the embedding calculus elsewhere in the paper, and leave the more precise comparison of the cohomology of $X_{g,n}$ and the cohomology of emedding spaces to future work.
\end{rem}

\subsection{A resolution of $\bGK^2_{g,n}$}
Our goal is to relate $\bGK_{g,n}^2$ to $X_{g,n}$ by a zigzag of quasi-isomorphisms. To this end, we now construct a quasi-isomorphism $\tGK_{g,n}^2 \xrightarrow{\sim} \bGK_{g,n}^2$ such that $\tGK_{g,n}^2$ also maps naturally onto $X_{g,n}$.  Roughly speaking $\tGK_{g,n}^2$ is constructed from $\bGK^2_{g,n}$ by resolving the space of decorations $H^2(\bMM_{0,k})$, as in Proposition~\ref{prop:H2 resolution}.
  One can also construct a quasi-isomorphism $X_{g,n} \to \bGK_{g,n}^2$; see Appendix~\ref{sec:direct map def}. The construction of the map is explicit, but the proof that it is a quasi-isomorphism uses the resolution $\tGK_{g,n}^2 \xrightarrow{\sim}\bGK_{g,n}^2$, and so we have chosen to focus on the zig-zag.

Concretely, $\tGK^2_{g,n}$ is a graph complex analogous to $\GK_{g,n}^2$ generated by graphs with the following features. All vertices have genus 0 and there is exactly one special vertex with extra decoration. There are no loop edges at non-special vertices. The special vertex is decorated with one of the following: 
\begin{itemize}
  \item A symbol $\delta_{0,A}$, where $A$ is a subset of the $k$ half-edges at the special vertex such that $2\leq |S|\leq k-2$.
    \item A symbol $\psi_i$, where $i$ is a half-edge at the special vertex. 
  \item A symbol $E_{ij}$ where $i,j$ are distinct half-edges at the special vertex.
\end{itemize}
We impose the relations $\delta_{0,A} = \delta_{0,A^c}$ and $E_{ij} = E_{ji}$.  Furthermore, if there is a loop at the special vertex with half-edges $t$ and $t'$, we impose the  \emph{loop relations}:
\begin{equation}\label{eq:looprels}
 \resizebox{0.91\hsize}{!}{$
\psi_i = 0;  \quad \quad \psi_t = \psi_{t'} = \frac12 \sum_{{t\in A, t'\notin A}} \delta_{0,A}; \quad \quad E_{it} = E_{it'} = 0; \quad \quad \delta_{0,A} = 0 \mbox { if } t,t' \in S \mbox{ (or $t,t' \in S^c$)}.
$}
\end{equation}
As usual, each generator comes equipped with an ordering of the structural edges, and permuting this ordering induces multiplication by the sign of the permutation.  Formally, these decorations should be interpreted as elements of a dg vector space resolving   $H^2(\bMM_{0,k})$, as in Proposition~\ref{prop:H2 resolution}. In this resolution, $\delta_{0,A}$ and $\psi_i$ have degree 2, while $E_{ij}$ has degree $1$. Thus, a generator with $n$ structural edges has degree $n+2$ unless its decoration is $E_{ij}$, in which case the degree is $n+1$.

We continue to depict the decoration at the special vertex by adding combinatorial features to the graph. The depictions of $\psi_i$ and $\delta_{0,S}$ are exactly as in section \ref{sec:graphical bdry psi}. 
We depict the decoration $E_{ij}$, with $i$ and $j$ distinct half-edges incident to the special vertex, by marking the two half-edges with arrows as follows.
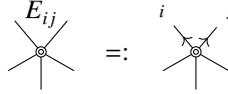
\begin{align*}
\begin{tikzpicture}
  \node[ext, accepting] (v) at (0,0){};
 \node (vi) at (130:.7){};
\node (vj) at (50:.7){};
\node (vk) at (90:.5){$E_{ij}$};
  \draw (v) edge (vi) edge (vj)
  (v) edge +(-30:.5) edge +(-90:.5) edge +(-150:.5);
\end{tikzpicture}
\ \ 
& =: \ \ 
\begin{tikzpicture}
  \node[ext, accepting] (v) at (0,0){};
  \node (vi) at (130:.7){$\scriptstyle i$};
  \node (vj) at (50:.7){$\scriptstyle j$};
  \draw (v) edge[->-] (vi) edge[->-] (vj)
  (v) edge +(-30:.5) edge +(-90:.5) edge +(-150:.5)
   ;
\end{tikzpicture}
\end{align*}

The loop relations \eqref{eq:looprels} then take on the following graphical form:
\begin{align}\label{equ:loop relations}
  \begin{tikzpicture}[scale=.9]
    \node[ext,accepting] (v) at (0,0){};
    \draw (v) edge[->-] +(0:.5) edge[loop] (v) edge +(-60:.5) edge +(180:.5) edge +(-120:.5);
  \end{tikzpicture}
  &
  =0
  &
  \begin{tikzpicture}[scale=.9]
    \node[ext,accepting] (v) at (0,0){};
    \draw (v) edge +(0:.5) edge[->-,loop] (v) edge +(-60:.5) edge +(180:.5) edge +(-120:.5);
  \end{tikzpicture}
  &=
\frac{1}{2}
  \sum  
   \begin{tikzpicture}[scale=.9]
    \node[ext] (v) at (0,0){};
    \node[ext] (w) at (-.7,0){};
    \draw (v) edge[crossed] (w) edge +(0:.5) edge[out=120, in=60] (w) edge +(-60:.5) 
    (w)  edge +(180:.5) edge +(-120:.5);
  \end{tikzpicture}
  &
  \begin{tikzpicture}[scale=.9]
    \node[ext, accepting] (v) at (0,0){};
    \draw (v) edge[->-,loop] (v)
    (v) edge[->-] +(-30:.5) edge +(-90:.5) edge +(-150:.5)
     ;
  \end{tikzpicture}
  &=
  \begin{tikzpicture}[scale=.9]
    \node[ext, accepting] (v) at (0,0){};
    \draw (v) edge[loop] (v)
    (v) edge[->] +(50:.25)
    (v) edge[->] +(130:.25)
    (v) edge +(-30:.5) edge +(-90:.5) edge +(-150:.5)
     ;
  \end{tikzpicture}
  =0
&
  \begin{tikzpicture}[scale=.9]
    \node[ext] (v) at (0,0){};
    \node[ext] (w) at (-.7,0){};
    \draw (v) edge[crossed] (w) edge +(0:.5) edge +(60:.5) edge +(-60:.5) 
    (w) edge[loop] (w) edge +(180:.5) edge +(-120:.5);
  \end{tikzpicture}
  &=0
  \end{align}

\noindent Note that any graph with two or more loops at the special vertex is set to zero by these relations. 

The differential on $\tGK^2_{g,n}$ is the sum of two parts
\[
\delta = \delta_{split} + \delta_{res}    
\]
where $\delta_{res}$ encodes the resolution of $H^2(\bMM_{0,n})$ from Proposition~\ref{lem:H2 presentation}. 
Recall that the resolution is given by 
\begin{equation*}
E_{ij} \mapsto \psi_i+\psi_j - \sum_{i\in S, j\in S^c} \delta_{0,S}.    
\end{equation*}
The differential $\delta_{res}$ applies this formula to 
the $E_{ij}$-decoration at the special vertex, with a conventional sign $(-1)^{e}$, where $e$ is the number of structural edges. Graphically, this may be depicted as 
\begin{equation}\label{equ:Eij delta int}
  \delta_{res} \colon
  \begin{tikzpicture}
    \node[ext, accepting] (v) at (0,0){};
    \node (vi) at (130:.7){};
    \node (vj) at (50:.7){};
    \draw (v) edge[->-] (vi) edge[->-] (vj)
    (v) edge +(-30:.5) edge +(-70:.5) edge +(-110:.5) edge +(-150:.5) ;
  \end{tikzpicture}
  \mapsto 
  (-1)^{e}
  \left(
  \begin{tikzpicture}
    \node[ext, accepting] (v) at (0,0){};
    \node (vi) at (130:.7){};
    \node (vj) at (50:.7){};
    \draw (v) edge (vi) edge[->-] (vj)
    (v) edge +(-30:.5) edge +(-70:.5) edge +(-110:.5) edge +(-150:.5) ;
  \end{tikzpicture}
  +
  \begin{tikzpicture}
    \node[ext, accepting] (v) at (0,0){};
    \node (vi) at (130:.7){};
    \node (vj) at (50:.7){};
    \draw (v) edge[->-] (vi) edge (vj)
    (v) edge +(-30:.5) edge +(-70:.5) edge +(-110:.5) edge +(-150:.5) ;
  \end{tikzpicture}
  -
  \sum\, 
  \begin{tikzpicture}
    \node[ext] (v) at (0,0){};
    \node[ext] (w) at (0.7,0){};
    \draw (v) edge[crossed] (w) edge +(130:.7) 
    (w) edge +(50:.7) 
    (v) edge +(-110:.5) edge +(-150:.5) 
    (w) edge +(-30:.5) edge +(-70:.5) ;
  \end{tikzpicture}
  \right)
\end{equation}
If the special vertex is decorated by $\delta_{0,A}$ or $\psi_i$, then $\delta_{res}$ vanishes. 
Note that $\delta_{res}$ is homogeneous of degree 1.

The part $\delta_{split}$ splits undecorated vertices as usual.
When the special vertex is split this is done as follows:
\begin{itemize}
  \item If decorated by $\psi_i$ the special vertex is again split as in \eqref{eq:splitpsi} on p.~\pageref{eq:splitpsi}.
  \item If decorated by $\delta_{0,S}$ the special vertex is split as in \eqref{equ:delta split delta} on p.~\pageref{equ:delta split delta}. 
  \item If decorated by $E_{ij}$ some subset $S$ of half-edges is split off that does not contain both $i$ and $j$ (but might contain neither of them), as shown: 
  \begin{equation}\label{equ:Eij split}
  \delta_{split}:
  \begin{tikzpicture}
    \node[ext, accepting] (v) at (0,0){};
    \node (vi) at (130:.7){};
    \node (vj) at (50:.7){};
    \draw (v) edge[->-] (vi) edge[->-] (vj)
    (v) edge +(-30:.5) edge +(-70:.5) edge +(-110:.5) edge +(-150:.5) ;
  \end{tikzpicture}
  \mapsto
  \sum\,
  \begin{tikzpicture}
    \node[ext, accepting] (v) at (0,0){};
    \node[ext] (w) at (130:.7) {};
    \node (vi) at (130:1.4){};
    \node (vj) at (50:.7){};
    \draw (v) edge[->-] (w) edge[->-] (vj)
    (w) edge (vi) edge +(-110:.5) edge +(-150:.5)
    (v) edge +(-30:.5) edge +(-70:.5)   ;
  \end{tikzpicture}
  +
  \sum\,
  \begin{tikzpicture}
    \node[ext, accepting] (v) at (0,0){};
    \node[ext] (w) at (50:.7) {};
    \node (vi) at (130:.7){};
    \node (vj) at (50:1.4){};
    \draw (v) edge[->-] (w) edge[->-] (vi)
    (w) edge (vj) edge +(-30:.5) edge +(-70:.5)
    (v)  edge +(-110:.5) edge +(-150:.5)  ;
  \end{tikzpicture}
  +
  \sum\,
  \begin{tikzpicture}
    \node[ext, accepting] (v) at (0,0){};
    \node[ext] (w) at (0,-.7) {};
    \node (vi) at (130:.7){};
    \node (vj) at (50:.7){};
    \draw (v) edge (w) edge[->-] (vi) edge[->-] (vj)
    (w) edge +(-110:.5) edge +(-70:.5)
    (v) edge +(-150:.5) edge +(-30:.5) ;
  \end{tikzpicture}
  \end{equation}
  \end{itemize}

\begin{prop} 
The differential on $\tGK_{g,n}^2$ is well-defined and squares to zero.
\end{prop}
\begin{proof}
First, we need to show that this differential is well-defined, i.e., that it respects the loop relations \eqref{equ:loop relations}.
That it respects the first and third relations is relatively straightforward.  To see that it respects the second loop relation, 
note that the differential applied to the left-hand side is
\begin{equation}\label{equ:welldef 1}
\sum 
\begin{tikzpicture}[scale=1]
  \node[ext, accepting] (v) at (0,0){};
  \node[ext] (w) at (-.7,0){};
  \draw (v) edge (w) edge +(0:.5) edge[out=120, in=60,->-] (w) edge +(-60:.5) 
  (w)  edge +(180:.5) edge +(-120:.5);
\end{tikzpicture}
+
\sum 
\begin{tikzpicture}
  \node[ext, accepting] (v) at (0,0){};
  \node[ext] (w) at (0,-.7) {};
  \draw (v) edge (w) edge[->-, loop] (v)
  (w) edge +(-110:.5) edge +(-70:.5)
  (v) edge +(-150:.5) edge +(-30:.5) ;
\end{tikzpicture}\,. 
\end{equation}
The differential applied to the right-hand side produces 
\begin{equation}\label{equ:welldef 2}
  -
  \sum  
   \begin{tikzpicture}[scale=1]
    \node[ext, accepting] (v) at (0,0){};
    \node[ext] (w) at (-.7,0){};
    \draw (v) edge[->-] (w) edge +(0:.5) edge[out=120, in=60] (w) edge +(-60:.5) 
    (w)  edge +(180:.5) edge +(-120:.5);
  \end{tikzpicture}  
+
\sum  
\begin{tikzpicture}[scale=1]
 \node[ext] (v) at (0,0){};
 \node[ext] (w) at (-.7,0){};
 \node[ext] (ww) at (.7,0){};
 \draw (v) edge[crossed] (w) 
 edge (ww) edge[out=120, in=60] (w) 
 (w)  edge +(180:.5) edge +(-120:.5)
 (ww)  edge +(0:.5) edge +(-60:.5) ;
\end{tikzpicture} 
+
\sum  
\begin{tikzpicture}[scale=1]
 \node[ext] (v) at (0,0){};
 \node[ext] (w) at (-.7,0){};
 \node[ext] (ww) at (-.35,-.5){};
 \draw (v) edge[crossed] (w) edge +(0:.5)
 edge (ww) 
 (w)   edge (ww) edge +(180:.5) edge +(-120:.5)
 (ww)  edge +(-60:.5) ;
\end{tikzpicture}  \, .
\end{equation}
The first term in \eqref{equ:welldef 1} agrees with that in \eqref{equ:welldef 2}, with the sign due to interchanging the two edges. The second term agrees with the second term in \eqref{equ:welldef 2} due to the relations. Each summand in the third term appears twice, with opposite sign, by applying $\delta_{split}$ to two different terms in the sum appearing in \eqref{equ:loop relations}. 
The verification that the differential respects the fourth loop relation is similar.

Now we must check that the differential squares to zero. Suppose $\Gamma$ is a generator for $\tGK^2_{g,n}$. We must show:
\begin{equation*}
 (\delta_{split}+\delta_{res})^2\Gamma
  =
  \delta_{split}^2\Gamma + (\delta_{split}\delta_{res}+\delta_{res}\delta_{split})\Gamma = 0.
\end{equation*} 
The argument for cancellation of terms that arise from splitting non-special vertices is identical to the proof that differential squares to zero on the ordinary commutative graph complex. The same argument adapts easily to cases where the special vertex is decorated with $\psi_i$ or $\delta_{0,A}$, because in these cases $\delta_{res} = 0$. 

Suppose the special vertex is decorated by $E_{ij}$. 
The standard argument shows that $\delta_{split}^2 \Gamma =0$. It remains to show $(\delta_{split}\delta_{res}+\delta_{res}\delta_{split})\Gamma=0$.
Applying $\delta_{res}$ to the right-hand side of \eqref{equ:Eij split} gives $(-1)^{e+1}$ times 
\vspace{-2 pt}
\begin{equation} \label{equ:d square many terms}
\begin{gathered}    
\sum\,
  \begin{tikzpicture}[scale=0.7]
    \node[ext, accepting] (v) at (0,0){};
    \node[ext] (w) at (130:.7) {};
    \node (vi) at (130:1.4){};
    \node (vj) at (50:.7){};
    \draw (v) edge[->-] (w) edge (vj)
    (w) edge (vi) edge +(-110:.5) edge +(-150:.5)
    (v) edge +(-30:.5) edge +(-70:.5)   ;
  \end{tikzpicture}
  +
  \sum\,
  \begin{tikzpicture}[scale=0.7]
    \node[ext, accepting] (v) at (0,0){};
    \node[ext] (w) at (130:.7) {};
    \node (vi) at (130:1.4){};
    \node (vj) at (50:.7){};
    \draw (v) edge (w) edge[->-] (vj)
    (w) edge (vi) edge +(-110:.5) edge +(-150:.5)
    (v) edge +(-30:.5) edge +(-70:.5)   ;
  \end{tikzpicture}
  -
  \sum\, 
  \begin{tikzpicture}[scale=0.7]
    \node[ext] (v) at (0,0){};
    \node[ext] (w) at (0.7,0){};
    \node[ext] (ww) at (130:0.7){};
    \draw (v) edge[crossed] (w) edge (ww)
    (ww) edge +(130:.7)  
    (w) edge +(50:.7) 
    (ww) edge +(-110:.5) edge +(-150:.5) 
    (w) edge +(-30:.5) 
    (v) edge +(-70:.5) ;
  \end{tikzpicture}
  +
  \sum\,
  \begin{tikzpicture}[scale=0.7]
    \node[ext, accepting] (v) at (0,0){};
    \node[ext] (w) at (50:.7) {};
    \node (vi) at (130:.7){};
    \node (vj) at (50:1.4){};
    \draw (v) edge (w) edge[->-] (vi)
    (w) edge (vj) edge +(-30:.5) edge +(-70:.5)
    (v)  edge +(-110:.5) edge +(-150:.5)  ;
  \end{tikzpicture}
  +
  \sum\,
  \begin{tikzpicture}[scale=0.7]
    \node[ext, accepting] (v) at (0,0){};
    \node[ext] (w) at (50:.7) {};
    \node (vi) at (130:.7){};
    \node (vj) at (50:1.4){};
    \draw (v) edge[->-] (w) edge (vi)
    (w) edge (vj) edge +(-30:.5) edge +(-70:.5)
    (v)  edge +(-110:.5) edge +(-150:.5)  ;
  \end{tikzpicture}
  \\
  -
  \sum\, 
  \begin{tikzpicture}[scale=0.7]
    \node[ext] (v) at (-.7,0){};
    \node[ext] (ww) at (50:.7){};
    \node[ext] (w) at (0,0){};
    \draw (v) edge[crossed] (w) edge +(130:.7)  
    (ww) edge +(50:.7)
    (w)  edge (ww)
    (ww)   edge +(-30:.5) edge +(-70:.5)
    (w) edge +(-110:.5)
    (v) edge +(-150:.5) ;
  \end{tikzpicture}
  +\sum\,
  \begin{tikzpicture}[scale=0.7]
  \begin{scope}[shift={(0,.5)}]
    \node[ext, accepting] (v) at (0,0){};
    \node[ext] (w) at (0,-.7) {};
    \node (vi) at (130:.7){};
    \node (vj) at (50:.7){};
    \draw (v) edge (w) edge[->-] (vi) edge (vj)
    (w) edge +(-110:.5) edge +(-70:.5)
    (v) edge +(-150:.5) edge +(-30:.5) ;
    \end{scope}
  \end{tikzpicture}
  +
  \sum\,
  \begin{tikzpicture}[scale=0.7]
    \begin{scope}[shift={(0,.5)}]
    \node[ext, accepting] (v) at (0,0){};
    \node[ext] (w) at (0,-.7) {};
    \node (vi) at (130:.7){};
    \node (vj) at (50:.7){};
    \draw (v) edge (w) edge (vi) edge[->-] (vj)
    (w) edge +(-110:.5) edge +(-70:.5)
    (v) edge +(-150:.5) edge +(-30:.5) ;
        \end{scope}
  \end{tikzpicture}
  -
  \sum\, 
  \begin{tikzpicture}[scale=0.7]
    \begin{scope}[shift={(0,.5)}]
    \node[ext] (v) at (0,0){};
    \node[ext] (w) at (0.7,0){};
    \node[ext] (ww) at (0,-0.7){};
    \draw (v) edge[crossed] (w) edge (ww) edge +(130:.7)  
    (w) edge +(50:.7) 
    (ww)    edge +(-70:.5) edge +(-110:.5)
    (w) edge +(-30:.5) 
    (v) edge +(-150:.5) ;
        \end{scope}
  \end{tikzpicture}
  -
  \sum\, 
  \begin{tikzpicture}  [scale=0.7]
  \begin{scope}[shift={(0,.5)}]
    \node[ext] (v) at (0,0){};
    \node[ext] (w) at (0.7,0){};
    \node[ext] (ww) at (0.7,-0.7){};
    \draw (v) edge[crossed] (w)  edge +(130:.7)  
    (w) edge +(50:.7) edge (ww)
    (ww)    edge +(-70:.5) edge +(-110:.5)
    (w) edge +(-30:.5) 
    (v) edge +(-150:.5) ;
          \end{scope}
  \end{tikzpicture}
  \quad .
\end{gathered}
\end{equation}
\vspace{0 pt}

\noindent We need to compare this to $\delta_{split}$ applied to the right-hand side of \eqref{equ:Eij delta int}. The fourth and seventh terms together match $\delta_{split}$ applied to the second term of \eqref{equ:Eij delta int}, with opposite sign. Likewise, the second and the eighth term of \eqref{equ:d square many terms} match $\delta_{split}$ applied to the first term of \eqref{equ:Eij delta int}. The cancellation of the remaining terms is similar.
\end{proof}

\begin{prop} \label{prop:tGKqiso}
  There is a quasi-isomorphism
  \[
      \tGK_{g,n}^2 \to \bGK_{g,n}^2
  \]
  that sets $E_{ij}$ to zero and imposes the relation $\psi_i+\psi_j - \displaystyle \sum_{i\in S, j\in S^c} \delta_{0,S}$.
  \end{prop}
\begin{proof}
  We first verify that this projection is a well-defined map of complexes. It clearly respects the differential on all generators where the special vertex is decorated by $\psi$ or $\delta_{0,A}$.    Consider a generator $\Gamma\in \tGK_{g,n}^2$ whose special vertex is decorated by $E_{ij}$. Then $\delta_{split}\Gamma$ also has an $E_{ij}$-decorated special vertex and is hence sent to zero. 
  Furthermore $\delta_{res}\Gamma$ is sent to zero since its special vertex is just decorated by the relation between $\psi$- and $\delta$-classes.

  To see that this projection is a quasi-isomorphism, we filter both sides by the number of vertices and consider the associated spectral sequences.
  We claim that the map $\tGK^2_{g,n} \to \bGK_{g,n}^2$ induces a quasi-isomorphism already on the associated graded complex. 
  Proceeding as in Section~\ref{sec:thm quotient proof actual} one can see it therefore suffices to consider the complex associated to the special vertex. 
  The differential on the associated graded of $\bGK^2_{g,n}$ is zero, and the differential on the associated graded of $\tGK^2_{g,n}$ does not add any edges. We can therefore consider separate cases according to whether or not there is a loop at the special vertex.
  
  If there is no loop, then the loop relations listed above do not come into play, and the statement boils down to Proposition~\ref{prop:H2 resolution}, which says that 
\[
\bigoplus_{1\leq i<j\leq k} \Q E_{ij}
  \to 
  \Big( \bigoplus_{1\leq i\leq k} \Q\psi_i \Big) \oplus  \bigoplus \Q \delta_{0,S}
\]
is a resolution of $H^2(\bMM_{0,k})$.

It remains to consider the case where there is a loop at the special vertex. Let $t, t'$ denote the last two marked points on $\bMM_{0,k+2}$. We must show that the complex
\begin{equation} \label{eq:loopres}
\bigoplus_{1\leq i<j\leq k} \Q E_{ij} \to \bigoplus_{A \subset \{1, \ldots, k\}} \Q \delta_{0,A \cup \{t\}},  \to H^2(\bMM_{0,k+2})_{\bbS_2} / \sim
\end{equation}
is exact, where $\sim$ is the quotient by the loop relations $\delta_{0, A} = 0$ for $A \subset \{1, \ldots, k \}$.  To see this, we begin by taking the $\bbS_2$-coinvariants of the resolution of $H^2(\bMM_{0,k+2})$ from Proposition~\ref{prop:H2 resolution}.  Then the $\bbS_2$-coinvariant space of the first term $\bigoplus \Q E_{ij}$  has a basis consisting of 
\[
E_{ij}, \ 1 \leq i < j \leq k; \quad \quad \quad E_{it}, \ 1 \leq i \leq k; \quad \quad \quad E_{tt'}. 
\]
Similarly, the $\bbS_2$-coinvariant space of the second term $(\bigoplus \Q \psi_i) \oplus (\bigoplus \Q \delta_{0,A})$ has a basis consisting of
\[
\resizebox{0.95\hsize}{!}{$
\psi_i, \ 1 \leq i \leq k; \quad \quad \psi_t; \quad \quad \delta_{0,A}, \ A \subset \{1, \ldots, k\}, \ |S| \geq 2; \quad \quad \delta_{0, A \cup \{t\}}, \ A \subset \{1, \ldots, k\}, \ |A| \geq 1, \  |A^c| \geq 1.
$}
\]
By Lemma~\ref{lem:S2independent}, the subspace spanned by the basis elements $\delta_{0,A}$, for  $\ A \subset \{1, \ldots, k\}$, $|A| \geq 2$ maps isomorphically onto its image in $H^2(\bMM_{0,k+2})_{\bbS_2}$, which is precisely the span of the loop relations.  We can then use $E_{it}$ to eliminate $\psi_i$ and $E_{tt'}$ to eliminate $\psi_t$ and see that \eqref{eq:loopres} is exact, as required.
\end{proof}

\subsection{Completing the zig-zag}

We now complete the zig-zag from $\bGK_{g,n}^2$ to $X_{g,n}$ by producing a quasi-isomorphism
\[
\Phi : \tGK_{g,n} \xrightarrow{\sim} X_{g,n}. 
\]
Let $\Gamma$ be a graph in $\tGK_{g,n}$ with $e$ structural edges. Say that the half-edges incident to external vertices are \emph{legs}. 
\begin{itemize}
\item If the special vertex of $\Gamma$ is decorated by $\delta_{0,A}$ we define $\Phi(\Gamma)=0$.
\item If the special vertex of $\Gamma$ is decorated by $E_{ij}$, then we let $\Phi(\Gamma)$ be the marked graph obtained by deleting the special vertex, making the half-edges incident to the special vertex into legs, and decorating the external vertices on $i,j$ by $\omega$ and the other newly created external vertices by $\epsilon$.
\begin{equation}\label{eq:PhiEij}
\begin{tikzpicture}
  \node[ext, accepting] (v) at (0,0){};
  \node (vi) at (130:.7){$\scriptstyle i$};
  \node (vj) at (50:.7){$\scriptstyle j$};
  \draw (v) edge[->-] (vi) edge[->-] (vj)
  (v) edge +(-30:.5) edge +(-90:.5) edge +(-150:.5)
   ;
\end{tikzpicture}
\quad \xrightarrow{\Phi}
\quad 
  \begin{tikzpicture}
    \node (vi) at (130:1.1){$\scriptstyle i$};
    \node (vj) at (50:1.1){$\scriptstyle j$};
    \node (vi0) at (130:.3) {$\scriptstyle \omega$};
    \node (vj0) at (50:.3) {$\scriptstyle \omega$};
    \node (w0) at (-150:.3) {$\scriptstyle \epsilon$};
    \node (w1) at (-30:.3) {$\scriptstyle \epsilon$};
    \node (w2) at (-90:.3) {$\scriptstyle \epsilon$};
    \draw (vi0) edge (vi) (vj0) edge (vj) 
    (w0)  edge +(-150:.5) 
    (w1) edge +(-30:.5) 
    (w2) edge +(-90:.5) 
     ;
  \end{tikzpicture}
\end{equation}
Note that the structural edges in $\Gamma$ and in $\Phi(\Gamma)$ are in 1-1-correspondence; we retain their ordering.
\item If the special vertex of $\Gamma$ is decorated by $\psi_i$ then we define $\Phi(\Gamma)$ by cutting the edge and pairing each of its half edges with an $\omega$-decorated leg, with an overall sign of $(-1)^e$.
\begin{equation}\label{equ:Phi pic1}
\begin{tikzpicture}[scale=1]
  \node[ext] (v) at (0,0){};
  \node[ext, accepting] (w) at (-.7,0){};
  \draw (v) edge +(0:.5) edge +(60:.5) edge +(-60:.5) 
  (w) edge[->-] (v) edge +(120:.5) edge +(180:.5) edge +(-120:.5);
\end{tikzpicture}
\quad \xrightarrow{\Phi}
\quad 
(-1)^e
\,
\begin{tikzpicture}[scale=1]
  \node[int] (v) at (0,0){};
  \node[int] (w) at (-1.2,0){};
  \node (v1) at (-.5,.3) {$\scriptstyle \omega$};
  \node (w1) at (-.7,-.3) {$\scriptstyle \omega$};
  \draw (v) edge +(0:.5) edge +(60:.5) edge +(-60:.5) edge (v1)
  (w) edge (w1) edge +(120:.5) edge +(180:.5) edge +(-120:.5);
\end{tikzpicture}
\end{equation}
Note in particular that if the half-edge opposite $i$ is a leg, say incident to an external vertex with some marking $j$, then $\Phi(\Gamma)$ will have a connected component that is an $(\omega, j)$ edge: 
\begin{equation}\label{equ:Phi pic1 leg}
\begin{tikzpicture}[scale=1]
  \node (v) at (0,0){$\scriptstyle j$};
  \node[ext, accepting] (w) at (-.7,0){};
  \draw 
  (w) edge[->-] (v) edge +(120:.5) edge +(180:.5) edge +(-120:.5);
\end{tikzpicture}
\quad \xrightarrow{\Phi}
\quad 
(-1)^e
\,
\begin{tikzpicture}[scale=1]
  \node (v) at (.5,0){$\scriptstyle j$};
  \node[int] (w) at (-1,0){};
  \node (v1) at (0,.3) {$\scriptstyle \omega$};
  \node (w1) at (-.5,-.3) {$\scriptstyle \omega$};
  \draw (v) edge (v1)
  (w) edge (w1)  edge +(120:.5) edge +(180:.5) edge +(-120:.5);
\end{tikzpicture}
\end{equation}
In both cases \eqref{equ:Phi pic1} and \eqref{equ:Phi pic1 leg}, $\Phi(\Gamma)$ has one more edge than $\Gamma$. We may assume, without loss of generality,  that the decorated edge (of which $i$ is one half-edge) is the first in the ordering of edges. Then we order the structural edges in $\Phi(\Gamma)$ such that the edge containing $i$ is first, the edge containing the opposite half-edge of $i$ is second, and the relative order of the remaining edges is unchanged.
\end{itemize}

\begin{lemma}
The map $\Phi$ is a map of complexes.
\end{lemma}
\begin{proof}
It is clear that the map $\Phi$ respects the degrees and genera and is hence a well-defined map of graded vector spaces. We need to verify that it commutes with the differentials,
\[
\delta\Phi(\Gamma) = \Phi(\delta \Gamma).  
\] 
Recall that the differential on $\tGK_{g,n}$ is $\delta=\delta_{split}+\delta_{res}$, while the differential on $X_{g,n}$ is $\delta=\delta_{split}+\delta_{join}$.
We furthermore decompose the operation $\delta_{split}$ on $\tGK_{g,n}$ into two terms, 
\[
\delta_{split} = \delta_{split}^{s}+ \delta_{split}^{o},  
\]
with $\delta_{split}^{s}$ splitting the special vertex and $\delta_{split}^{o}$ the other vertices.
Similarly, we write $\delta_{split}\Phi(\Gamma)=(\delta_{split}^{s}+ \delta_{split}^{o})\Phi(\Gamma)$ in $X_{g,n}$ with $\delta_{split}^{s}$ splitting the vertex that is the image of the special vertex, when the decoration on $\Gamma$ is $\psi_i$, and $\delta_{split}^{o}$ splitting the other vertices.  (When the decoration on $\Gamma$ is $E_{ij}$, we set $\delta_{split}^s \Phi(\Gamma) = 0$.) 
Since away from the special vertex the graph is not altered by $\Phi$ we have 
\[
  \Phi(\delta_{split}^{o}\Gamma)=\delta_{split}^o\Phi(\Gamma).
\]
Keep in mind that if $\Gamma$ has a $\psi_j$-decoration then $\Phi(\Gamma)$ has one more structural edge than $\Gamma$, producing an additional $-$ sign upon applying $\delta_{split}^o$. However, due to the sign $(-1)^e$ in the definition of $\Phi$ the signs on both sides of the above equation still agree.
Next, we then need to check that 
\[
(\delta_{split}^s+\delta_{join}) \Phi(\Gamma) = \Phi((\delta_{split}^s+\delta_{res}) \Gamma).  
\] 
We consider cases according to the decoration at the special vertex of $\Gamma$.

First, suppose that the special vertex of $\Gamma$ is decorated by $\delta_{0,A}$. Then we have $\Phi(\Gamma)=0$ and $\delta_{res}\Gamma=0$, so we need to check that 
\[
  \Phi(\delta_{split}^s \Gamma)=0.
\]
The only terms of $\delta_{split}^s \Gamma$ that do not themselves carry a $\delta$-decoration (and are hence sent to zero by $\Phi$) are those appearing in the last line of \eqref{equ:delta split delta} on p.~\pageref{equ:delta split delta}. Both of these terms are mapped to the same graph via $\Phi$ as depicted in \eqref{equ:Phi pic1}, with opposite orderings of the two $\omega$-edges, and hence the matching terms cancel.

Next, suppose that the special vertex of $\Gamma$ is decorated by $\psi_i$. 
We have $\delta_{res}\Gamma=0$, and, since $\Phi(\Gamma)$ has no $\epsilon$-legs, $\delta_{join}\Phi(\Gamma)=0$. We need to check that
\[
  \delta_{split}^s\Phi(\Gamma)= \Phi(\delta_{split}^s \Gamma).
\]
Applying the definitions shows that both sides have the same form 
\[
  \sum 
  \begin{tikzpicture}[scale=1]
    \node (v) at (.5,0){};
    \node[int] (w) at (-1,0){};
    \node[int] (wi) at (-1.7,0){};
    \node (v1) at (0,0) {$\scriptstyle \omega$};
    \node (w1) at (-.5,0) {$\scriptstyle \omega$};
    \draw (v) edge (v1)
    (w) edge (w1)  edge (wi) edge +(120:.5) edge +(180:.5) edge +(-120:.5)
    (wi) edge +(-.5,.5) edge +(-.5,0) edge +(-.5,-.5);
  \end{tikzpicture}\, .
\]

Finally, suppose the special vertex of $\Gamma$ is decorated by $E_{ij}$.  The differential $\delta_{res}$ sends $\Gamma$ to $(-1)^e$-times a graph whose special vertex is decorated by $\psi_i+\psi_j -\sum_{i\in S, j\in S^c}\delta_{0,S}$. The $\delta_{0,S}$-terms can be dropped upon applying $\Phi$. The two $\psi$-terms produce graphs 
\begin{equation}\label{equ:Phi delta int}
  \begin{tikzpicture}
    \node[int] (v) at (0,0){};
    \node (vi) at (130:.5){$\scriptstyle \omega$};
    \node (vii) at (130:.9){$\scriptstyle \omega$};
    \node (viii) at (130:1.3){$\scriptstyle i$};
    \node (vj) at (50:.7){$\scriptstyle j$};
    \draw (v) edge (vi) edge (vj)
    (vii) edge (viii) 
    (v) edge +(-30:.5) edge +(-90:.5) edge +(-150:.5)
     ;
  \end{tikzpicture}
  +
  \begin{tikzpicture}
    \node[int] (v) at (0,0){};
    \node (vi) at (50:.5){$\scriptstyle \omega$};
    \node (vii) at (50:.9){$\scriptstyle \omega$};
    \node (viii) at (50:1.3){$\scriptstyle j$};
    \node (vj) at (130:.7){$\scriptstyle i$};
    \draw (v) edge (vi) edge (vj)
    (vii) edge (viii) 
    (v) edge +(-30:.5) edge +(-90:.5) edge +(-150:.5)
     ;
  \end{tikzpicture},
\end{equation}
where one has to remember that in each case the $\omega$-edge adjacent to the depicted vertex is the first in the ordering.
Furthermore, note that the two factors $(-1)^e$ from the definitions of $\Phi$ and $\delta_{res}$ cancel.
Next consider the terms $\delta_{split}^s\Gamma$, which are given by replacing the special vertex by two vertices in one of three ways as follows:
\[
  \sum 
  \begin{tikzpicture}
    \node[ext, accepting] (v) at (0,0){};
    \node[ext] (w) at (-30:.5) {};
    \node (vi) at (130:.7){$\scriptstyle i$};
    \node (vj) at (50:.7){$\scriptstyle j$};
    \draw (v) edge[->-] (vi) edge[->-] (vj) edge (w)
    (v)  edge +(-150:.5) edge (w)
     (w) edge +(-30:.5) edge +(-90:.5)
     ;
  \end{tikzpicture}
  + 
  \sum 
  \begin{tikzpicture}
    \node[ext, accepting] (v) at (0,0){};
    \node[ext] (w) at (130:.6) {};
    \node (vi) at (130:1.2){$\scriptstyle i$};
    \node (vj) at (50:.7){$\scriptstyle j$};
    \draw (v)  edge[->-] (vj) edge (w)
    (v)  edge +(-30:.5) edge[->-] (w)
     (w) edge +(-150:.5) edge +(-90:.5) edge (vi)
     ;
  \end{tikzpicture}
  +
  \sum 
  \begin{tikzpicture}
    \node[ext, accepting] (v) at (0,0){};
    \node[ext] (w) at (50:.6) {};
    \node (vi) at (130:.7){$\scriptstyle i$};
    \node (vj) at (50:1.2){$\scriptstyle j$};
    \draw (v)  edge[->-] (vi) edge (w)
    (v)  edge +(-150:.5) edge[->-] (w)
     (w) edge +(-30:.5) edge +(-90:.5) edge (vj)
     ;
  \end{tikzpicture}
\]
Applying $\Phi$ we obtain terms 
\begin{equation}\label{equ:Phi delta split}
  \sum 
  \begin{tikzpicture}
    \node[int] (w) at (-30:.7) {};
    \node (vi) at (130:1.1){$\scriptstyle i$};
    \node (vj) at (50:1.1){$\scriptstyle j$};
    \node (vi0) at (130:.3) {$\scriptstyle \omega$};
    \node (vj0) at (50:.3) {$\scriptstyle \omega$};
    \node (w0) at (-150:.3) {$\scriptstyle \epsilon$};
    \node (w1) at (-30:.3) {$\scriptstyle \epsilon$};
    \draw (vi0) edge (vi) (vj0) edge (vj) 
    (w0)  edge +(-150:.5) 
    (w1) edge (w)
     (w) edge +(-30:.5) edge +(-90:.5)
     ;
  \end{tikzpicture}
  + 
  \sum 
  \begin{tikzpicture}
    \node[int] (w) at (130:.8) {};
    \node (vi) at (130:1.4){$\scriptstyle i$};
    \node (vj) at (50:1.0){$\scriptstyle j$};
    \node (vi0) at (130:.3) {$\scriptstyle \omega$};
    \node (vj0) at (50:.3) {$\scriptstyle \omega$};
    \node (w0) at (-30:.3) {$\scriptstyle \epsilon$};
    \draw (vj0) edge  (vj) (vi0) edge (w)
    (w0)  edge +(-30:.7) 
     (w) edge +(-150:.5) edge +(-90:.5) edge (vi)
     ;
  \end{tikzpicture}
  +
  \sum 
  \begin{tikzpicture}
    \node[int] (w) at (50:.8) {};
    \node (vi) at (50:1.4){$\scriptstyle j$};
    \node (vj) at (130:1.0){$\scriptstyle i$};
    \node (vi0) at (50:.3) {$\scriptstyle \omega$};
    \node (vj0) at (130:.3) {$\scriptstyle \omega$};
    \node (w0) at (-150:.3) {$\scriptstyle \epsilon$};
    \draw (vj0) edge  (vj) (vi0) edge (w)
    (w0)  edge +(-150:.7) 
     (w) edge +(-30:.5) edge +(-90:.5) edge (vi)
     ;
  \end{tikzpicture}\, .
\end{equation}
Note that in all terms there is at least one $\epsilon$-leg.
Next we look at $\Phi(\Gamma)$. In this case $\delta_{split}^s\Phi(\Gamma)=0$, since $\Phi$ removes the special vertex. Applying $\delta_{join}$ to the right-hand side of \eqref{eq:PhiEij} produces several graphs, by fusing a subset of $\epsilon$-legs together to produce a new $\epsilon$-leg, or fusing a subset of the $\epsilon$-legs to one $\omega$-leg.
The two cases in which all the $\epsilon$-hairs are fused to one $\omega$-leg precisely contribute \eqref{equ:Phi delta int}. The cases where one or more $\epsilon$-legs remain contribute the terms \eqref{equ:Phi delta split}.
To confirm the signs, note that in all cases considered the newly added edge is the first in the ordering. 

Hence we conclude that 
$
\delta_{join} \Phi(\Gamma) = \Phi(\delta_{split}^s\Gamma+\delta_{res}\Gamma),
$
as required.
\end{proof}

\begin{prop}\label{prop:Phiqiso}
The map $\Phi$ is a surjective quasi-isomorphism with kernel generated by graphs in which the special vertex is decorated by some $\delta_{0,S}$, and by symmetric combinations of graphs with $\psi$-decorations on either half of some edge:
\[
  \begin{tikzpicture}
    \node[draw, ellipse] (v) at (0,.5) {$\cdots \cdots$};
    \node[ext,accepting] (w1) at (-.5,-.5) {};
    \node[ext] (w2) at (.5,-.5) {};
    \draw (w1) edge (v.north) 
    (w1) edge (v.west)  edge[->-] (w2)
    (w2) edge (v.north) 
    (w2) edge (v.east) ;
    \node[draw, ellipse, fill, fill=white] (vd) at (0,.5) {$\cdots \cdots$};
  \end{tikzpicture}
  +
  \begin{tikzpicture}
    \node[draw, ellipse, fill, fill=white] (v) at (0,.5) {$\cdots \cdots$};
    \node[ext] (w1) at (-.5,-.5) {};
    \node[ext,accepting] (w2) at (.5,-.5) {};
    \draw (w1) edge (v.north) 
    (w1) edge (v.west)  
    (w2) edge[->-] (w1)
    (w2) edge (v.north) 
    (w2) edge (v.east) ;
    \node[draw, ellipse, fill, fill=white] (vd) at (0,.5) {$\cdots \cdots$};
  \end{tikzpicture}
  \ \ .
\]
\end{prop}
\begin{proof}
We begin by showing that $\Phi$ is surjective. Let $\Gamma$ be a graph in $X_{g,n}$. We will construct a graph $\widetilde \Gamma\in \tGK_{g,n}$ such that $\Phi(\widetilde \Gamma)=\Gamma$. 
First, consider the case where $\Gamma$ has at least one $\epsilon$-decorated external vertex. Then we build $\widetilde \Gamma$ by joining the $\epsilon$- and $\omega$-decorated external vertices to a new internal vertex, decorated by $E_{ij}$, with $i,j$ corresponding to the legs at the two $\omega$ decorations:
   \begin{align*}
    \begin{tikzpicture}
      \node[draw, ellipse] (v) at (0,0) {$\cdots \cdots$};
      \node (vo1) at (0,-0.7) {$\omega$};
      \node (vo2) at (0.5,-0.7) {$\omega$};
      \node (ve1) at (-.5,-0.7) {$\epsilon$};
      \node (ve2) at (-.8,-0.7) {$\epsilon$};
      \draw (v.south east) edge (vo2) 
      (v.south) edge (vo1) 
      (v.south west) edge (ve1) 
      (v.west) edge (ve2);
    \end{tikzpicture}
&\mapsto 
\begin{tikzpicture}
  \node[draw, ellipse] (v) at (0,0) {$\cdots \cdots$};
  \node[ext,accepting] (w) at (0,-1) {};
  \draw (w) edge[->-, bend right] (v.south east) 
  (w) edge[->-] (v.south) 
  (w) edge (v.south west) 
  (w) edge[bend left] (v.west) ;
\end{tikzpicture} \ .
  \end{align*}
 This cannot produce a ``forbidden graph" with decoration $E_{it}$ or $E_{tt'}$, where $t$ and $t'$ are the half-edges of a loop, because, in defining $X_{g,n}$, we excluded graphs with $(\epsilon, \omega)$- and $(\omega,\omega)$-edges. By construction, $\Phi(\widetilde \Gamma)=\Gamma$. 
  
  It remains to consider the case where $\Gamma$ has no $\epsilon$-decorated external vertices.
  Then we build $\widetilde \Gamma$ by deleting the two $\omega$-decorated external vertices and their incident legs, joining the two unpaired half-edges into a new edge, and adding a $\psi$-decoration on one of the two, as shown.
  
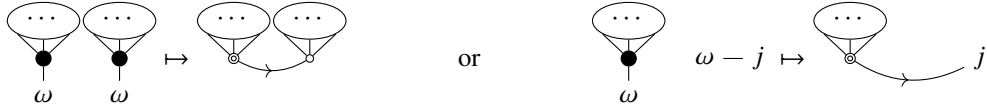
\begin{align*}
  \begin{tikzpicture}
    \node[int] (v) at (0,0) {};
    \node (vo) at (0,-0.5) {$\omega$};
    \node[draw, ellipse] (vd) at (0,.5) {$\cdots$};
    \draw (v) edge (vo) edge (vd.south) edge (vd.south west) edge (vd.south east);
    \node[int] (w) at (1,0) {};
    \node (wo) at (1,-0.5) {$\omega$};
    \node[draw, ellipse] (wd) at (1,.5) {$\cdots$};
    \draw (w) edge (wo) edge (wd.south) edge (wd.south west) edge (wd.south east);
  \end{tikzpicture}
  &\mapsto 
  \begin{tikzpicture}
    \node[ext,accepting] (v) at (0,0) {};
    \node[draw, ellipse] (vd) at (0,.5) {$\cdots$};
    \node[ext] (w) at (1,0) {};
    \node[draw, ellipse] (wd) at (1,.5) {$\cdots$};
    \draw (v)  edge (vd.south) edge (vd.south west) edge (vd.south east) edge[->-,bend right] (w);
    \draw (w) edge (wd.south) edge (wd.south west) edge (wd.south east);
  \end{tikzpicture}
  & &
  \text{or}
  &
  \begin{tikzpicture}
    \node[int] (v) at (0,0) {};
    \node (vo) at (0,-0.5) {$\omega$};
    \node[draw, ellipse] (vd) at (0,.5) {$\cdots$};
    \draw (v) edge (vo) edge (vd.south) edge (vd.south west) edge (vd.south east);
    \node (w) at (1.7,0) {$j$};
    \node (wo) at (1,0) {$\omega$};
    \draw (w) edge (wo) ;
  \end{tikzpicture}
  &\mapsto 
  \begin{tikzpicture}
    \node[ext, accepting] (v) at (0,0) {};
    \node[draw, ellipse] (vd) at (0,.5) {$\cdots$};
    \node (w) at (1.7,0) {$j$};
    \draw (v) edge[->-,bend right] (w) edge (vd.south) edge (vd.south west) edge (vd.south east);
  \end{tikzpicture}
  \end{align*}
In either case, $\Phi(\widetilde \Gamma)=\Gamma$.
  Here we also use that at least one of the $\omega$-legs is adjacent to an internal vertex. Indeed, if this were not the case, then the graph would be either a single $(\omega,\omega)$-edge, or the union of an $(\omega,1)$-edge and an $(\omega,2)$-edge. In either case, we would have $2g + n < 3$, which we have excluded from consideration.
  
The map takes graphs with decoration $\delta_{0,S}$ to zero, and the orientation data ensures that symmetric combinations of $\psi$-decorations on paired half-edges also map to zero.  Otherwise, distinct generators for $\tGK_{g,n}$ map to distinct generators of $X_{g,n}$, so nothing else is in the kernel.
  
It remains to check that $J:=\ker \Phi$ is acyclic. Decompose $J=J_\delta\oplus J_\psi$, where $J_\delta$ and $J_\psi$ are linear combinations of graphs with $\delta$- or $\psi$-decorations, respectively.
  The differential $\delta_{split}$ on $J$ then splits into the following pieces 
  \[
  \begin{tikzcd}
    J_\delta \ar{r}{u} \ar[loop above]{} & J_\psi \ar[loop above]{}
  \end{tikzcd}\, .
  \]
Note that $u$ appears in the last terms in \eqref{equ:delta split delta} on p.~\pageref{equ:delta split delta}, and takes the form 
  \[
  u : 
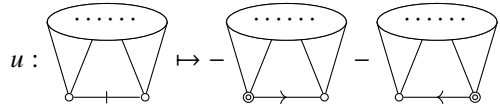
  \begin{tikzpicture}
    \node[draw, ellipse] (v) at (0,.5) {$\cdots \cdots$};
    \node[ext] (w1) at (-.5,-.5) {};
    \node[ext] (w2) at (.5,-.5) {};
    \draw (w1) edge (v.north) 
    (w1) edge (v.west)  
    edge[crossed] (w2)
    (w2) edge (v.north) 
    (w2) edge (v.east) ;
    \node[draw, ellipse, fill, fill=white] (vd) at (0,.5) {$\cdots \cdots$};
  \end{tikzpicture}
  \mapsto 
  -
  \begin{tikzpicture}
    \node[draw, ellipse] (v) at (0,.5) {$\cdots \cdots$};
    \node[ext,accepting] (w1) at (-.5,-.5) {};
    \node[ext] (w2) at (.5,-.5) {};
    \draw (w1) edge (v.north) 
    (w1) edge (v.west)  edge[->-] (w2)
    (w2) edge (v.north) 
    (w2) edge (v.east) ;
    \node[draw, ellipse, fill, fill=white] (vd) at (0,.5) {$\cdots \cdots$};
  \end{tikzpicture}
  -
  \begin{tikzpicture}
    \node[draw, ellipse, fill, fill=white] (v) at (0,.5) {$\cdots \cdots$};
    \node[ext] (w1) at (-.5,-.5) {};
    \node[ext,accepting] (w2) at (.5,-.5) {};
    \draw (w1) edge (v.north) 
    (w1) edge (v.west)  
    (w2) edge[->-] (w1)
    (w2) edge (v.north) 
    (w2) edge (v.east) ;
    \node[draw, ellipse, fill, fill=white] (vd) at (0,.5) {$\cdots \cdots$};
  \end{tikzpicture}
  \]
In these graphical depictions, $u$ replaces a marked edge with a symmetric combination of $\psi$-decorations on its paired half-edges. This map gives a bijection between bases for $J_\delta$ and for $J_\psi$, and is hence an isomorphism.
It then follows by Lemma \ref{lem:acyclicity} that $J$ is acyclic, as required. 
  \end{proof}
  
  \noindent Combining Propositions \ref{prop:tGKqiso} and \ref{prop:Phiqiso} gives the desired zig-zag of quasi-isomorphisms
      $
          \bGK_{g,n} \leftarrow \tGK_{g,n} \xrightarrow{\Phi} X_{g,n}.
      $

  \section{A quasi-isomorphic subcomplex}\label{sec:classes}
  Although we do not fully understand the cohomology of $X_{g,n}$, we can describe many nontrivial classes.  
We do so by identifying a quasi-isomorphic subcomplex $\Xast_{g,n}$ and constructing an involution on $\Xast_{g,n}$ that simplifies the differential. In this simplified subcomplex, we find direct summands whose cohomology we can either compute or bound from below.

\subsection{A quasi-isomorphic subcomplex}
Recall that every connected component of a generator $\Gamma$ for $X_{g,n}$ contains an external vertex labeled by $\epsilon$ or $\omega$. 
\begin{defi}
A connected component that contains an $\omega$-decoration is an \emph{$\omega$-component}.  All other connected components are \emph{$\epsilon$-components}.
\end{defi}

\begin{defi} \label{def:Xast}
Let $\Xast_{g,n} \subset X_{g,n}$ be the subspace spanned by generators in which the union of all $\epsilon$-components contains no internal vertices and at most one decoration from $\{1,\dots,n\}$.
\end{defi}

\noindent In other words, if $\Gamma$ is a generator for $\Xast_{g,n}$, then the union of its $\epsilon$-components is either empty or consists of an $(\epsilon, \epsilon)$-edge, an $(\epsilon, j)$-edge, or one of each. Here are two examples:
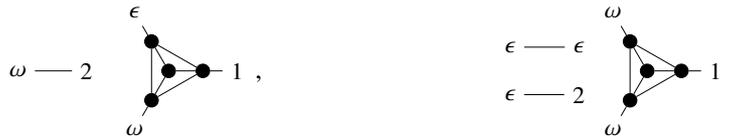
\begin{align*}
  &
  \scalebox{.9}{\begin{tikzpicture}
    \node[int] (v1) at (0:.5) {};
    \node[int] (v2) at (120:.5) {};
    \node[int] (v3) at (-120:.5) {};
    \node[int] (v0) at (0,0) {};
    \node (w1) at (0:1) {$1$};
    \node (w2) at (120:1) {$\epsilon$};
    \node (w3) at (-120:1) {$\omega$};
    \node (u1) at (-2.2,0) {$\omega$};
    \node (u2) at (-1.2,0) {$2$};
    \draw (v0) edge (v1) edge (v2) edge (v3)
        (v1) edge (v2) edge (v3) edge (w1)
         (v2) edge (v3) edge (w2) (v3) edge (w3);
    \draw (u1) edge (u2);
    \end{tikzpicture}}\, ,
    &
    &\scalebox{.9}{\begin{tikzpicture}
      \node[int] (v1) at (0:.5) {};
      \node[int] (v2) at (120:.5) {};
      \node[int] (v3) at (-120:.5) {};
      \node[int] (v0) at (0,0) {};
      \node (w1) at (0:1) {$1$};
      \node (w2) at (120:1) {$\omega$};
      \node (w3) at (-120:1) {$\omega$};
      \node (u1) at (-2,-.35) {$\epsilon$};
      \node (u3) at (-2,.35) {$\epsilon$};
      \node (u4) at (-1,.35) {$\epsilon$};
      \node (u2) at (-1,-.35) {$2$};
      \draw (v0) edge (v1) edge (v2) edge (v3)
          (v1) edge (v2) edge (v3) edge (w1)
           (v2) edge (v3) edge (w2) (v3) edge (w3);
      \draw (u1) edge (u2) (u3) edge (u4);
      \end{tikzpicture}} \ .
\end{align*}

\begin{lemma} \label{lem:subcomplex}
  The graded subspace $\Xast_{g,n}\subset X_{g,n}$ is a subcomplex.
\end{lemma}
\begin{proof}
Let $\Gamma$ be a generator for $\Xast_{g,n}$.  The differential is a sum of two parts $\delta_{split} + \delta_{join}$. The part $\delta_{split}$ acts separately on the internal vertices of each connected component. It cannot create internal vertices in a component that does not already have any, and it does change the decorations. So $\delta_{split} \Gamma$ is a linear combination of generators for $\Xast_{g,n}$.

Now consider the terms that appear in $\delta_{join} \Gamma$, which are obtained by joining external vertices of one or several components.  If at least one of the joined components is an $\omega$-component, then the resulting graph is a generator for $\Xast_{g,n}$. However, if all components that are joined are $\epsilon$-components, then the resulting graph has either a loop or two $\epsilon$-labeled external vertices adjacent to the same internal vertex, and hence is zero in $X_{g,n}$.
\end{proof}

\begin{prop}\label{prop:WHGC prime WHGC inclusion qiso}
  The inclusion $\Xast_{g,n}\subset X_{g,n}$ is a quasi-isomorphism.
\end{prop}
\begin{proof}
Filter the complexes $\Xast_{g,n}$ and $X_{g,n}$ by the number of internal vertices in $\omega$-components. 
We claim that the inclusion induces a quasi-isomorphism between the associated graded complexes with respect to this filtration.

The associated graded complexes have differential $\delta_{join}^{\epsilon,\epsilon}+\delta_{split}^\epsilon$, with $\delta_{join}^{\epsilon,\epsilon}$ the part of $\delta_{join}^\epsilon$ (as defined in the proof of Lemma~\ref{lem:subcomplex}) that joins $\epsilon$-legs of $\epsilon$-components only and $\delta_{split}^\epsilon$ the part of $\delta_{split}$ that splits vertices in $\epsilon$-components only. Note that the differential on $\Xast_{g,n}$ is 0.
Furthermore, we have the direct sum decomposition 
\[
  (X_{g,n},\delta_{join}^{\epsilon,\epsilon}+\delta_{split}^\epsilon)
  =
  \Xast_{g,n}  \oplus (W, \delta_{join}^{\epsilon,\epsilon}+\delta_{split}^\epsilon),
\]
with $W \subset X_{g,n}$ spanned by graphs that are not generators for $\Xast_{g,n}$. 
Generators of $W$ are graphs in which the $\epsilon$-components contain an internal vertex  or more than one leg labelled $1,\dots,n$.
We further decompose 
\[
W= W' \oplus W'',  
\]
where $W'$ is spanned by graphs that do not contain an $(\epsilon, \epsilon)$-edge, and $W''$ is spanned by those that do.
Note that $W''$ is isomorphic (up to degree shift) to the $W'$ summand that arises in $X_{g-1,n}$, so it suffices to show that $W'$ is acyclic. 
To this end consider the decomposition of graded vector spaces 
\[
  W'=  
  \begin{tikzcd}[column sep=.1cm]
  W_1'  \ar[loop above]{} & 
  \oplus & 
  W_{\geq 2}' \ar[bend right]{ll}[above]{\delta_{join}^{\epsilon,\epsilon}}  \ar[loop above]{}
  \end{tikzcd},
\]
with $W_1'$ (resp. $W_{\geq 2}'$) spanned by graphs that have 1 (resp. $\geq 2$) $\epsilon$-decorations in $\epsilon$-components. 
Using again Lemma \ref{lem:acyclicity} it suffices to check that the map (part of the differential $\delta_{join}^{\epsilon,\epsilon}$)
\[
  W_{\geq 2}' \to  W_1'
\]
is an isomorphism. Combinatorially, this map joins all $\epsilon$-decorated external vertices in $\epsilon$-components, attaching a new internal vertex together with an $\epsilon$-leg:
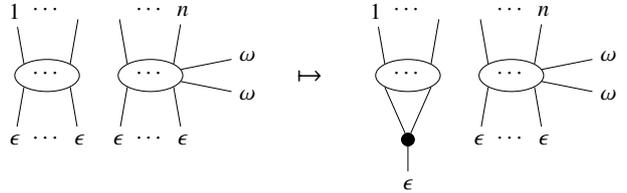
\begin{align*}
\scalebox{.85}{
\begin{tikzpicture}
\node[draw, ellipse, minimum width=1cm, minimum height=.5cm] (v) at (0,0) {$\cdots$};
\node (u1) at (-.5,1) {$1$};
\node (u2) at (0,1) {$\cdots$};
\node (u3) at (.5,1) {};
\node (x1) at (-.5,-1) {$\epsilon$};
\node (x2) at (0,-1) {$\cdots$};
\node (x3) at (.5,-1) {$\epsilon$};
\draw 
(v.north east) edge (u3) (v.north west) edge (u1)
(v.south east) edge (x3) (v.south west) edge (x1);
\begin{scope}[xshift=1.6cm]
\node[draw, ellipse, minimum width=1cm, minimum height=.5cm] (v) at (0,0) {$\cdots$};
\node (w1) at (1.5,.3) {$\omega$};
\node (w2) at (1.5,-.3) {$\omega$};
\node (u1) at (-.5,1) {};
\node (u2) at (0,1) {$\cdots$};
\node (u3) at (.5,1) {$n$};
\node (x1) at (-.5,-1) {$\epsilon$};
\node (x2) at (0,-1) {$\cdots$};
\node (x3) at (.5,-1) {$\epsilon$};
\draw (v) edge (w1) (v) edge (w2) 
(v.north east) edge (u3) (v.north west) edge (u1)
(v.south east) edge (x3) (v.south west) edge (x1);
\end{scope}
\end{tikzpicture}}
\quad
&\mapsto
\quad
\scalebox{.85}{
 \begin{tikzpicture}
  \node[draw, ellipse, minimum width=1cm, minimum height=.5cm] (v) at (0,0) {$\cdots$};
  \node (u1) at (-.5,1) {$1$};
  \node (u2) at (0,1) {$\cdots$};
  \node (u3) at (.5,1) {};
  \node (x1) at (0,-1.7) {$\epsilon$};
  \node[int] (ww) at (0,-1) {};
  \draw  
  (v.north east) edge (u3) (v.north west) edge (u1)
  (v.south east) edge (ww) (v.south west) edge (ww)
  (ww) edge (x1);
  \begin{scope}[xshift=1.6cm]
    \node[draw, ellipse, minimum width=1cm, minimum height=.5cm] (v) at (0,0) {$\cdots$};
    \node (w1) at (1.5,.3) {$\omega$};
    \node (w2) at (1.5,-.3) {$\omega$};
    \node (u1) at (-.5,1) {};
    \node (u2) at (0,1) {$\cdots$};
    \node (u3) at (.5,1) {$n$};
    \node (x1) at (-.5,-1) {$\epsilon$};
    \node (x2) at (0,-1) {$\cdots$};
    \node (x3) at (.5,-1) {$\epsilon$};
    \draw (v) edge (w1) (v) edge (w2) 
    (v.north east) edge (u3) (v.north west) edge (u1)
    (v.south east) edge (x3) (v.south west) edge (x1);
    \end{scope}
   \end{tikzpicture}}\, .
\end{align*}
This is map is injective; it has a one-sided inverse obtained by removing both new vertices and the edge between them, and adding an external vertex labeled $\epsilon$ to each of the dangling edges.
This map is also surjective, since the unique $\epsilon$-decoration in an $\epsilon$-component of a graph in $W_1'$ must be adjacent to an internal vertex.
(If it was adjacent to an $j$-decorated leg the graph would not be in $W$, and $(\epsilon,\omega)$-edges are forbidden in $X_{g,n}$.)
\end{proof}

We note that in the above proof the presence of the part of the differential $\delta_{join}^\omega$ played no role. The same proof also shows the following auxiliary result, which we will use in the proof of Proposition~\ref{prop:Hg inclusion qiso}. 

\begin{lemma}\label{lem:X Xprime delta eps}
The inclusion $(\Xast_{g},\delta_{join}^\epsilon+ \delta_{split}) \to (X_{g},\delta_{join}^\epsilon+ \delta_{split})$ is a quasi-isomorphism of dg vector spaces.
\end{lemma}

\subsection{A simplifying involution}
We now construct an involution of $\Xast_{g,n}$ that simplifies the differential.  This involution is given by reattaching subsets of the $\epsilon$-decorated external vertices in all possible ways and motivates the introduction of the subcomplex $\Xast_{g,n}$; see Remark~\ref{rem:involutionwd}. 

Let $\Gamma$ be a generator for $\Xast_{g,n}$. We write $R_S(\Gamma)$ for the sum of all graphs obtained by reattaching a subset $S$ of the $\epsilon$-decorated external vertices of $\Gamma$ to internal vertices, in all possible ways without forming loop edges: 
\[
\begin{tikzpicture}[baseline=-.65ex,
decoration={brace,mirror,amplitude=3}]
\node[draw, circle] (v) at (0,.3) {$\Gamma$};
\draw (v) edge +(-.5,-.6) edge +(-.25,-.6)  edge +(0,-.6) edge +(.25,-.6) edge +(.5,-.6);  
\draw [decorate] (-.5,-.5) -- (-.25,-.5) 
node [pos=0.5,anchor=north,yshift=-0.1cm] {$S$}; 
\end{tikzpicture}
\quad
\mapsto
\quad
R_S(\Gamma)=
\sum
\begin{tikzpicture}[baseline=-.65ex]
\node[draw, circle] (v) at (0,.3) {$\Gamma$};
\node (xx) at (0,0.9) {};
\path[overlay] (v) edge[out=-135,in=135, loop] coordinate[midway](X) (v) edge[out=-115,in=115, loop] (v) ;
\draw (v) edge +(0,-.6) edge +(.25,-.6) edge +(.5,-.6);
\path (X);  
\end{tikzpicture}\, .
\]
The structural edges of each graph in $R_S(\Gamma)$ are in bijection with those of $\Gamma$, and we keep the given ordering.

\begin{rem} \label{rem:involutionwd}
The restriction on the union of all $\epsilon$-components of graphs in $\Xast_{g,n}$ (Definition~\ref{def:Xast}) guarantees that the reattachment operation does not produce connected components without any $\epsilon$- or $\omega$-decoration. For this reason, $R_S$ is well-defined on $\Xast_{g,n}$. It is not well-defined on $X_{g,n}$.
\end{rem}

We now consider the graded endomorphism $\XPhi$ of $\Xast_{g,n}$, given by 
\[
  \XPhi(\Gamma) = (-1)^{\# \epsilon} \sum_S R_S(\Gamma),
\]
where the sum runs over all subsets of the set of $\epsilon$-decorations, and $\#\epsilon$ is the number of $\epsilon$-decorations in $\Gamma$.

\begin{lemma}
The map $\XPhi$ is an involution, i.e., it is invertible and $\XPhi^{-1} = \XPhi$.
\end{lemma}

\begin{proof}
Any graph appearing in $\XPhi \circ \XPhi \, (\Gamma)$ is obtained by reattaching some subset $S$ of the $\epsilon$-decorated legs to internal vertices. Each such graph appears $2^{|S|}$ different ways, with signs that cancel unless $S = \emptyset$.
\end{proof}

For $\Gamma\in \Xast_{g,n}$ a graph, $S$ a non-empty subset of its $\epsilon$-decorated external vertices and $j=1,\dots,n$, let $R_S^j (\Gamma)$
be the graph obtained by reattaching the external vertices in $S$ to the midpoint of the edge that contains the $j$ decoration.
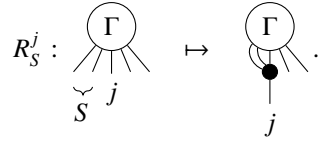
\[
  R_S^j : 
\begin{tikzpicture}[baseline=-.65ex,decoration={brace,mirror,amplitude=3}]
  \node[draw, circle] (v) at (0,.3) {$\Gamma$};
  \node (vj) at (0,-.6) {$j$};
\draw (v) edge +(-.5,-.6) edge +(-.25,-.6)  edge (vj) edge +(.25,-.6) edge +(.5,-.6);  
\draw [decorate] (-.5,-.5) -- (-.25,-.5) 
node [pos=0.5,anchor=north,yshift=-0.1cm] {$S$}; 
\end{tikzpicture}
\quad
\mapsto
\quad
\begin{tikzpicture}[baseline=-.65ex]
  \node[draw, circle] (v) at (0,.3) {$\Gamma$};
  \node[int] (w) at (0,-.3) {};
  \node (vj) at (0,-1) {$j$};
\draw (v) edge[out=-135, in=160] (w)  (v) edge[bend right] (w) edge (w) edge +(.25,-.6) edge +(.5,-.6)
(w) edge (vj) ;  
\end{tikzpicture}
\, .
\]
This operation creates precisely one new structural edge, which we take to be first in the ordering, preserving the relative ordering of the remaining edges.
Note that the operation also makes sense if applied not to a numbered leg, but an $\omega$-labeled leg, and we shall denote the sum of such operations applied to the two $\omega$-legs by $R_S^\omega$.

We also define 
\[
  \tilde \delta_{join}^j(\Gamma) 
  =
  \sum_{|S|\geq 1} |\,S| \cdot  R_S^j(\Gamma),
\]
where the sum runs over all non-empty subsets of the set of $\epsilon$-legs of $\Gamma$.
We then have the following result:

\begin{prop}\label{thm:WHGC simpl 2}
Let
\[
\tilde \delta := \delta_{split} + \delta_{join}^\epsilon - \sum_{j=1}^n \tilde \delta_{join}^j
\]
Then $\tilde \delta = \XPhi \circ \delta \circ \XPhi$.
In particular, $\XPhi$ defines an isomorphism of dg vector spaces 
\begin{equation*}
  \XPhi \colon (\Xast_{g,n}, \tilde \delta) \to (\Xast_{g,n}, \delta).
\end{equation*}
\end{prop}
\begin{proof}
  The proof is similar to that of \cite[Lemma 3.6]{TWspherical}; the argument there goes through essentially unchanged, even though the complexes $\Xast_{g,n}$ we consider here are different. We give only a condensed sketch of the proof. 
  
  We introduce the following notation.   For $S$ a subset of the $\epsilon$-legs in some graph $\Gamma\in \Xast_{g,n}$ we denote by $R_S^\epsilon(\Gamma)$ the graph obtained by reattaching the legs in $S$ to a new $\epsilon$-decorated vertex.
  \[
  R_S^\epsilon \colon  
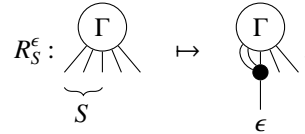
\begin{tikzpicture}[baseline=-.65ex,decoration={brace,mirror,amplitude=3}]
  \node[draw, circle] (v) at (0,.3) {$\Gamma$};
\draw (v) edge +(-.5,-.6) edge +(-.25,-.6)  edge +(0,-.6) edge +(.25,-.6) edge +(.5,-.6);  
\draw [decorate] (-.5,-.5) -- (0,-.5) 
node [pos=0.5,anchor=north,yshift=-0.1cm] {$S$}; 
\end{tikzpicture}
\quad
\mapsto
\quad
\begin{tikzpicture}[baseline=-.65ex]
  \node[draw, circle] (v) at (0,.3) {$\Gamma$};
  \node[int] (w) at (0,-.3) {};
  \node (vj) at (0,-1) {$\epsilon$};
\draw (v) edge[out=-135, in=160] (w)  (v) edge[bend right] (w) edge (w) edge +(.25,-.6) edge +(.5,-.6)
(w) edge (vj) ;  
\end{tikzpicture}
\, .
\]
Furthermore, we denote by $R'_S(\Gamma)$ the sum of graphs obtained from $R_S^\epsilon(\Gamma)$ by attaching the newly formed $\epsilon$-leg to an internal vertex of $\Gamma$. 
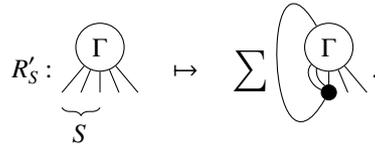
\[
  R'_S \colon  
\begin{tikzpicture}[baseline=-.65ex,decoration={brace,mirror,amplitude=3}]
  \node[draw, circle] (v) at (0,.3) {$\Gamma$};
\draw (v) edge +(-.5,-.6) edge +(-.25,-.6)  edge +(0,-.6) edge +(.25,-.6) edge +(.5,-.6);  
\draw [decorate] (-.5,-.5) -- (0,-.5) 
node [pos=0.5,anchor=north,yshift=-0.1cm] {$S$}; 
\end{tikzpicture}
\quad
\mapsto
\quad
\sum\,
\begin{tikzpicture}[baseline=-.65ex]
  \node[draw, circle] (v) at (0,.3) {$\Gamma$};
  \node[int] (w) at (0,-.3) {};
\draw (v) edge[out=-135, in=160] (w)  ;
\draw (v) edge[bend right] (w) edge (w) edge +(.25,-.6) edge +(.5,-.6) ;
\draw[overlay] (w) edge[out=-120, in=120, looseness=4] coordinate[midway] (X) coordinate[near end] (XX) coordinate[near start] (XXX) (v); 
\path (X) (XX) (XXX); 
\end{tikzpicture}
\, .
\]
Using the fact that $\XPhi^{-1} = \XPhi$, one computes that 
\begin{equation*}
(\XPhi \circ \delta_{split}\circ\XPhi )(\Gamma)
=
\delta_{split} (\Gamma)
+ \sum_{|S|\geq 2}  |S| \cdot R_S^\epsilon (\Gamma) 
+ \sum_{|S|\geq 2}  (|S|-1) \cdot R'_S (\Gamma)
+ \sum_{|S|\geq 1} \sum_{j\in \{1,\dots,n,\omega \} } R_S^j (\Gamma),
\end{equation*}
Here $S$ runs over subsets of the set of $\epsilon$-legs.
Similarly, one may compute
\begin{equation*}
(\XPhi \circ \delta_{join}^\epsilon\circ\XPhi )(\Gamma)
=
 \sum_{ 
 |S|\geq 2}  (1-|S|) \cdot R_S^\epsilon (\Gamma) 
+ \sum_{ 
|S|\geq 2}  (1-|S|) \cdot R'_S (\Gamma).
\end{equation*}
Together we find that 
\begin{align*}
  (\XPhi \circ (\delta_{split}+ \delta_{join}^\epsilon)\circ\XPhi )
  (\Gamma)
   &=
   (\delta_{split}
   +
   \delta_{join}^\epsilon
   +
   \delta_{join}^\omega) \Gamma
    + 
    \sum_{ 
    |S|\geq 1} \sum_{j\in \{1,\dots,n\} } R_S^j (\Gamma)
    \\&=
    \delta (\Gamma) + \sum_{ 
    |S|\geq 1} \sum_{j\in \{1,\dots,n\} } R_S^j (\Gamma).
\end{align*}
Then one may finally compute that 
\begin{equation*}
 \Big( \XPhi\circ \sum_{|S| \geq 1} R_S^j \circ \XPhi \Big) \, (\Gamma)
 =
\sum_{ 
|S|\geq 1} |S| \cdot R_S^j (\Gamma)
=
\tilde \delta_{join}^j (\Gamma).
\end{equation*}
Putting the above computations together we find that 
\begin{align*}
\XPhi\circ \tilde \delta \circ \XPhi 
&=
\XPhi\circ \Big( \delta_{split}+\delta_{join}^\epsilon-\sum_j \tilde \delta_{join}^j \Big)\circ \XPhi 
=
\delta + \sum_j \delta_{join}^j -
\sum_j\delta_{join}^j = \delta,
\end{align*}
which proves the proposition.
\end{proof}

 \section{The weight 2 compactly supported cohomology of $\MM_g$}
 We now focus on the special case where $n=0$. Using the quasi-isomorphic subcomplex $\Xast_g \subset X_g$, with its simplified differential $\tilde \delta$, we shall see that the cohomology of $X_{g}$ can be fully expressed through the cohomology of the ordinary commutative graph complexes $G^{(g',n')}$ of Section \ref{sec:svert weight0}, with $n'=1, 2$ and $g'=g,g-1,g-2$. 
 
When $n = 0$, the differential $\tilde \delta$ in Proposition~\ref{thm:WHGC simpl 2} further simplifies to 
$\delta_{split}+ \delta_{join}^\epsilon$.
We define the subcomplex 
\begin{equation} \label{equ:Hg inclusion}
  H_{g} \subset (\Xast_{g},\delta_{split}+ \delta_{join}^\epsilon)
\end{equation}
spanned by graphs with either no $\epsilon$-decorations, or exactly two on an $(\epsilon,\epsilon)$-edge. Note that $\delta_{join}^\epsilon$ vanishes on $H_g$.

We also introduce the space 
\[
G^{(\bullet,n)} = \bigoplus_{g} G^{(g,n)},  
\]
that inherits an additional grading by the genus.
We also define the exterior product 
\[
\mathbb{W}_\bullet = \bigwedge^2 G^{(\bullet,1)}
\]
This comes with a grading inherited from the genus grading; let $\mathbb{W}_g$ be the homogeneous part of genus $g$.

\begin{lemma}\label{lem:Hg isomorphism}
The subspace $H_{g}$ 
is a subcomplex with respect to $\tilde \delta = \delta_{split}+ \delta_{join}^\epsilon$, and $(H_g, \tilde \delta)$ is isomorphic to
\[
\mathbb{W}_{g}[-3]
  \oplus 
  G^{(g-1,2)}_{as}[-3]
  \oplus
  \mathbb{W}_{g-1}[-4]
  \oplus 
  G^{(g-2,2)}_{as}[-4].
\]
Here the subscript $(-)_{as}$ refers to taking the antisymmetric part under the $\bbS_2$-action permuting the labels $1,2$.
\end{lemma}
\begin{proof}
  Graphs in $H_g$ have exactly two $\omega$-decorated legs, and every connected component has an $\omega$- or $\epsilon$-decorated leg.  There are four possibilities, with and without an $(\epsilon,\epsilon)$-edge, and with 1 or 2 $\omega$-components:
    \[
      \resizebox{0.91\hsize}{!}{$
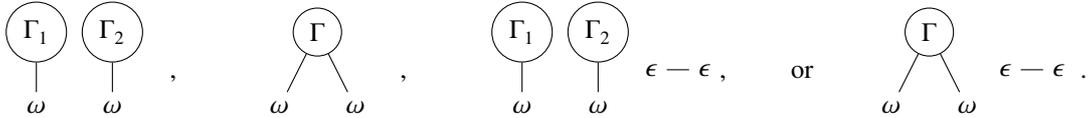
      \begin{tikzpicture}
        \node[draw, circle] (v1) at (-.5,.5) {$\Gamma_1$};
        \node[draw, circle] (v2) at (.5,.5) {$\Gamma_2$};
        \node (o1) at (-.5,-.5) {$\omega$};
        \node (o2) at (.5,-.5) {$\omega$};
        \draw (v1) edge (o1) (v2) edge (o2);
      \end{tikzpicture}
      \quad , \quad \quad \quad 
\begin{tikzpicture}
  \node[draw, circle] (v) at (0,.5) {$\Gamma$};
  \node (o1) at (-.5,-.5) {$\omega$};
  \node (o2) at (.5,-.5) {$\omega$};
  \draw (v) edge (o1) edge (o2);
\end{tikzpicture}
      \quad , \quad \quad \quad
      \begin{tikzpicture}
        \node[draw, circle] (v1) at (-.5,.5) {$\Gamma_1$};
        \node[draw, circle] (v2) at (.5,.5) {$\Gamma_2$};
        \node (o1) at (-.5,-.5) {$\omega$};
        \node (o2) at (.5,-.5) {$\omega$};
        \node (e1) at (1.2,0) {$\epsilon$};
        \node (e2) at (1.9,0) {$\epsilon$};
        \draw (e1) edge (e2);
        \draw (v1) edge (o1) (v2) edge (o2);
      \end{tikzpicture}
         , \quad  \quad \mbox{ or } \quad \quad 
    \begin{tikzpicture}
      \node[draw, circle] (v) at (0,.5) {$\Gamma$};
      \node (o1) at (-.5,-.5) {$\omega$};
      \node (o2) at (.5,-.5) {$\omega$};
      \node (e1) at (1,0) {$\epsilon$};
      \node (e2) at (1.7,0) {$\epsilon$};
      \draw (v) edge (o1) edge (o2);
      \draw (e1) edge (e2);
    \end{tikzpicture} \ .
    $}
    \]
The graded vector space $H_g$ decomposes into a direct sum of 4 subcomplexes accordingly; these subcomplexes are identified with  $\mathbb{W}_{g}[-3]$, $G^{(g-1,2)}_{as}[-3]$, $\mathbb{W}_{g-1}[-4]$, and $G^{(g-2,2)}_{as}[-4]$, respectively, since $\delta_{join}^\epsilon$ vanishes on $H_g$.
\end{proof}

 \begin{prop}\label{prop:Hg inclusion qiso}
 The inclusion \eqref{equ:Hg inclusion} is a quasi-isomorphism.
 \end{prop}
 \begin{proof}
 First note that one has a commutative diagram of morphisms of dg vector spaces
 \[
   \begin{tikzcd}
    (H_g,\delta_{split}) \ar[hookrightarrow]{r}{\text{\eqref{equ:Hg inclusion}}}
    \ar[hookrightarrow]{dr}
    & 
    (\Xast_{g}, \delta_{split}+\delta_{join}^\epsilon)
    \ar[hookrightarrow]{d}{\sim}
    \\
    & (X_{g}, \delta_{split}+\delta_{join}^\epsilon)
   \end{tikzcd}.
 \]
The vertical map is a quasi-isomorphism by Lemma~\ref{lem:X Xprime delta eps}.
It hence suffices to show that the diagonal inclusion
\[
  (H_g,\delta_{split})\to (X_{g}, \delta_{split}+\delta_{join}^\epsilon)
\]
is a quasi-isomorphism. 
Each complex splits, as in the proof of Proposition~\ref{prop:WHGC prime WHGC inclusion qiso}, 
into one piece generated by graphs that do not have any $(\epsilon,\epsilon)$-edge and a complementary piece spanned by those that do. So, we write $H'_g$ and $X'_g$ for the subcomplexes of $H_g$ and $X_g$, respectively, spanned by graphs that do not have any $(\epsilon,\epsilon)$-edge.  Similarly, we write $H''_g$ and $X''_g$ for the subcomplexes spanned by graphs with an $(\epsilon,\epsilon)$-edge.  We then have 
\begin{align*}
  H_g &\cong H_g'
  \oplus H_g''
  &
  X_g &\cong X_g' 
  \oplus X_g''.
\end{align*}
Note that $H_g'' \cong   H_{g-1}'[-1]$ and  $X_g'' \cong X_{g-1}'[-1]$, so 
it suffices to show that $H_g' \subset X'_g$ is a quasi-isomorphism.

We furthermore have a direct sum decomposition of complexes 
\[
  (X_g',  \delta_{split}+\delta_{join}^\epsilon)
  \cong 
  (H_g',\delta_{split})
  \oplus 
  (U, \delta_{split}+\delta_{join}^\epsilon),
\]
where $U\subset X_g'$ is the subcomplex spanned by graphs with at least one $\epsilon$-leg. 
It remains to show that $U$ is acyclic. 

We proceed as in Proposition \ref{prop:WHGC prime WHGC inclusion qiso} and decompose 
 \[
  U
  =  
 \begin{tikzcd}[column sep=.1cm]
 U_1 \ar[loop above]{} & 
 \oplus & 
 U_{\geq 2} \ar[loop above]{} \ar[bend right]{ll}[above]{D}
 \end{tikzcd},
 \]
 with the graded subspace $U_1\subset U$ (resp. $U_{\geq 2}\subset U$) being spanned by graphs with exactly $1$ (respectively $\geq 2$) $\epsilon$-legs.
The drawn arrows indicate various parts of the differential.
In particular, the part 
\[
  D : U_{\geq 2} \to U_1
\]
arising from $\delta_{join}^\epsilon$ joins all $\epsilon$-legs into one.

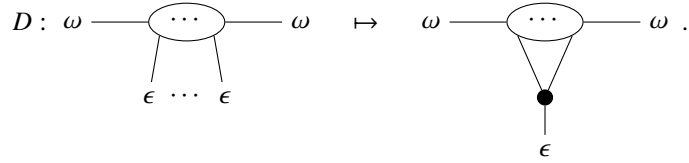
\begin{align*}
  D: 
  \begin{tikzpicture}
  \node[draw, ellipse, minimum width=1cm, minimum height=.5cm] (v) at (0,0) {$\cdots$};
  \node (w1) at (-1.5,0) {$\omega$};
  \node (w2) at (1.5,0) {$\omega$};
  \node (x1) at (-.5,-1) {$\epsilon$};
  \node (x2) at (0,-1) {$\cdots$};
  \node (x3) at (.5,-1) {$\epsilon$};
  \draw (v.west) edge (w1) (v.east) edge (w2) 
  (v.south east) edge (x3) (v.south west) edge (x1);
  \end{tikzpicture}
  \quad
  &\mapsto
  \quad
   \begin{tikzpicture}
    \node[draw, ellipse, minimum width=1cm, minimum height=.5cm] (v) at (0,0) {$\cdots$};
    \node (w1) at (-1.5,0) {$\omega$};
    \node (w2) at (1.5,0) {$\omega$};
    \node (x1) at (0,-1.7) {$\epsilon$};
    \node[int] (ww) at (0,-1) {};
    \draw (v.west) edge (w1) (v.east) edge (w2) 
    (v.south east) edge (ww) (v.south west) edge (ww)
    (ww) edge (x1);
   \end{tikzpicture}
    \, .
\end{align*}

This part is injective, since the operation can be undone by removing the two newly added vertices, and adding back the $\epsilon$-decorations on legs.
By Lemma \ref{lem:acyclicity} we hence conclude that the projection 
\[
U \to \coker D  
\]
is a quasi-isomorphism.
The cokernel of $D$ is spanned by graphs in $U_1$ such that the unique $\epsilon$-leg is connected to a vertex $v$ that is connected to an $\omega$-leg. (Recall that we forbade $(\epsilon,\omega)$-edges in the definition of $X_{g,n}$.) 
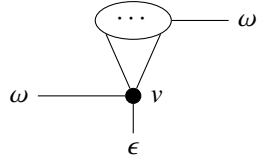
\[
\begin{tikzpicture}
  \node[draw, ellipse, minimum width=1cm, minimum height=.5cm] (v) at (0,0) {$\cdots$};
  \node (w1) at (-1.5,-1) {$\omega$};
  \node (w2) at (1.5,0) {$\omega$};
  \node (x1) at (0,-1.7) {$\epsilon$};
  \node[int,label=0:{$v$}] (ww) at (0,-1) {};
  \draw (ww) edge (w1) (v.east) edge (w2) 
  (v.south east) edge (ww) (v.south west) edge (ww)
  (ww) edge (x1);
\end{tikzpicture}
\]
However, we can continue in the same fashion and filter the cokernel of $D$ as 
\[
\coker (D) = 
\begin{tikzcd}[column sep=.1cm]
 U_1' \ar[loop above]{}& 
 \oplus & 
 U_{2}' \ar[bend right]{ll}[above]{D'}
 \ar[loop above]{}
 \end{tikzcd},
\]
with $U_1'$ (resp. $U_{2}'$) being spanned by graphs for which the vertex $v$ has valence $3$ (resp. $\geq 4$).
The part of the differential 
\[
  D': U_{2}' \to  U_1'  
\]
is injective. Hence, applying Lemma \ref{lem:acyclicity} again, we find that the projection 
\[
  \coker (D) \to \coker (D')
\]
is a quasi-isomorphism.
The cokernel of $D'$ is spanned by graphs in $U_1'$ such that the unique vertex $w$ neighboring $v$ has an $\omega$-hair attached.
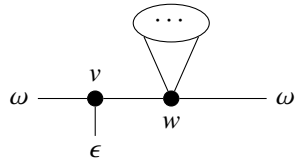
\[
\begin{tikzpicture}
  \node[draw, ellipse, minimum width=1cm, minimum height=.5cm] (v) at (0,0) {$\cdots$};
  \node (w1) at (-2,-1) {$\omega$};
  \node (w2) at (1.5,-1) {$\omega$};
  \node (x1) at (-1,-1.7) {$\epsilon$};
  \node[int,label=90:{$v$}] (ww) at (-1,-1) {};
  \node[int,label=-90:{$w$}] (www) at (0,-1) {};
  \draw (ww) edge (w1) (www) edge (w2) 
  (v.south east) edge (www) (v.south west) edge (www)
  (ww) edge (x1)
  (www) edge (ww);
\end{tikzpicture}
\]
Repeating the argument once more we split 
\[
\coker (D') = 
\begin{tikzcd}[column sep=.1cm]
 U_1'' & 
 \oplus & 
 U_{2}'' \ar[bend right]{ll}[above]{D''}
 \end{tikzcd},
\]
with $U_1''$ (resp. $U_{2}''$) being spanned by graphs for which the vertex $w$ has valency $3$ (resp. $\geq 4$).
Now the piece of the differential 
\[
  D'': U_{2}'' \to  U_1''  
\]
is a bijection, since there are only two $\omega$-legs.
Hence by Lemma \ref{lem:acyclicity} we have that $\coker (D')$ is acyclic, and hence so are $\coker (D)$ and $U$. This proves the proposition.
\end{proof}

\begin{proof}[Proof of Theorem \ref{cor:n0}]
  By Propositions \ref{prop:WHGC prime WHGC inclusion qiso}, \ref{thm:WHGC simpl 2}, and \ref{prop:Hg inclusion qiso}, the composition 
  \[
    H_g
    \hookrightarrow 
    (\Xast_{g},\delta_{split}+\delta_{join}^\epsilon)
    \xrightarrow{\Phi}
    (\Xast_{g},\delta_{split}+\delta_{join})
    \hookrightarrow 
    X_{g}
  \]
  is a quasi-isomorphism. The cohomology of the right-hand side is $\gr^2 H_c^\bullet(\MM_{g})$. The cohomology of $H_g$ is expressed via Lemma \ref{lem:Hg isomorphism} through the cohomology of the complexes $G^{(g',n')}$ that compute $\gr^0 H_c^\bullet(\MM_{g',n'})$ by Theorem~\ref{thm:weight0}, and Theorem \ref{cor:n0} follows.
\end{proof}
 
\section{Direct summands with marked points} \label{sec:classes n1}
Let $\Gamma$ be a generator for $\Xast_{g,n}$, and let us now assume $n \geq 1$. Recall that each such generator has some number of $\epsilon$-decorations and precisely two $\omega$-decorations. The remaining external vertices are decorated by a bijection to $\{1, \ldots, n\}$.  Each component has at least one decoration from $\{\epsilon, \omega\}$.  

\begin{defi}
A connected component of $\Gamma$ is \emph{isolated} if it contains no decorations from $\{1, \ldots, n, \epsilon\}$.
\end{defi}

\noindent In other words, a component is isolated if all of its external vertices are decorated with $\omega$.  Components that do have a decoration from $\{1, \ldots, n, \epsilon\}$ are \emph{non-isolated}.  We decompose $\Xast_{g,n}$ according to the number of $\omega$ decorations that are contained non-isolated components.

\begin{defi}
Let $H_{g,n}$, $J_{g,n}$, and $K_{g,n}$ be the graded subspaces of $\Xast_{g,n}$ spanned by graphs with $0$, $1$, or $2$ of their $\omega$-decorations contained in non-isolated components, respectively.
\end{defi}  
\noindent No term of the differential $\tilde \delta$ can make an isolated connected component into a non-isolated one or vice versa, so 
we have a direct sum decomposition of dg vector spaces 
\begin{equation}\label{eq:Xast direct sum}
  (\Xast_{g,n},\tilde \delta)
  = 
  H_{g,n}\oplus J_{g,n} \oplus K_{g,n}.
\end{equation}

\noindent We will study the cohomology of the summands separately.  Unfortunately, we have nothing to say about $H(K_{g,n})$.  

\subsection{The cohomology of $H_{g,n}$}

The summand $H_{g,n}$ contributes to the cohomology of $\Xast_{g,n}$ only when $n = 1$, and that contribution is well-understood.
 
\begin{lemma}\label{lem:Hgn}
We have $H_{g,1} \cong H_g$ and $H_{g,n} = 0$ for $n \geq 2$. 
\end{lemma}
\begin{proof}
By the definition of $\Xast_{g,n}$, the non-isolated components that do not contain $\omega$-decorations consist of either an $(\epsilon, \epsilon)$-edge, an $(\epsilon,j)$-edge, or one of each. If $n \geq 2$, then none of the generators for $\Xast_{g,n}$ are in $H_{g,n}$.  Moreover, the generators for $H_{g,1}$ are precisely the graphs of the form 
\[
  \begin{tikzpicture}
    \node[circle, draw, minimum width=5mm] (vv) at (-1,0) {$\Gamma$};
      \node(v) at (0,0) {$\epsilon$};
      \node(w) at (.7,0) {$1$};
      \draw (v) edge (w);
  \end{tikzpicture},
\]
where $\Gamma$ is a (possibly disconnected) generator for $H_g$, and the lemma follows.

\end{proof}

\noindent We remark that the identification $H_{g,1} \cong H_g$ of the lemma is induced by the pullback under the forgetful map $\bMM_{g,1}\to \bMM_g$, see Appendix \ref{sec:pullback map} below.

\noindent Combining Lemmas \ref{lem:Hg isomorphism} and \ref{lem:Hgn}, we can express the cohomology of $H_{g,n}$ through the graph cohomology $H(G^{(g',n')})$, or equivalently through $W_0H_c(\MM_{g',n'})$.

\subsection{The graphs in $J_{g,n}$}

Each generator for $J_{g,n}$ has exactly one isolated component, and can be drawn as: 
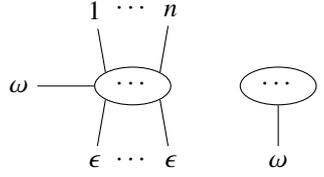
\[
  \begin{tikzpicture}
    \node[draw, ellipse, minimum width=1cm, minimum height=.5cm] (v) at (0,0) {$\cdots$};
    \node (w1) at (-1.5,0) {$\omega$};
    \node (u1) at (-.5,1) {$1$};
    \node (u2) at (0,1) {$\cdots$};
    \node (u3) at (.5,1) {$n$};
    \node (x1) at (-.5,-1) {$\epsilon$};
    \node (x2) at (0,-1) {$\cdots$};
    \node (x3) at (.5,-1) {$\epsilon$};
    \draw (v.west) edge (w1) 
    (v.north east) edge (u3) (v.north west) edge (u1)
    (v.south east) edge (x3) (v.south west) edge (x1);
  \end{tikzpicture}
\quad \quad 
  \begin{tikzpicture}
    \node[draw, ellipse, minimum width=1cm, minimum height=.5cm] (v) at (0,0) {$\cdots$};
    \node (w1) at (0,-1) {$\omega$};
    \draw (v) edge (w1);
  \end{tikzpicture} \ .
\]
The differential acts by $\delta_{split}$ on the isolated component, so $J_{g,n}$ decomposes as a direct sum of subcomplexes determined by the genus of the isolated component.  The subcomplex where the non-isolated part has genus $h$ can be written as a tensor product $\Jast_{h,n} \otimes G^{(h',1)}[-1]$, where $h + h' = g$ and $\Jast_{h,n}$ is a graph complex perfectly analogous to $\Xast_{h,n}$, except that each generator has exactly one $\omega$ decoration instead of two, and this $\omega$-decoration must be part of a non-isolated component.\footnote{The latter condition is automatically satisfied if $n\geq 2$.}

We then have a decomposition
\begin{equation}\label{equ:Jgn decomp}
  J_{g,n}
  =
  \bigoplus_{h+h'=g}
  \Jast_{h,n} \otimes G^{(h',1)}[-1].
\end{equation}
The degree shift on the factor $G^{(h',1)}$ is due to the additional structural edge that contains the $\omega$ decoration.
The degree of a generator for $\Jast_{h,n}$ is the number of structural edges plus one, just as for $X_{g',n}$. The cohomology of $\Jast_{h,n}$ is difficult to evaluate when $h$ is large; we leave its study for $h\geq 2$ to future work. Theorems~\ref{cor:Psi injection} and \ref{cor:n2} use only the cohomology of $\Jast_{0,n}$ and $\Jast_{1,n}$, which we now describe.

\subsection{The cohomology of $\Jast_{0,n}$} We now explain how the cohomology of $\Jast_{0,n}$ is closely related to $W_0H^\bullet_c(\MM_{0,n})$ and $W_0H^\bullet_c(\MM_{0,n+1})$. Let $V_1 \subset J^*_{0,n}$ be the subcomplex generated by graphs with no $\epsilon$ decorations. Each generator for $V_1$ is a rooted tree with $n$ leaves labeled $\{1, \ldots, n\}$ and a root labeled $\omega$.  Let $V_0$ be the complementary graded subspace generated by graphs with an $\epsilon$-decoration. Since the genus is zero, there are no $(\epsilon,\epsilon)$-edges, and no connected component can have both an $\epsilon$-decoration and also an $\omega$-decoration.  So each generator for $V_0$ has two connected components, an $(\epsilon, j)$-edge and an $\omega$-rooted tree with $(n-1)$ leaves labeled $\{1, \ldots, \hat \jmath, \ldots n\}$. As an immediate consequence, we see that  
\begin{equation*}
\Jast_{0,0}=0 \quad \quad \mbox{ and }  \quad \quad \Jast_{0,1} \cong \Q[ -1], \quad  \mbox {spanned by the graph} \quad   \begin{tikzpicture}
    \node (v) at (0,0) {$1$};
    \node (w) at (0.7,0) {$\omega$};
    \draw (v) edge (w);
  \end{tikzpicture} .
\end{equation*}

We now compute the cohomology of $\Jast_{0,n}$ for $n \geq 2$.

\begin{prop} \label{lem:J0np}
For $n\geq 2$, $H(\Jast_{0,n}, \delta)$ has dimension $(n-2)!$ and is concentrated in degree $n-1$.
More precisely, there is an isomorphism of $\bbS_n$-modules 
\[
H^{n-1}(\Jast_{0,n}, \delta) 
\cong H_c^{n-3}(\MM_{0,n})
\cong \Lie((n))\otimes \sgn_n
\]
with the $n$-ary part of the cyclic Lie operad $\Lie((n))$.
\end{prop}

\begin{proof}
The differential on $\Jast_{0,n}$ splits into pieces as 
  \[
  \begin{tikzcd}
    V_0 \ar{r}{\delta_{join}} \ar[loop above]{} & V_1 \ar[loop above]{}
  \end{tikzcd}\, ,
  \]
where the internal arrows on $V_0$ and $V_1$ are $\delta_{split}$.  Note that $V_1 \cong G^{(0,n+1)}[-2]$, and $(V_0,\delta_{split}) \cong \bigoplus_{j=1}^n G^{(0,n)}[-2]$. Therefore $H(V_1) \cong W_0H_c(\MM_{0,n+1})[-2]$ has dimension $(n-1)!$, supported in degree $n$, and $H(V_0,\delta_{split})$ has dimension $n(n-2)!$, supported in degree $n-1$. By Lemma~\ref{lem:acyclicity}\eqref{it:surj}, the first statement of Proposition \ref{lem:J0np} follows if we can show that $$\delta_{join} \colon H(V_0, \delta_{split}) \to H(V_1)$$ is surjective. For $1 \leq j \leq n$, let $D_j \subset \MM_{0,n+1}$ be the locally closed divisor parametrizing 1-nodal curves where $j$ and $\omega$ collide.  Then $D_j \cong \MM_{0,n}$, the union $\MM_{0,n+1} \sqcup \Big( \bigsqcup_j D_j \Big)$ is open in $\bMM_{0,n+1}$, and $\bigsqcup_j D_j$ is closed in this union.  Then $\delta_{join}$ is naturally identified with the excision coboundary map 
\[
\delta \colon \bigoplus_j W_0H^\bullet_c(D_j) \to W_0H^{\bullet + 1}_c(\MM_{0,n+1}).
\]
We must show that $\delta$ is surjective. We will give a short algebraic proof of this surjectivity.

To this end, it will be convenient to work with the $\bbS_n$-modules
\[
  \tJast_{0,n} := \Jast_{0,n}\otimes \sgn_n.
\]
Tensoring with $\sgn_n$ can be nicely incorporated in the sign conventions on graph complexes; we take generators for $\tJast_{0,n}$ to be graphs with a total ordering of all edges (not just the structural edges), and impose the relation that permuting the edges induces multiplication by the sign of the permutation. Decompose $\tJast_{0,n} \cong \tilde V_0 \oplus \tilde V_1$, as above.

Then we can identify the cohomology groups with subspaces of free Lie algebras 
\begin{align*}
H(\tilde V_0, \tilde \delta_{split})
&= \bigoplus_{j=1}^n \Lie(x_1,\dots,\hat x_j,\dots,x_n)[-n]
&
H(\tilde V_1) 
&= \Lie(x_1,\dots,x_n)[-n-1] 
\end{align*}
with $\Lie(x_1,\dots,x_n)$ the part of the free Lie algebra with generators $x_1,\dots,x_n$ for which each generator appears exactly once.  Here the identification between (trivalent) trees and Lie words is such that the root of our tree is taken to be $\omega$, every vertex is replaced by one Lie bracket, and the $j$-labeled leaf is replaced by $x_j$, e.g. 
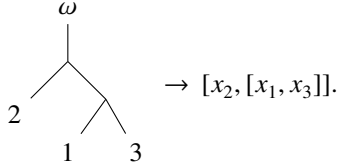
\[
\begin{tikzpicture}
  \node (w) at (0,1) {$\omega$};
  \coordinate (v1) at (0,.3);
  \coordinate (v2) at (0.5,-.2) {};
  \node (e1) at (-.7,-.4) {$2$};
  \node (e2) at (0,-.9) {$1$};
  \node (e3) at (.9,-.9) {$3$};
  \draw (v1) edge (w) edge (v2) edge (e1)
    (v2) edge (e2) edge (e3);
\end{tikzpicture}
\, 
\to 
\,
[x_2,[x_1,x_3]].
\]
The differential  is then given by 
\begin{equation}\label{equ:deltajoin F expl}
 (F_1,\dots,F_n) \to \sum_{j=1}^n [x_j,F_j],
\end{equation}
which is surjective, as required.

For the second statement of Proposition \ref{lem:J0np} we take for granted the well-known fact that $H_c^{n-3}(\MM_{0,n})\cong \Lie((n))\otimes \sgn_n$.
It then remains to check that the kernel of \eqref{equ:deltajoin F expl} can be identified with $\Lie((n))$ as an $\bbS_n$-module.
We may do this by providng an injective map 
\[
\iota :  \Lie((n))
\to 
\bigoplus_{j=1}^n \Lie(x_1,\dots,\hat x_j,\dots,x_n)
\]
whose image is in the kernel of \eqref{equ:deltajoin F expl}.
To construct $\iota$, we use the following notation 
for elements of $\Lie((n))$.
We have that $\Lie((n))\cong \Lie(x_1,\dots,x_{n-1})$
and we write 
\[
 f(x_1,\dots,x_n) := \langle f_n(x_1,\dots,x_{n-1}),x_n\rangle \in \Lie((n))
\]
for the element corresponding to a Lie expression $f_n(x_1,\dots,x_{n-1})\in \Lie(x_1,\dots,x_{n-1})$.
The $\bbS_n$-action on $\Lie((n))$ is defined by permuting variables, considering $\langle-,-\rangle$ formally as an invariant inner product.
In particular, for each $j$ we have Lie expressions 
$f_j(x_1,\dots,\hat x_j ,\dots, x_{n})\in \Lie(x_1,\dots,\hat x_j,\dots,x_{n})$ such that 
\[
f(x_1,\dots,x_n) = \langle f_j(x_1,\dots,\hat x_j, \dots,x_{n}),x_j\rangle.
\]
We then define the map $\iota$ such that 
\[
\iota(f) := (f_1,f_2,\dots,f_n)\in 
\bigoplus_{j=1}^n \Lie(x_1,\dots,\hat x_j,\dots,x_n).  
\]
It is clear that $\iota$ is an injection since $f$ may be recovered from any of the $f_j$.

We are hence left with checking that the composition of $\iota$ and \eqref{equ:deltajoin F expl} is zero.
To this end, note that in $\Lie((n+1))$ we have the equality
\begin{align*}
\langle [x_j, f_j(x_1,\dots,\hat x_j,\dots,x_n)], x_{n+1} \rangle 
&=-
\langle  f_j(x_1,\dots,\hat x_j,\dots,x_n), [x_j,x_{n+1}] \rangle 
\\&=f(x_1,\dots, [x_{n+1},x_j], \dots, x_n)
\\&=
\langle f_n(x_1,\dots,[x_{n+1},x_j],\dots,x_{n-1}), x_{n} \rangle,   
\end{align*}
for $j=1,\dots,n-1$, using invariance of our formal inner product $\langle-,-\rangle$ and the definition of the $f_j$.
Thus in $\Lie((n+1))$ we have that
\begin{align*}
    &\langle \sum_{j=1}^n [x_j, f_j(x_1,\dots,\hat x_j,\dots,x_n)], x_{n+1} \rangle 
    \\&=
    \langle 
    -[x_{n+1}, f_n(x_1,\dots,x_{n-1})]
    +
    \sum_{j=1}^{n-1}f_n(x_1,\dots,[x_{n+1},x_j],\dots,x_{n-1}) 
    ,x_n \rangle 
    =0
\end{align*}
by the Jacobi identity.
Hence $\sum_{j=1}^n [x_j, f_j(x_1,\dots,\hat x_j,\dots,x_n)]=0$ as desired and the second statement of Proposition \ref{lem:J0np} follows.

\end{proof}

The second statement of Proposition \ref{lem:J0np} was kindly pointed out to us by the anonymous referee, whom we thank for her or his contribution.

\begin{ex}\label{ex:J02}
  In particular, Proposition \ref{lem:J0np} states that $H(\Jast_{0,2})$ is $1$-dimensional, concentrated in degree $1$.
  Tracing the proof, one can see that a representative is given by the cocycle
  \[ 
  \begin{tikzpicture}
    \node (a1) at (-1, .4) {$1$};
    \node (b1) at (-2.2, .4) {$\omega$};
    \node (a2) at (-1, -.4) {$2$};
    \node (b2) at (-2.2, -.4) {$\epsilon$};
    \draw (a1) edge (b1) (a2) edge (b2);
  \end{tikzpicture} 
  \, - \, 
  \begin{tikzpicture}
    \node (a1) at (-1, .4) {$2$};
    \node (b1) at (-2.2, .4) {$\omega$};
    \node (a2) at (-1, -.4) {$1$};
    \node (b2) at (-2.2, -.4) {$\epsilon$};
    \draw (a1) edge (b1) (a2) edge (b2);
  \end{tikzpicture}\, .
  \]
\end{ex}

\subsection{The cohomology of $\Jast_{1,n}$} \label{sec:J1n}

Each generator for $\Jast_{1,n}$ has one of six combinatorial types: with or without an $(\epsilon,j)$-edge, and with 0, 1, or 2 other $\epsilon$-decorations.  The following diagram depicts one generator of each type (for varying values of $n$). We denote the corresponding graded subspaces of $\Jast_{1,n}$ by $V_{i,k}$, as shown. The notation is chosen so that a generator for $V_{i,k}$ has $1-k$ edges of type $(\epsilon,j)$ and $2-i$ $\epsilon$-decorations on other componentss.
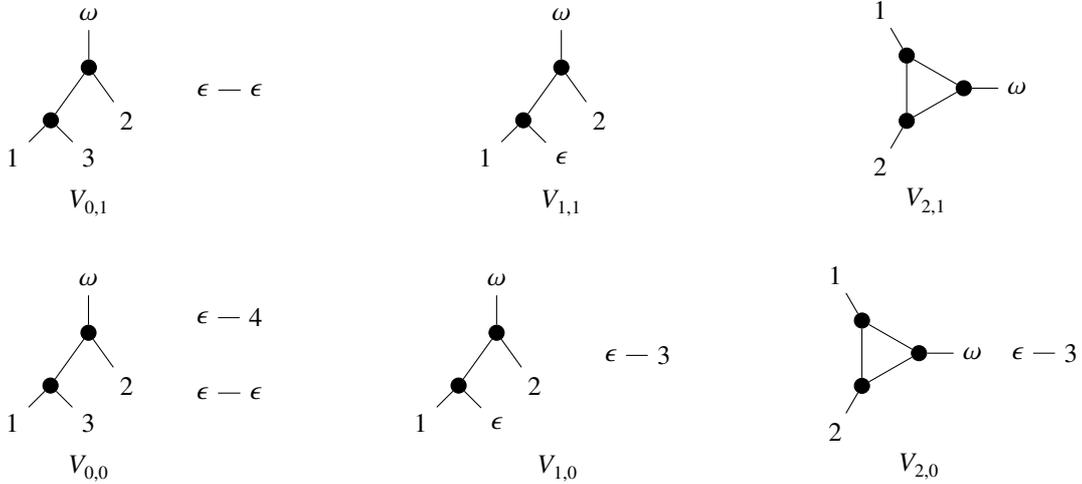
\[ 
  \resizebox{0.91\hsize}{!}{$
\arraycolsep=25 pt 
\begin{array}{ccc}
  \begin{tikzpicture}
    \node[int] (v1) at (0,.8) {};
    \node[] (w1) at (0,1.5) {$\omega$};
    \node[int] (v2) at (-.5,.1) {};
    \node[] (w2) at (-1,-.4) {$1$};
    \node[] (e) at (0,-.4) {$3$};
    \node[] (w3) at (.5,.1) {$2$};
    \draw (v1) edge (v2) edge (w3) edge (w1)
    (v2) edge (e) edge (w2);
    \node (e) at (1.5,.5) {$\epsilon$};
    \node (w4) at (2.2,.5) {$\epsilon$};
    \draw (e) edge (w4);
    \node[] at (0,-1) {$V_{0,1}$};
  \end{tikzpicture}
  &
  \begin{tikzpicture}
    \node[int] (v1) at (0,.8) {};
    \node[] (w1) at (0,1.5) {$\omega$};
    \node[int] (v2) at (-.5,.1) {};
    \node[] (w2) at (-1,-.4) {$1$};
    \node[] (e) at (0,-.4) {$\epsilon$};
    \node[] (w3) at (.5,.1) {$2$};
    \draw (v1) edge (v2) edge (w3) edge (w1)
    (v2) edge (e) edge (w2);
        \node[] at (0,-1) {$V_{1,1}$};
  \end{tikzpicture}
&
\raisebox{15 pt}{
  \begin{tikzpicture}
    \node[int] (v1) at (0:.5) {};
    \node[] (w1) at (0:1.2) {$\omega$};
    \node[int] (v2) at (120:.5) {};
    \node[] (w2) at (120:1.2) {$1$};
    \node[int] (v3) at (-120:.5) {};
    \node[] (w3) at (-120:1.2) {$2$};
    \draw (v1) edge (v2) edge (v3) edge (w1)
    (v2) edge (v3) edge (w2)
    (v3) edge (w3);
        \node[] at (0,-1.5) {$V_{2,1}$};
  \end{tikzpicture} }
\\ \\
  \begin{tikzpicture}
    \node[int] (v1) at (0,.8) {};
    \node[] (w1) at (0,1.5) {$\omega$};
    \node[int] (v2) at (-.5,.1) {};
    \node[] (w2) at (-1,-.4) {$1$};
    \node[] (e) at (0,-.4) {$3$};
    \node[] (w3) at (.5,.1) {$2$};
    \draw (v1) edge (v2) edge (w3) edge (w1)
    (v2) edge (e) edge (w2);
    \node (e) at (1.5,1) {$\epsilon$};
    \node (w4) at (2.2,1) {$4$};
    \node (ee) at (1.5,0) {$\epsilon$};
    \node (ww4) at (2.2,0) {$\epsilon$};
    \draw (e) edge (w4) (ee) edge (ww4);
        \node[] at (0,-1) {$V_{0,0}$};
  \end{tikzpicture}
  &
  \begin{tikzpicture}
    \node[int] (v1) at (0,.8) {};
    \node[] (w1) at (0,1.5) {$\omega$};
    \node[int] (v2) at (-.5,.1) {};
    \node[] (w2) at (-1,-.4) {$1$};
    \node[] (e) at (0,-.4) {$\epsilon$};
    \node[] (w3) at (.5,.1) {$2$};
    \draw (v1) edge (v2) edge (w3) edge (w1)
    (v2) edge (e) edge (w2);
    \node (e) at (1.5,.5) {$\epsilon$};
    \node (w4) at (2.2,.5) {$3$};
    \draw (e) edge (w4);
        \node[] at (.8,-1) {$V_{1,0}$};
  \end{tikzpicture}
  &
  \raisebox{15 pt}{
  \begin{tikzpicture}
    \node[int] (v1) at (0:.5) {};
    \node[] (w1) at (0:1.2) {$\omega$};
    \node[int] (v2) at (120:.5) {};
    \node[] (w2) at (120:1.2) {$1$};
    \node[int] (v3) at (-120:.5) {};
    \node[] (w3) at (-120:1.2) {$2$};
    \node (e) at (1.8,0) {$\epsilon$};
    \node (w4) at (2.5,0) {$3$};
    \draw (v1) edge (v2) edge (v3) edge (w1)
    (v2) edge (v3) edge (w2)
    (v3) edge (w3)
    (e) edge (w4);
        \node[] at (.5,-1.5) {$V_{2,0}$};
  \end{tikzpicture}}
\end{array}
$}
\]
Let us filter $J^\ast_{1,n}$ by the number of $\epsilon$-decorations and consider the associated spectral sequence.  On the $E_0$-page (the associated graded), the differential is $\delta_{split}$, which preserves each of the six combinatorial types. So $E_1$ is the direct sum of the homologies of the associated graph complexes, with respect to $\delta_{split}$.  Each term then has an evident interpretation in terms of weight zero cohomology of moduli spaces; for instance
\[
H(V_{1,0}) \cong \bigoplus_{j=1}^n W_0H_c(\MM_{0,n+1})[-3] \quad \quad \mbox{ and } \quad \quad H(V_{2,1}) \cong W_0H_c(\MM_{1,n+1})[-2].
\]
Note that the cohomology of $V_{i,k}$ is supported in degree $n+ i + k$.

The differential on $E_1$ is the part of $\delta_{join}$ that reduces the number of $\epsilon$-decorations by exactly 1. This further decomposes as two parts, one that eliminates an $(\epsilon,j)$-edge, and one that decreases the number of other $\epsilon$-decorations by $1$.  Therefore, the $E_1$-page is the total complex of a diagram of the following form:
\begin{equation}\label{equ:E1 diag moduli}
\adjustbox{scale=.8}{
  \begin{tikzcd}
    W_0H_c(\MM_{0,n+1})[-3] \ar{r} 
    & W_0 H_c \MM_{0,n+2}[-3] \ar{r}
    &  W_0 H_c \MM_{1,n+1}[-2] \\
    \bigoplus_{j=1}^n
    W_0 H_c (\MM_{0,n})[-3]
    \ar{r} \ar{u}
    &
    \ar{r} \ar{u}
 \bigoplus_{j=1}^n
    W_0 H_c (\MM_{0,n+1})[-3]
    &
    \ar{u}
    \bigoplus_{j=1}^n
    W_0 H_c \MM_{1,n}[-2].
  \end{tikzcd}
  }
\end{equation} 
The arrows are once again coboundary maps in excision sequences arising from stratifications of moduli spaces.

As in our computation of $H(\Jast_{0,n})$, we find it helpful to tensor once again with $\sgn_n$ and then give an algebraic interpretation for the resulting diagram.  Let $\tJast_{1,n} := \Jast_{1,n} \otimes \sgn_n$, with differential $\tilde \delta = \tilde \delta_{split} + \tilde \delta_{join}$. Similarly, let $\tilde V_{i,k} := V_{i,k} \otimes \sgn_n$.  Recall that tensoring with $\sgn_n$ amounts to equipping each generator with a total ordering of all of its edges, not just the structural edges, and imposing the usual relation that reordering the edges is multiplication by the sign of the induced permutation.  We then have 
\begin{align} 
  H(\tilde V_{0,1},\tilde \delta_{split})&\cong \Lie(x_1,\dots,x_n)[-n-1]\nonumber
  \\ 
   H(\tilde V_{1,1}, \tilde \delta_{split}) &\cong \Ass(x_1,\dots,x_n)[-n-2]   \label{equ:J assoc id}
   \\ 
     H(\tilde V_{2,1},\tilde \delta_{split}) &\cong \Ass(x_1,\dots,x_n)_{\bbS_2}[-n-3].
     \label{equ:J assoc id 2}
\end{align}

Here, $\Ass(x_1, \ldots, x_n)$ is the subspace of the free associative algebra spanned by words in which each variable appears exactly once, and $\bbS_2$ acts by reversing each word (corresponding to the symmetry reversing the orientation of a based loop).
Under the identification \eqref{equ:J assoc id}, the associative word $x_1\cdots x_n$ corresponds to the graph 
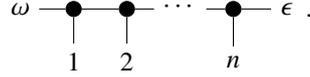
\[
\begin{tikzpicture}
    \node (w) at (0,0) {$\omega$};
    \node[int] (v1) at (0.7,0) {};
    \node[int] (v2) at (1.4,0) {};
    \node[] (vd) at (2.1,0) {$\cdots$};
    \node[int] (vn) at (2.8,0) {};
    \node[] (e) at (3.5,0) {$\epsilon$};
    \node (w1) at (0.7,-.7) {$1$};
    \node (w2) at (1.4,-.7) {$2$};
    \node (wn) at (2.8,-.7) {$n$};
    \draw (v1) edge (w) edge (v2) edge (w1)
    (vd) edge (vn) edge (v2)
    (v2) edge (w2)
    (vn) edge (e) edge (wn);
\end{tikzpicture}\, .
\]
Similarly, under the identification \eqref{equ:J assoc id 2} the associative word $x_1\cdots x_n$ corresponds to the graph
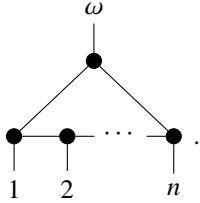
\[
\begin{tikzpicture}
    \node (w) at (1.75,1.7) {$\omega$};
    \node[int] (wi) at (1.75,1) {};
    \node[int] (v1) at (0.7,0) {};
    \node[int] (v2) at (1.4,0) {};
    \node[] (vd) at (2.1,0) {$\cdots$};
    \node[int] (vn) at (2.8,0) {};
    \node (w1) at (0.7,-.7) {$1$};
    \node (w2) at (1.4,-.7) {$2$};
    \node (wn) at (2.8,-.7) {$n$};
    \draw (v1) edge (wi) edge (v2) edge (w1)
    (vd) edge (vn) edge (v2)
    (v2) edge (w2)
    (vn) edge (wi) edge (wn)
    (wi) edge (w);
\end{tikzpicture}\, .
\]
The cohomology groups in the second row of \eqref{equ:E1 diag moduli} are similar, but one must take a direct sum over the decorations $j$ that appears in an $(\epsilon,j)$-edge, and omit the variable $x_j$,  i.e., $H(\tilde V_{0,0},\delta_{split}) \cong \bigoplus_{j=1}^n \Lie(x_1,\dots, \hat x_j,\dots, x_n)[-n]$, and so on.

With these identifications, the $E_1$-page of our spectral sequence (see diagram \eqref{equ:E1 diag moduli}) hence becomes isomorphic to the total complex of the following diagram, from which we omit degree shifts for notational brevity:\footnote{The bottom left space is concentrated in degree $n$, and each arrow increases the degree by one.}
\begin{equation}\label{equ:E1 diag pre}
\adjustbox{scale=.65}{
  \begin{tikzcd}
    \Lie(x_1,\dots,x_n) \ar{r} 
    & \Ass(x_1,\dots,x_n) \ar{r}
    &  \Ass(x_1,\dots,x_n)_{\bbS_2} \\
    \bigoplus_{j=1}^n
    \Lie(x_1,\dots, \hat x_j,\dots, x_n)
    \ar{r} \ar{u}
    &
    \ar{r} \ar{u}
    \bigoplus_{j=1}^n
    \Ass(x_1,\dots, \hat x_j,\dots, x_n)
    &
    \ar{u}
    \bigoplus_{j=1}^n
    \Ass(x_1,\dots, \hat x_j,\dots, x_n)_{\bbS_2}.
  \end{tikzcd}
  }
\end{equation} 

Each vertical arrow in \eqref{equ:E1 diag pre} is defined by the formula 
\begin{equation}\label{equ:FFFdiff}
  (F_1,\dots,F_n) \mapsto \sum_{j=1}^n [x_j,F_j].
\end{equation}
The horizontal arrows are the canonical inclusions and projections, so that the cohomology of each row of \eqref{equ:E1 diag pre} is concentrated in the middle term. It is hence easy to check (for example, using Lemma \ref{lem:acyclicity} again) that the total complex of \eqref{equ:E1 diag pre} is quasi-isomorphic to the two-term complex built from its row-wise cohomology.
It hence remains to study these middle cohomology groups of each row, and the induced map between them.

To this end let $\Pois(x_1,\dots,x_n)$ be the part of the free Poisson algebra in $x_1,\dots,x_n$ in which each $x_j$ appears exactly once.
Its elements are linear combinations of Poisson expressions of the form 
\[
F(x_1,\dots,x_n)= f_1(x_1,\dots,x_n)\wedge \cdots \wedge f_k(x_1,\dots,x_n), 
\]  
where the $f_j$ are Lie words, such that each variable $x_j$ appears exactly once in $F$. Let $\Pois^k(x_1,\dots,x_n)$ be the subspace spanned by expressions that are products of exactly $k$ Lie words. There is a natural map 
\begin{equation}\label{equ:Dk def}
D_n^k: \bigoplus_{j=1}^n \Pois^k(x_1,\dots,\hat x_j,\dots,x_n)  \to \Pois^k(x_1,\dots,x_n),
\end{equation}
given again by the expression \eqref{equ:FFFdiff}, i.e., 
$(F_1,\dots,F_n)
\mapsto 
\sum_{j=1}^n [x_j, F_j].$.  Let us define 
\begin{align*} 
  \mathbb A_n^k &:= \ker D_n^k & \mathbb B_n^k := \coker D_n^k,
\end{align*}
and 
\begin{align}\label{equ:An Bn def} 
  \mathbb A_n &:= \bigoplus_{j\geq 1} \mathbb A_n^{2j+1} & \mathbb B_n := \bigoplus_{j\geq 1} \mathbb B_n^{2j+1}.
\end{align}
All objects here are naturally $\bbS_n$-representations, by changing the variable indices and order of summands on the left-hand side of \eqref{equ:Dk def}. For example, we have $\mathbb A_1=\mathbb B_1=\mathbb A_2=\mathbb B_2=\mathbb A_3=0$,  $\mathbb B_3$ and $\mathbb A_4$ are one-dimensional trivial representations, and $\mathbb B_4$ is a three-dimensional irreducible representation of $\bbS_4$.

\begin{prop}\label{lem:J1np}
  We have $\Jast_{1,0}=0$. For $n\geq 1$ we have an isomorphism of graded $\bbS_n$-modules 
  \[
  H(\Jast_{1,n}, \delta)
    \cong
    (\mathbb A_n[-n-1] \oplus \mathbb B_n[-n-2]) \otimes \sgn_n.
  \]
\end{prop}

\begin{proof}
Equivalently, we must show that $H(\tJast_{1,n}, \tilde \delta) \cong \mathbb A_n[-n-1] \oplus \mathbb B_n[-n-2]$.  Poincar\'e-Birkhoff-Witt gives an $\bbS_n$-equivariant isomorphism $\Pois(x_1,\dots,x_n)\cong \Ass(x_1,\dots,x_n)$. Applying this to the center and righthand columns of \eqref{equ:E1 diag pre}, 
we get the diagram 
\begin{equation}\label{equ:E1 diag}
\adjustbox{scale=.65}{
  \begin{tikzcd}
    \Lie(x_1,\dots,x_n) \ar{r} 
    & \Pois(x_1,\dots,x_n) \ar{r} 
    &  \Pois(x_1,\dots,x_n)_{\bbS_2} \\
    \bigoplus_{j=1}^n
    \Lie(x_1,\dots, \hat x_j,\dots, x_n) \ar{r}\ar{u}
    &
    \bigoplus_{j=1}^n
    \Pois(x_1,\dots, \hat x_j,\dots, x_n) \ar{r} \ar{u}
    & \ar{u}
    \bigoplus_{j=1}^n
    \Pois(x_1,\dots, \hat x_j,\dots, x_n)_{\bbS_2}.
  \end{tikzcd}
  }
\end{equation}
The vertical arrows in \eqref{equ:E1 diag} are still given by \eqref{equ:FFFdiff}.
Furthermore, note that taking $\bbS_2$-coinvariants in the right-hand column is the same is reducing to Poisson words containing evenly many Lie words, i.e.,
\[
  \Pois(x_1,\dots,x_n)_{\bbS_2}
  \cong \bigoplus_{j> 0} \Pois^{2j}(x_1,\dots,x_n).
\]

Then the total complex of \eqref{equ:E1 diag} is quasi-isomorphic to the complex of its row-wise cohomology groups 
\begin{equation}\label{equ:E1 diag post} 
  \begin{tikzcd}
    \bigoplus_{k\geq 1} \Pois^{2k+1}(x_1,\dots,x_n) 
    \\
    \bigoplus_{j=1}^n
    \bigoplus_{k\geq 1}
    \Pois^{2k+1}(x_1,\dots, \hat x_j,\dots, x_n) \ar{u}{\sum_k D_n^k}.
  \end{tikzcd}\, .
\end{equation}
The cohomology of \eqref{equ:E1 diag post} is the kernel plus the cokernel of the differential $\sum_k D_n^k$, i.e., the direct sum of $\mathbb A_n$ and $\mathbb B_n$, as defined above. The proposition follows, after suitably accounting for the required degree shifts. 
\end{proof}

The above description of the spaces $\mathbb A_n$ and $\mathbb B_n$ is not very explicit.
However, we can at least provide the following lower bound on the dimensions.
\begin{lemma}\label{lem:Bn dim bound}
For $n\geq 3$ we have  
$$\dim \mathbb B_n\geq (n-2)!.$$
\end{lemma}
\begin{proof}
From the definition of $\mathbb B_n$, we have
\[
\resizebox{0.95\hsize}{!}{$
  \dim \mathbb B_n \geq \dim \left(\bigoplus_{k\geq 1} \Pois^{2k+1}(x_1,\dots,x_n) \right)
  -
  \dim \left(\bigoplus_{j=1}^n
  \bigoplus_{k\geq 1}
  \Pois^{2k+1}(x_1,\dots, \hat x_j,\dots, x_n) \right). 
  $}
\]
From the top row of \eqref{equ:E1 diag pre}, we have
\begin{center}
\scalebox{.75}{\parbox{\linewidth}{%
\begin{align*}
  \dim \left(\bigoplus_{k\geq 1} \Pois^{2k+1}(x_1,\dots,x_n) \right)
  &= 
  \dim\left( \Ass(x_1,\dots,x_n) \right)
  - \dim\left( \Ass(x_1,\dots,x_n)_{\bbS_2} \right)
  - \dim \left( \Lie(x_1,\dots,x_n) \right)
  \\&=
  n! - \frac{n!}2 - (n-1)!, 
\end{align*}
}}
\end{center}
which simplifies to $\frac12 (n-2)(n-1)!$ Here, we have used that the action of $\bbS_2$ on associative words of length $\geq 2$ is faithful in computing $\dim\left( \Ass(x_1,\dots,x_n)_{\bbS_2} \right)=\frac {n!}2$.
By the same computation (valid for $n\geq 3$) we also see that 
\begin{align*}
  \dim \left(\bigoplus_{j=1}^n
  \bigoplus_{k\geq 1}
  \Pois^{2k+1}(x_1,\dots, \hat x_j,\dots, x_n) \right)
  =
  \frac 12 n (n-3)(n-2)!
\end{align*}
Hence 
\[
  \dim \mathbb B_n \geq 
  \frac12 \left( (n-2)(n-1) - n(n-3)  \right)(n-2)!,
\]
which is $(n-2)!$, as required.
\end{proof}

\begin{proof}[Proof of Theorems \ref{cor:Psi injection} and \ref{cor:n2}]
By \eqref{eq:Xast direct sum}, we know that $H(H_{g,n})\oplus H(J_{g,n})$ injects into the cohomology of $\Xast_{g,n}$. The latter is identified with the cohomology of $X_{g,n}$, and hence with $\gr_2H_c(\MM_{g,n})$.  The summand $H(H_{g,n})$ is evaluated by Lemma \ref{lem:Hgn}. It contributes only for $n=0,1$. For $n=1$ it produces an injection of the form \eqref{equ:gr2 injection n1} on p.~\pageref{equ:gr2 injection n1}. This injection agrees with the pullback map $\pi^*$, as shown in Appendix \ref{sec:pullback map}.

For the summand $H(J_{g,n})$ we can use that by the decomposition \eqref{equ:Jgn decomp} we have an injection of 
\begin{equation}\label{equ:J otimes G}
H(\Jast_{0,n}) \otimes H(G^{(g,1)})[-1]
\oplus
H(\Jast_{1,n}) \otimes H(G^{(g-1,1)})[-1]
\end{equation}
into its cohomology.
The second factors in the tensor products compute the weight zero cohomology of $\MM_{g,1}$ and $\MM_{g-1,1}$ respectively.
The first factors $H(\Jast_{0,n})$ and $H(\Jast_{1,n})$ are evaluated by Proposition~\ref{lem:J0np} and \ref{lem:J1np} respectively. 

For $n=1$ one only has a contribution from $H(\Jast_{0,1})\cong \Q[-1]$, while $H(\Jast_{1,1})=0$.
This produces the injection \eqref{equ:Psi injection} on p.~\pageref{equ:Psi injection}.
To see that this injection indeed corresponds to multiplication with the $\psi$-class at the marking as is claimed in Theorem \ref{cor:Psi injection} we have to trace the representatives.
Say $\Gamma\in G^{(g,1)}$ is a cocycle representing a cohomology class in $W_0H_c^k(\MM_{g,1})$.
Then the corresponding class in $H(\Xast_{g,1},\tilde \delta)$ is represented by 
\[
  \Gamma'=
  \begin{tikzpicture}
    \node[circle, draw, minimum width=5mm] (vv) at (-1,.5) {$\Gamma$};
    \node (ww) at (-1,-.5) {$\omega$};
      \node(v) at (0,0) {$\omega$};
      \node(w) at (.7,0) {$1$};
      \draw (v) edge (w) (vv) edge (ww);
  \end{tikzpicture},
\]
The action of the morphism $\XPhi$ on this cocycle is trivial since there are no $\epsilon$-legs. Hence $\Gamma'$ is also a cocycle in $(X_{g,1},\tilde \delta)$ and in $X_{g,1}$.
From \eqref{equ:Phi pic1 leg} on p.~\pageref{equ:Phi pic1 leg}, we see that the corresponding class in $\tGK_{g,1}^2$ is obtained, up to a sign, by decorating the unique leg in $\Gamma$ with a $\psi$-class.  We claim that this is the graphical representation of multiplying with a $\psi$-class at the marking in $H_c^{\bullet}(\MM_{g,1})$.  To see this, note that $H^\bullet$ acts on $H^\bullet_c$ via cup product, and recall that the weight spectral sequence abutting to $\gr H^\bullet(\MM_{g,1})$ is induced by a natural filtration on the sheaf of smooth differential forms with logarithmic poles along the boundary of $\bMM_{g,1}$ \cite[Chapter 4, Section 3]{KulikovKurchanov}.  It is related via Poincar\'e duality to the spectral sequence that we study here, abutting to $\gr H^\bullet_c(\MM_{g,1})$, which is likewise induced by a natural filtration on currents.  The $\psi$-class is of pure weight 2 in $H^2(\bMM_{g,1})$ and is represented by a smooth and closed differential form $\eta$ without logarithmic poles.  The cup product $\psi \cup \colon \gr_\bullet H_c^{\bullet}(\MM_{g,1}) \to \gr_{\bullet + 2} H_c^{\bullet + 2}(\MM_{g,1})$ is therefore induced by the action of $\eta$ on currents, which also induces a compatible action on every term on every page in our weight spectral sequence.  The $\xi_\Gamma$-pullback of $\psi$ is the corresponding $\psi$-class on $\bMM_\Gamma$, and it follows that $\psi \cup$ is induced by multiplication by $\psi$ on every term of $\GK_{g,1}$.  The claim follows by descending from $\GK^2_{g,1}$ to $\bGK^2_{g,1}$ and then lifting to $\tGK^2_{g,1}$.  This finishes the proof of Theorem \ref{cor:Psi injection}.

To show Theorem \ref{cor:n2} we consider again the injection from \eqref{equ:J otimes G} to $\gr_2H_c^\bullet(\MM_{g,n})$, for $n\geq 2$.
Then Proposition~\ref{lem:J0np} directly produces the first summand in the formula of Theorem \ref{cor:n2}, and Lemma \ref{lem:J1np} produces the second summand.
The assertion on the dimension of $\mathbb B_n$ is shown in Lemma \ref{lem:Bn dim bound}. 
\end{proof}

\begin{rem} \label{rem:An}
Proposition~\ref{lem:J1np} yields an additional summand to $\gr_2 H^\bullet_c(\MM_{g,n})$ corresponding to the subspace $$(\mathbb A_n[-n-1] \otimes \sgn_n) \subset H(\Jast_{1,n}).$$
Lacking lower bounds on the dimension of $\mathbb A_n$, we omitted this summand from the statement of  Theorem \ref{cor:n2}.
\end{rem}

\subsubsection{Examples} 
Let us illustrate Theorem \ref{cor:n2} with pictures of the simplest non-trivial cohomology classes that  it provides.  First, consider the first summand in Theorem \ref{cor:n2} for $(g,n) = (3,2)$. 
This summand has dimension 1, concentrated in degree 8. A generating cocycle in $(\Xast_{3,2}, \tilde \delta)$ is given by the linear combination 
\[
  \begin{tikzpicture}
    \node[int] (v1) at (0:.5) {};
    \node[int] (v2) at (90:.5) {};
    \node[int] (v3) at (180:.5) {};
    \node[int] (v4) at (-90:.5) {};
    \node (w2) at (-90:1.2) {$\omega$};
    \draw (v2) edge (v1) edge (v3)
      (v4) edge (v1) edge (v3) edge (v2)
      (v4) edge (w2) (v1) edge (v3);
    \node (a1) at (-1, .5) {$1$};
    \node (b1) at (-1.7, .5) {$\omega$};
    \node (a2) at (-1, -.5) {$2$};
    \node (b2) at (-1.7, -.5) {$\epsilon$};
    \draw (a1) edge (b1) (a2) edge (b2);
  \end{tikzpicture} 
  \quad   - \quad 
  \begin{tikzpicture}
  \node[int] (v1) at (0:.5) {};
  \node[int] (v2) at (90:.5) {};
  \node[int] (v3) at (180:.5) {};
  \node[int] (v4) at (-90:.5) {};
  \node (w2) at (-90:1.2) {$\omega$};
  \draw (v2) edge (v1) edge (v3)
    (v4) edge (v1) edge (v3) edge (v2)
    (v4) edge (w2) (v1) edge (v3);
  \node (a1) at (-1, .5) {$2$};
  \node (b1) at (-1.7, .5) {$\omega$};
  \node (a2) at (-1, -.5) {$1$};
  \node (b2) at (-1.7, -.5) {$\epsilon$};
  \draw (a1) edge (b1) (a2) edge (b2);
\end{tikzpicture} 
\]
Here the component with internal vertices is the generator of $H^6(G^{(3,1)})$, which is isomorphic to $sW_0H_c^6(\MM_{3,1})$. The disjoint union of two isolated edges corresponds to the generator of $H(\Jast_{0,2})$ of Example \ref{ex:J02}. (Note that $H(\Jast_{0,2})$ has dimension $1$, by Proposition~\ref{lem:J0np}.) 
The corresponding cocycle in $(\Xast_{3,2},\delta)$ is obtained by applying the involution $\XPhi$. Up to a conventional overall sign that we omit, it is:
\[
  \begin{tikzpicture}
    \node[int] (v1) at (0:.5) {};
    \node[int] (v2) at (90:.5) {};
    \node[int] (v3) at (180:.5) {};
    \node[int] (v4) at (-90:.5) {};
    \node (w2) at (-90:1.2) {$\omega$};
    \draw (v2) edge (v1) edge (v3)
      (v4) edge (v1) edge (v3) edge (v2)
      (v4) edge (w2) (v1) edge (v3);
    \node (a1) at (-1, .5) {$1$};
    \node (b1) at (-1.7, .5) {$\omega$};
    \node (a2) at (-1, -.5) {$2$};
    \node (b2) at (-1.7, -.5) {$\epsilon$};
    \draw (a1) edge (b1) (a2) edge (b2);
  \end{tikzpicture} 
  \quad + \quad 3 \cdot 
    \begin{tikzpicture}
      \node[int] (v1) at (0:.5) {};
      \node[int] (v2) at (90:.5) {};
      \node[int] (v3) at (180:.5) {};
      \node[int] (v4) at (-90:.5) {};
      \node (w2) at (-90:1.2) {$\omega$};
      \draw (v2) edge (v1) edge (v3)
        (v4) edge (v1) edge (v3) edge (v2)
        (v4) edge (w2) (v1) edge (v3);
      \node (a1) at (-1, .5) {$1$};
      \node (b1) at (-1.7, .5) {$\omega$};
      \node (a2) at (-1, -.5) {$2$};
      \draw (a1) edge (b1) (a2) edge (v3);
    \end{tikzpicture} 
    \quad + \quad 
    \begin{tikzpicture}
      \node[int] (v1) at (0:.5) {};
      \node[int] (v2) at (90:.5) {};
      \node[int] (v3) at (180:.5) {};
      \node[int] (v4) at (-90:.5) {};
      \node (w2) at (-90:1.2) {$\omega$};
      \draw (v2) edge (v1) edge (v3)
        (v4) edge (v1) edge (v3) edge (v2)
        (v4) edge (w2) (v1) edge (v3);
      \node (a1) at (-1, .5) {$1$};
      \node (b1) at (-1.7, .5) {$\omega$};
      \node (a2) at (-1, -.5) {$2$};
      \draw (a1) edge (b1) (a2) edge (v4);
    \end{tikzpicture} 
    \quad - \quad (1\leftrightarrow 2) \, .
  \]
  Here ``$(1\leftrightarrow 2)$" stands for the same terms, with the labels 1 and 2 interchanged. 
  The corresponding cocycle in $\GK_{3,2}^2$ is then 
  \[
    3\cdot 
    \begin{tikzpicture}
      \node[int] (v1) at (0:.5) {};
      \node[int] (v2) at (90:.5) {};
      \node[int] (v3) at (180:.5) {};
      \node[ext, accepting] (v4) at (-90:.5) {};
      \node (w2) at (-90:1.2) {$1$};
      \draw (v2) edge (v1) edge (v3)
        (v4) edge (v1) edge (v3) edge (v2)
        (v4) edge[->-] (w2) (v1) edge (v3);
      \node (a2) at (-1, -.5) {$2$};
      \draw (a2) edge (v3);
    \end{tikzpicture}
    \quad + \quad 
    \begin{tikzpicture}
      \node[int] (v1) at (0:.5) {};
      \node[int] (v2) at (90:.5) {};
      \node[int] (v3) at (180:.5) {};
      \node[ext, accepting] (v4) at (-90:.5) {};
      \node (w2) at (-90:1.2) {$1$};
      \draw (v2) edge (v1) edge (v3)
        (v4) edge (v1) edge (v3) edge (v2)
        (v4) edge[->-] (w2) (v1) edge (v3);
      \node (a2) at (-1, -.5) {$2$};
      \draw (a2) edge (v4);
    \end{tikzpicture}
    \quad - \quad (1\leftrightarrow 2)\,.
  \]
By \cite{Tommasi07}, we know that $\gr_2 H_c^\bullet(\MM_{3,2}) \cong \Q \oplus \sgn_2$, supported in degree $8$.  Our computation shows that the summand $\sgn_2$ is naturally identified with  $H(\Jast_{0,2}) \otimes H(G^{(3,1)})[-1]$. 

Next, we turn to the second summand in Theorem \ref{cor:n2}.
The first nontrivial class is found in $\gr_2H_c^{12}(\MM_{4,3})$. It corresponds to the generator of $\mathbb B_3\otimes H^6(G^{(3,n)})[-1]$, and is represented by the following degree 12 cocycle in $(\Xast_{4,3},\tilde \delta)$.
\[
  \sum_{\sigma\in \bbS_3}\sgn(\sigma)
  \begin{tikzpicture}
\begin{scope}[xshift=-2.5cm,yshift=-.6cm]
    \node[int] (v1) at (0,.8) {};
    \node[] (w1) at (0,1.5) {$\omega$};
    \node[int] (v2) at (.4,.4) {};
    \node[int] (v3) at (.8,0) {};
    \node[] (w2) at (0,-.3) {$\sigma(2)$};
    \node[] (e) at (1.2,-.7) {$\epsilon$};
    \node[] (w3) at (-.5,.1) {$\sigma(1)$};
    \node[] (w4) at (.4,-.7) {$\sigma(3)$};
    \draw (v1) edge (v2) edge (w3) edge (w1)
    (v2) edge (w2) edge (v3)
    (v3) edge (e) edge (w4);
\end{scope}
    \node[int] (iv1) at (0:.5) {};
    \node[int] (iv2) at (90:.5) {};
    \node[int] (iv3) at (180:.5) {};
    \node[int] (iv4) at (-90:.5) {};
    \node (iw2) at (-90:1.2) {$\omega$};
    \draw (iv2) edge (iv1) edge (iv3)
      (iv4) edge (iv1) edge (iv3) edge (iv2)
      (iv4) edge (iw2) (iv1) edge (iv3);
  \end{tikzpicture} \ .
\]

Finally, as explained in Remark~\ref{rem:An}, there is one more summand of $\gr_2 H^\bullet_c(\MM_{g,n})$ coming from $\mathbb{A}_n$ which is not mentioned in Theorem~\ref{cor:n2}.  The first case where this summand is nontrivial arises when $(g,n) = (4,4)$. We then have $\mathbb A_4 \cong \Q$, and the graph cocycle in $(\Xast_{4,4},\tilde \delta)$ corresponding to $\mathbb A_4 \otimes H^6(G^{(3,1)})[-1]$ has the form 
\[
  \sum_{\sigma\in \bbS_4}\sgn(\sigma)
  \begin{tikzpicture}
\begin{scope}[xshift=-3.2cm,yshift=-.6cm]
    \node[int] (v1) at (0,.8) {};
    \node[] (w1) at (0,1.5) {$\omega$};
    \node[int] (v2) at (.4,.4) {};
    \node[int] (v3) at (.8,0) {};
    \node[] (w2) at (0,-.3) {$\sigma(2)$};
    \node[] (e) at (1.2,-.7) {$\epsilon$};
    \node[] (w3) at (-.5,.1) {$\sigma(1)$};
    \node[] (w4) at (.4,-.7) {$\sigma(3)$};
    \draw (v1) edge (v2) edge (w3) edge (w1)
    (v2) edge (w2) edge (v3)
    (v3) edge (e) edge (w4);
    \node (u1) at (1,1) {$\sigma(4)$};
    \node (u2) at (2,1) {$\epsilon$};
    \draw (u1) edge (u2);
\end{scope}
    \node[int] (iv1) at (0:.5) {};
    \node[int] (iv2) at (90:.5) {};
    \node[int] (iv3) at (180:.5) {};
    \node[int] (iv4) at (-90:.5) {};
    \node (iw2) at (-90:1.2) {$\omega$};
    \draw (iv2) edge (iv1) edge (iv3)
      (iv4) edge (iv1) edge (iv3) edge (iv2)
      (iv4) edge (iw2) (iv1) edge (iv3);
  \end{tikzpicture} \ .
\]

\appendix 
\section{A direct map $X_{g,n}\to \bGK_{g,n}^2$}\label{sec:direct map def}
We have connected the complexes $X_{g,n}$ and $\bGK_{g,n}^2$ through a zig-zag of quasi-isomorphisms. Here, we construct a direct map  
$
F \colon X_{g,n} \to \bGK_{g,n}^2
$
that fits into a homotopy commutative triangle with the maps in the zig-zag.

The map $F$ is useful for giving geometrically meaningful representatives in $\bGK_{g,n}^2$ of the cohomology classes we have constructed in $X_{g,n}$. By Theorem~\ref{thm:main quotient}, for $(g,n) \neq (1,1)$, the projection $\GK_{g,n}^2 \to \bGK_{g,n}^2$ is a quasi-isomorphism. Hence each cocycle in $\bGK_{g,n}^2$ lifts to a cocycle on $\GK_{g,n}^2$, uniquely up to exact terms. We now turn to the construction of the map $F$, which we define initially as a map of graded vector spaces.

Let $\Gamma$ be a generator for $X_{g,n}$ with $e$ structural edges.
  First, suppose $\Gamma$ has no $\epsilon$-decorations. At least one of the $\omega$-decorations must be adjacent to an internal vertex. If both $\omega$-decorations are adjacent to internal vertices, then  $F(\gamma)$ is obtained by joining the $\omega$-legs to form a new edge, and decorating either half-edge by $\psi_i$, in the antisymmetric linear combination:
  \begin{equation}\label{equ:F case 1}
  \begin{tikzpicture}
    \node[int] (v) at (-.4, .3) {};
    \node[int] (w) at (.4, .3) {};
    \node[] (vo) at (-.4, -.3) {$\scriptstyle \omega$};
    \node[] (wo) at (.4, -.3) {$\scriptstyle \omega$};
    \draw 
    (v) edge (vo) edge +(40:.5) edge +(90:.5) edge +(130:.5)
    (w) edge (wo) edge +(40:.5) edge +(90:.5) edge +(130:.5)
    ;
  \end{tikzpicture}
  \mapsto
  (-1)^e \frac 1 2 \left(
  \begin{tikzpicture}
    \node[ext,accepting] (v) at (-.4, -.1) {};
    \node[ext] (w) at (.4, -.1) {};
    \draw 
    (v) edge[bend right, ->-] (w) edge +(40:.5) edge +(90:.5) edge +(130:.5)
    (w) edge +(40:.5) edge +(90:.5) edge +(130:.5)
    ;
  \end{tikzpicture}
  -
  \begin{tikzpicture}
    \node[ext] (v) at (-.4, -.1) {};
    \node[ext, accepting] (w) at (.4, -.1) {};
    \draw 
    (v) edge +(40:.5) edge +(90:.5) edge +(130:.5)
    (w) edge[bend left, ->-] (v) edge +(40:.5) edge +(90:.5) edge +(130:.5)
    ;
  \end{tikzpicture}
  \right)
  \end{equation}
  To fix the sign, in this case and in all following cases, draw $\Gamma$ so that the left-hand $\omega$-edge precedes the right-hand $\omega$-edge in the ordering. Then the new edge comes first, the relative order of the remaining edges is preserved, and the sign is as indicated. Otherwise, if there are no $\epsilon$-decorations and only one $\omega$-decoration is adjacent to an internal vertex, then $F$ is given simply by 
    \[
  \begin{tikzpicture}
    \node[int] (v) at (-.4, .3) {};
    \node[] (w) at (.4, .3) {$\scriptstyle j$};
    \node[] (vo) at (-.4, -.3) {$\scriptstyle \omega$};
    \node[] (wo) at (.4, -.3) {$\scriptstyle \omega$};
    \draw (w) edge (wo)
    (v) edge (vo) edge +(40:.5) edge +(90:.5) edge +(130:.5)
    ;
  \end{tikzpicture}
  \mapsto 
-(-1)^e
  \begin{tikzpicture}
    \node[ext, accepting] (v) at (-.4, .3) {};
    \node[] (w) at (-.4, -.3) {$\scriptstyle j$};
    \draw 
    (v) edge[->-] (w) edge +(40:.5) edge +(90:.5) edge +(130:.5)
    ;
  \end{tikzpicture}
  \]

It remains to consider the cases where $\Gamma$ has a positive number of $\epsilon$-decorations. First, suppose both of the $\omega$-decorations are adjacent to internal vertices. By the orientation relations, we may assume these internal vertices are distinct. Call them $v$ and $v'$. Then let $\Gamma_v$ be the graph formed by making $v$ the special vertex, attaching all $\epsilon$- and $\omega$-decorated legs to that vertex and decorating it by $\delta_{0,S}$ with $S$ the union of the half-edges adjacent to the $\epsilon$-decorations and the half-edge towards $v'$.
Then $F(\Gamma)=\frac 1 2 (\Gamma_v-\Gamma_{v'})$. Pictorially:
\begin{equation}\label{equ:F case 3}
  \begin{tikzpicture}
    \node[int] (v) at (-.4, .3) {};
    \node[int] (w) at (.4, .3) {};
    \node[] (vo) at (-.4, -.3) {$\scriptstyle \omega$};
    \node[] (wo) at (.4, -.3) {$\scriptstyle \omega$};
    \node (e1) at (-.3, -.6) {$\scriptstyle \epsilon$};
    \node (e2) at (0, -.6) {$\scriptstyle \epsilon$};
    \node (e3) at (.3, -.6) {$\scriptstyle \epsilon$};
    \draw 
    (v) edge (vo) edge +(40:.5) edge +(90:.5) edge +(130:.5)
    (w) edge (wo) edge +(40:.5) edge +(90:.5) edge +(130:.5)
    (e1) edge +(0,-.4) (e2) edge +(0,-.4) (e3) edge +(0,-.4) 
    ;
  \end{tikzpicture}
  \mapsto
  (-1)^e\frac 1 2 \left(
  \begin{tikzpicture}
    \node[ext] (v) at (-.4, .3) {};
    \node[ext] (w) at (.4, .3) {};
    \node[ext] (u) at (0, -.3) {};
    \draw (u) edge[crossed] (v) edge (w) edge +(-40:.4) edge +(-90:.4) edge +(-130:.4) 
    (v)  edge +(40:.5) edge +(90:.5) edge +(130:.5)
    (w) edge +(40:.5) edge +(90:.5) edge +(130:.5)
    ;
  \end{tikzpicture}
  -
  \begin{tikzpicture}
    \node[ext] (v) at (-.4, .3) {};
    \node[ext] (w) at (.4, .3) {};
    \node[ext] (u) at (0, -.3) {};
    \draw (u) edge (v) edge[crossed] (w) edge +(-40:.4) edge +(-90:.4) edge +(-130:.4) 
    (v)  edge +(40:.5) edge +(90:.5) edge +(130:.5)
    (w) edge +(40:.5) edge +(90:.5) edge +(130:.5)
    ;
  \end{tikzpicture} 
  \right) \ .
\end{equation}

Next, if exactly one $\omega$-decoration is adjacent to an internal vertex $v$, 
we define $\Gamma_v$ as above and set $F(\gamma)=\frac 12 \Gamma_v$: 
\[
  \begin{tikzpicture}
    \node[int] (v) at (-.4, .3) {};
    \node[] (w) at (.4, .3) {$\scriptstyle j$};
    \node[] (vo) at (-.4, -.3) {$\scriptstyle \omega$};
    \node[] (wo) at (.4, -.3) {$\scriptstyle \omega$};
    \node (e1) at (-.3, -.6) {$\scriptstyle \epsilon$};
    \node (e2) at (0, -.6) {$\scriptstyle \epsilon$};
    \node (e3) at (.3, -.6) {$\scriptstyle \epsilon$};
    \draw 
    (v) edge (vo) edge +(40:.5) edge +(90:.5) edge +(130:.5)
    (w) edge (wo) 
    (e1) edge +(0,-.4) (e2) edge +(0,-.4) (e3) edge +(0,-.4) 
    ;
  \end{tikzpicture}
  \mapsto (-1)^e\frac 12 \,
  \begin{tikzpicture}
    \node[ext] (v) at (-.4, .3) {};
    \node[] (w) at (.4, .3) {$\scriptstyle j$};
    \node[ext] (u) at (0, -.3) {};
    \draw (u) edge[crossed] (v) edge (w) edge +(-40:.4) edge +(-90:.4) edge +(-130:.4) 
    (v)  edge +(40:.5) edge +(90:.5) edge +(130:.5)
    ;
  \end{tikzpicture} \ .
  \]

Finally, if neither $\omega$-decoration is adjacent to an internal vertex then we set $F(\gamma)=0$. 

\begin{prop}
The map $F \colon X_{g,n} \to \bGK_{g,n}^2$ is a map of complexes and fits into a homotopy commutative triangle 
\begin{equation*}
\begin{tikzcd}
  \tGK_{g,n} \ar{r}{\pi} \ar{d}[swap]{\Phi} & \bGK_{g,n}^2 \\
  X_{g,n}\ar{ur}[swap]{F}
\end{tikzcd}
\end{equation*}
\end{prop}

\begin{proof}
We begin by constructing a degree $-1$ map of graded vector spaces
\[
h \colon \tGK_{g,n}^2 \to \bGK_{g,n}^2,
\]
and show that it satisfies the required equation \eqref{equ:lem F Phi triangle} to give a chain homotopy from $F \circ \Phi$ to $\pi$. We will then use $h$ to show that $F$ is a map of complexes, and thereby prove the proposition.

Let $\Gamma$ be a generator for $\tGK_{g,n}^2$. We define $h(\Gamma)$ by a local replacement at the special vertex as follows.
Suppose the special vertex of $\Gamma$ is decorated by $\psi_i$, and $i$ is not joined to a leg, so the half-edge is joined to a half-edge at an internal vertex $v$. Then we let $h(\Gamma)$ be the graph obtained by merging the special vertex with $v$, and decorating the new special vertex by $\frac 1 2 \delta_{0,S}$, where $S$ is the set of half-edges incident to $v$. Pictorially:
  \[
  h \colon 
  \begin{tikzpicture}
    \node[ext] (v) at (0,0) {};
    \node[ext] (w) at (.7,0) {};
    \draw (v) edge[->-] (w) edge +(180:.5) edge +(150:.5) edge +(210:.5) 
    (w) edge +(0:.5) edge +(30:.5) edge +(-30:.5);
  \end{tikzpicture}
  \mapsto\frac12\,
  \begin{tikzpicture}
    \node[ext] (v) at (0,0) {};
    \node[ext] (w) at (.7,0) {};
    \draw (v) edge[crossed] (w) edge +(180:.5) edge +(150:.5) edge +(210:.5) 
    (w) edge +(0:.5) edge +(30:.5) edge +(-30:.5);
  \end{tikzpicture}
  \]
 To fix the sign, we assume here that the $\psi$-decorated edge is first in the ordering. 
 In all other cases we set $h(\Gamma)=0$.

We claim that $h$ satisfies the required equation to give a chain homotopy from $F \circ \Phi$ to $\pi$, i.e.:
\begin{equation}\label{equ:lem F Phi triangle}
  F(\Phi(\Gamma)) - \pi(\Gamma) = 
\delta h(\Gamma) + h(\delta(\Gamma)).   
\end{equation}
We prove this case-by-case, according to the decoration on the special vertex of $\Gamma$.

First, suppose the decoration on the special vertex of the graph $\Gamma\in \tGK_{g,n}$ is $\delta_{0,S}$. Then $\Phi(\Gamma)=0$ and $h(\Gamma)=0$.
Furthermore, $\pi(\Gamma)$ is the class in $\bGK^2_{g,n}$ represented by the decorated graph $\Gamma$. Note that $\delta \Gamma$ consists of some terms with $\delta_{0,T}$-decorations, plus terms with $\psi$-decoration, and only the latter contribute upon applying $h$. Furthermore, they contribute exactly $\Gamma$. Pictorially:
\[
h\left(\delta \begin{tikzpicture}
  \node[ext] (v) at (0,0) {};
  \node[ext] (w) at (.7,0) {};
  \draw (v) edge[crossed] (w) edge +(180:.5) edge +(150:.5) edge +(210:.5) 
  (w) edge +(0:.5) edge +(30:.5) edge +(-30:.5);
\end{tikzpicture}\right) = 
h\left(-\begin{tikzpicture}
  \node[ext] (v) at (0,0) {};
  \node[ext] (w) at (.7,0) {};
  \draw (v) edge[->-] (w) edge +(180:.5) edge +(150:.5) edge +(210:.5) 
  (w) edge +(0:.5) edge +(30:.5) edge +(-30:.5);
\end{tikzpicture}
-
\begin{tikzpicture}
  \node[ext] (v) at (0,0) {};
  \node[ext] (w) at (.7,0) {};
  \draw (v)  edge +(180:.5) edge +(150:.5) edge +(210:.5) 
  (w) edge +(0:.5) edge +(30:.5) edge +(-30:.5)
  (w) edge[->-] (v);
\end{tikzpicture} +(\cdots)
\right)
=
-2 \frac12 \begin{tikzpicture}
  \node[ext] (v) at (0,0) {};
  \node[ext] (w) at (.7,0) {};
  \draw (v) edge[crossed] (w) edge +(180:.5) edge +(150:.5) edge +(210:.5) 
  (w) edge +(0:.5) edge +(30:.5) edge +(-30:.5);
\end{tikzpicture},
\]
with $(\cdots)$ representing terms killed by $h(-)$.

Next, suppose that the special vertex of $\Gamma$ is decorated by $\psi_{i}$, with $i$ a half-edge incident to the special vertex. Then, $\pi(\Gamma)$ is once again the class in $\bGK_{g,n}^2$ represented by the decorated graph $\Gamma$.  If the half-edge $i$ connects to a leg, then $F(\Phi(\Gamma))=\Gamma$ and $h(\Gamma)=h(\delta\Gamma)=0$, so \eqref{equ:lem F Phi triangle} holds.

If the half-edge $i$ points towards an internal vertex, then $F(\Phi(\Gamma))$ is obtained by anti-symmetrizing over the two $\psi$-decorations one can put on the $i$-edge, 
\[
  F(\Phi(\Gamma)) = \frac 1 2 
  \left(
    \begin{tikzpicture}
      \node[ext] (v) at (0,0) {};
      \node[ext] (w) at (.7,0) {};
      \draw (v) edge[->-] (w) edge +(180:.5) edge +(150:.5) edge +(210:.5) 
      (w) edge +(0:.5) edge +(30:.5) edge +(-30:.5);
    \end{tikzpicture}
    -
    \begin{tikzpicture}
      \node[ext] (v) at (0,0) {};
      \node[ext] (w) at (.7,0) {};
      \draw (v)  edge +(180:.5) edge +(150:.5) edge +(210:.5) 
      (w) edge +(0:.5) edge +(30:.5) edge +(-30:.5)
      (w) edge[->-] (v);
    \end{tikzpicture}
  \right).
\]
In this case, we have 
\[
  \delta h(\Gamma) + h(\delta\Gamma) 
  = -\frac 1 2 
  \left(
    \begin{tikzpicture}
      \node[ext] (v) at (0,0) {};
      \node[ext] (w) at (.7,0) {};
      \draw (v) edge[->-] (w) edge +(180:.5) edge +(150:.5) edge +(210:.5) 
      (w) edge +(0:.5) edge +(30:.5) edge +(-30:.5);
    \end{tikzpicture}
    +
    \begin{tikzpicture}
      \node[ext] (v) at (0,0) {};
      \node[ext] (w) at (.7,0) {};
      \draw (v)  edge +(180:.5) edge +(150:.5) edge +(210:.5) 
      (w) edge +(0:.5) edge +(30:.5) edge +(-30:.5)
      (w) edge[->-] (v);
    \end{tikzpicture}
  \right),
\]
so \eqref{equ:lem F Phi triangle} holds.

Finally, suppose the special vertex of $\Gamma$ is decorated by $E_{ij}$, with $i,j$ half-edges incident to the special vertex. Then $\pi(\Gamma)=0$ and $h(\Gamma) = 0$. We must show that $F(\Phi(\Gamma)) = h(\delta\Gamma)$, and we consider cases according to whether $i$ and $j$ are joined to legs. If both are joined to legs, then $F(\Phi(\Gamma))=h(\delta\Gamma)=0$.  If only $j$ is joined to a leg, then one computes, with $e$ the number of structural edges of $\Gamma$,
\[
F(\Phi(\Gamma))
=
(-1)^e
\frac12 
\begin{tikzpicture}
  \node[ext] (v) at (-.4, .3) {};
  \node[] (w) at (.4, .3) {$\scriptstyle j$};
  \node[ext] (u) at (0, -.3) {};
  \draw (u) edge[crossed] (v) edge (w) edge +(-40:.4) edge +(-90:.4) edge +(-130:.4) 
  (v)  edge +(40:.5) edge +(90:.5) edge +(130:.5)
  ;
\end{tikzpicture} ,
\quad \quad \mbox{ and } \quad \quad 
    h(\delta\Gamma)
    =
    h(\delta_{res}\Gamma)
    =
    (-1)^e
    \frac12 
    \begin{tikzpicture}
      \node[ext] (v) at (-.4, .3) {};
      \node[] (w) at (.4, .3) {$\scriptstyle j$};
      \node[ext] (u) at (0, -.3) {};
      \draw (u) edge[crossed] (v) edge (w) edge +(-40:.4) edge +(-90:.4) edge +(-130:.4) 
      (v)  edge +(40:.5) edge +(90:.5) edge +(130:.5)
      ;
    \end{tikzpicture},
\]
as required.  If neither $i$ nor $j$ connects to a leg, then they connect to distinct internal vertices and one computes 
\[
  F(\Phi(\Gamma))
  =
  (-1)^e \frac 1 2 \left(
    \begin{tikzpicture}
      \node[ext] (v) at (-.4, .3) {};
      \node[ext] (w) at (.4, .3) {};
      \node[ext] (u) at (0, -.3) {};
      \draw (u) edge[crossed] (v) edge (w) edge +(-40:.4) edge +(-90:.4) edge +(-130:.4) 
      (v)  edge +(40:.5) edge +(90:.5) edge +(130:.5)
      (w) edge +(40:.5) edge +(90:.5) edge +(130:.5)
      ;
    \end{tikzpicture}
    -
    \begin{tikzpicture}
      \node[ext] (v) at (-.4, .3) {};
      \node[ext] (w) at (.4, .3) {};
      \node[ext] (u) at (0, -.3) {};
      \draw (u) edge (v) edge[crossed] (w) edge +(-40:.4) edge +(-90:.4) edge +(-130:.4) 
      (v)  edge +(40:.5) edge +(90:.5) edge +(130:.5)
      (w) edge +(40:.5) edge +(90:.5) edge +(130:.5)
      ;
    \end{tikzpicture}
    \right)
    =
    h(\delta\Gamma).
\]
This completes the proof of \eqref{equ:lem F Phi triangle}.

Having proved \eqref{equ:lem F Phi triangle}, we have that for all $\Gamma\in \tGK_{g,n}$
\[
\delta  F(\Phi(\Gamma))
-
F(\delta \Phi(\Gamma))
=
0.
\]
By Proposition~\ref{prop:Phiqiso}, the map $\Phi$ is surjective, and hence $\delta \circ F = F \circ \delta$. We conclude that $F$ is a map of chain complexes, and $h$ is chain homotopy from $F \circ \Phi$ to $\pi$, as required.
\end{proof}

\subsection{The pullback map}\label{sec:pullback map}
Recall that the projection $\pi \colon \MM_{g,1} \to \MM_g$ is proper and extends to a morphism on the stable curves compactifications, also denoted $\pi \colon \bMM_{g,1} \to \bMM_g$.  Recall also that the boundary divisors $\partial \MM_{g,1} := \bMM_{g,1} \smallsetminus \MM_{g,1}$ and $\partial \MM_g := \bMM_g \smallsetminus \MM_g$ have normal crossings.  Applying the ``fundamental simplicial constructions" as in \cite[\S\textbf{II}.I]{ElZein83} gives an induced morphism $\pi^*$ from the simplicial resolution of the constant sheaf $\Q$ on $\partial \MM_g$ to that of the constant sheaf $\Q$ on $\partial \MM_{g,1}$, induced by the normal crossings structure.  After realizing the compactly supported cohomology of $\MM_{g,1}$ and of $\MM_g$ as the reduced cohomology of the mapping cones for the inclusions $\partial \MM_{g,1} \to \bMM_{g,1}$ and $\partial \MM_g \to \bMM_g$, respectively, we get an induced pullback morphism between the weight spectral sequences that abuts to $\pi^* \colon \gr H_c (\MM_g) \to \gr H_c (\MM_{g,1})$.

We consider the graphical interpretation of the the induced map on the $E_1$-page of the weight spectral sequence, i.e., the induced map between the Getzler-Kapranov graph complexes.  We claim that 
$
  \pi^* \colon \GK_{g,0} \to \GK_{g,1}
$
is given by the formula 
\begin{equation} \label{eq:pv}
\pi^*\Gamma = \sum_{v} p_v(\Gamma),  
\end{equation}
where the sum is over all vertices of the graph $\Gamma\in\GK_{g,0}$, and $p_v(\Gamma)$ is obtained from $\Gamma$ by attaching a leg labeled by $1$ to vertex $v$. If the old decoration of $v$ is $\alpha\in H^{i}(\bMM_{h,k})$, then the new decoration on the vertex is given by $\pi^* \alpha \in H^{i}(\bMM_{h,k+1})$, with $\pi\colon \bMM_{h,k+1}\to \bMM_{h,k}$ forgetting the new marking.  To see this, note that if $\Gamma$ has $j$ edges then the underlying graphs of $p_v(\Gamma)$, obtained by attaching a leg labeled $1$ to a vertex $v$ of $\Gamma$, correspond precisely to the strata of codimension $j$ that map onto the codimension $j$ stratum $\xi_\Gamma(\MM_\Gamma) \subset \bMM_{\Gamma}$, and the operation on decorations that we have described corresponds, via the K\"unneth decomposition, to the induced pullback on the cohomology of strata in the simplicial resolution.

The same formula \eqref{eq:pv} also defines a pullback operation 
\[
  \pi^* \colon \bGK_{g,0} \to \bGK_{g,1},
\]
and the projection map $\GK_{g,n}\to \bGK_{g,n}$ intertwines the two operations $\pi^*$. 
We furthermore define another operation 
\[
\varpi^* \colon (H_g,\tilde \delta) \to (\Xast_{g,1},\tilde \delta)
\]
by taking minus the disjoint union with an $(\epsilon,1)$-edge,
\[
  \varpi^* \colon
  \begin{tikzpicture}
    \node[draw, circle, minimum width=5mm] (v) at (0,0) {$\Gamma$};
  \end{tikzpicture}
  \mapsto 
  -
  \,
  \begin{tikzpicture}
    \node[draw, circle, minimum width=5mm] (v) at (0,0) {$\Gamma$};
    \node (v1) at (.7,0) {$\epsilon$};
    \node (v2) at (1.5,0) {$1$};
    \draw (v1) edge (v2);
  \end{tikzpicture}\, .
\]
\begin{lemma}
The following diagram commutes 
\begin{equation}\label{equ:pi diag}
\begin{tikzcd}
  (\Xast_{g,1},\tilde \delta) \ar{r}{\XPhi}
  &
  (\Xast_{g,1},\delta)
  \ar[hookrightarrow]{r}
  &
  X_{g,1}
  \ar{r}{F}
  &
  \bGK_{g,1}
  \\
    (H_g,\tilde \delta) \ar{r}{\XPhi} \ar{u}{\varpi^*} 
  &
  (\Xast_{g},\delta)
  \ar[hookrightarrow]{r}
  &
  X_{g}
  \ar{r}{F}
  &
  \bGK_{g,}
  \ar{u}{\pi^*}
\end{tikzcd}\, .
\end{equation}
\end{lemma}
\begin{proof}[Proof sketch]
A graph $\Gamma\in H_g$ contains either no $\epsilon$-leg, or exactly two, on an $(\epsilon,\epsilon)$-edge.
We consider both cases in turn.
Suppose first that $\Gamma\in H_g$ contains no $\epsilon$-leg.
Then its image under the first two horizontal maps in \eqref{equ:pi diag} is just the same unaltered graph $\Gamma$.
Applying $F$ produces the linear combination \eqref{equ:F case 1}. Applying $\pi^*$ afterwards produces a linear combination of graphs of the form
\begin{equation}\label{equ:pi first case}
  -(-1)^e \frac 1 2
  \left( 
    \sum 
    \begin{tikzpicture}
      \node[ext,accepting] (v) at (-.4, .3) {};
      \node[ext] (w) at (.4, .3) {};
      \node (v1) at (.4, -1) {$1$};
      \draw 
      (v) edge[bend right, ->-] (w) edge +(40:.5) edge +(90:.5) edge +(130:.5)
      (w) edge +(40:.5) edge +(90:.5) edge +(130:.5)
      (w) edge[dashed] (v1)
      ;
    \end{tikzpicture}
    -
    \sum
    \begin{tikzpicture}
      \node[ext] (v) at (-.4, .3) {};
      \node[ext, accepting] (w) at (.4, .3) {};
      \node (v1) at (.4, -1) {$1$};
      \draw 
      (v) edge +(40:.5) edge +(90:.5) edge +(130:.5)
      (w) edge[bend left, ->-] (v) edge +(40:.5) edge +(90:.5) edge +(130:.5)
      (w) edge[dashed] (v1)
      ;
    \end{tikzpicture}
    -
    \begin{tikzpicture}
      \node[ext] (v) at (-.4, .3) {};
      \node[ext] (w) at (.4, .3) {};
      \node[ext] (u) at (0, -.3) {};
      \node (v1) at (0,-1) {$1$};
      \draw (u) edge[crossed] (v) edge (w) 
      edge (v1)
      (v)  edge +(40:.5) edge +(90:.5) edge +(130:.5)
      (w) edge +(40:.5) edge +(90:.5) edge +(130:.5)
      ;
    \end{tikzpicture}
    +
    \begin{tikzpicture}
      \node[ext] (v) at (-.4, .3) {};
      \node[ext] (w) at (.4, .3) {};
      \node[ext] (u) at (0, -.3) {};
      \node (v1) at (0,-1) {$1$};
      \draw (u) edge (v) edge[crossed] (w) 
      edge (v1)
      (v)  edge +(40:.5) edge +(90:.5) edge +(130:.5)
      (w) edge +(40:.5) edge +(90:.5) edge +(130:.5)
      ;
    \end{tikzpicture}
  \right)
\end{equation}
In the first two sums one sums over all ways of attaching the $1$-labeled leg to a vertex, including the special vertex.
Here we are using \cite[Proposition 3.1 (ii)]{ArbarelloCornalba98} to compute the pullback of the $\psi$-class, and the additional $\delta_{0,S}$-terms there produce the last two summands.
To compare, let us now follow our graph $\Gamma$ along the lower rim of \eqref{equ:pi diag}.
The vertical arrow $\varpi^*$ adds one $(\epsilon,1)$-edge to $\Gamma$. Call the resulting graph $-\Gamma'\in \Xast_{g,1}$.
Applying $\Xi$ then produces $\Gamma'$ plus the sum of all graphs obtained by attaching a $1$-labeled leg to $\Gamma$.
Applying $F$ produces from this latter sum the two sums in \eqref{equ:pi first case}.
Applying $F$ to $\Gamma'$ produces the last two summands in \eqref{equ:pi first case} via \eqref{equ:F case 3}.
This shows commutativity of \eqref{equ:pi diag} for $\Gamma\in H_g$ without $(\epsilon,\epsilon)$-edge.
The proof in the case where $\Gamma\in H_g$ has an $(\epsilon,\epsilon)$-edge is essentially similar, although the map $\XPhi$ produces more terms in presence of an $(\epsilon,\epsilon)$-edge. 
\end{proof}

\section{Numerical results}

We record the cohomology groups of the graph complexes $G^{(g,n)}$, $X_{g,n}$, and $\Jast_{g,n}$, with the characters of their $\bbS_n$-representations for small $g$ and $n$, obtained by calculations in Sage. For example, the entry $t^{11}s_{2,1}+ t^{12}s_{1,1,1}$ in Figure~\ref{Fig:Ggn} for $(g,n) = (4,3)$ indicates that   $W_0H^{11}_c(\MM_{4,3})$ is the $2$-dimensional irreducible representation of $\bbS_3$ with character $s_{2,1}$, $W_0H^{12}_c(\MM_{4,3})$ is the sign representation, and $W_0H^\bullet_c(\MM_{4,3})$ vanishes in all other degrees.  The full cohomology of $\MM_{g,n}$ is well-understood for $g = 0$, since $\MM_{0,n}$ is a hyperplane arrangement complement, and also for $g = 1$ \cite{Getzler, Gorinov}.  For $g = 2$, we understand that Tommasi has computed the cohomology of $\MM_{2,n}$, for $n \leq 5$ \cite{Tommasi}, cf. \cite{BCGY} for weight 0 and $n \leq 4$. The additional computations in Figure~\ref{Fig:Ggn} were previously known as $\Q$-vector spaces \cite[Appendix~A]{CGP2}.  We include the details here to facilitate comparison with the computations in Figures~\ref{Fig:Xgn} and \ref{Fig:Jgn}.  To the best of our understanding, the computations in Figure~\ref{Fig:Xgn} are new for $g \geq 3$, as are all of the computations in Figure~\ref{Fig:Jgn}.

\vskip -5 pt
\begin{figure}[h]
\[
\scalebox{.8}{
\begin{tabular}{|g|M|M|M|M|M|M|M|}
    \hline
    \rowcolor{Gray}
    g,n & 0& 1 & 2 & 3 & 4 & 5 & 6 \\
    \hline
    0 & - & - & - & $s_3$ & $t^1 s_{2,2}$ & $t^2s_{3,1,1}$ & $t^3(s_{3,3} + s_{4,1,1}+s_{2,2,1,1})$\\ \hline
    1 & - & 0 & 0 & $t^3 s_{1,1,1}$ & $t^4s_{3,1}$ & $\scriptstyle{t^5(s_5+s_{3,2}+s_{2,2,1}+s_{1,1,1,1,1})}$& \\ \hline
    2 & 0 & 0 & $t^5 s_2$ & 0 & $t^6 s_4+t^7s_{2,1,1}$ &  & \\ \hline
    3 & $t^6$ & $t^6s_1 $ & 0 & $t^9 s_3$ & $t^9s_{2,1,1}+t^{10}s_{2,2}$&  & \\ \hline
    4 & 0 & 0 & 0 & $t^{11}s_{2,1}+ t^{12}s_{1,1,1}$ &  & & \\\hline
    5 & $t^{10}$ & $t^{10} s_1$ & $t^{14}s_2$ & $t^{13}s_3+t^{14}s_{2,1}$ &  & & \\\hline
    6 & $t^{15}$ & $t^{15}s_1$ &  &  & & & \\ 
    \hline
\end{tabular}
}
\]
\caption{\small{The $\bbS_n$-equivariant Poincar\'e polynomials of $H^\bullet(G^{(g,n)}) \cong W_0H^\bullet_c(\MM_{g,n})$ for small $g$ and $n$.}} \label{Fig:Ggn}
\end{figure}
\begin{figure}[h]
\[
\scalebox{.7}{
\begin{tabular}{|g|M|M|M|M|M|M|M|}
    \hline
    \rowcolor{Gray}
    g,n & 0& 1 & 2 & 3 & 4 & 5 & 6\\
    \hline
    0 & - & - & - & 0 & $t^2 s_4$ & $t^3 s_{3,2}$ & $t^4(s_{3,2,1}+s_{4,1,1})$\\ \hline
    1 & - & 0 & 0 & 0 & $t^5 s_{2,1,1}$ & $t^6(s_{3,1,1}+s_{3,2}+s_{4,1})$&  \\ \hline
    2 & 0 & 0 & 0 & $ t^7(s_3+s_{2,1})$ & $t^8s_4$ &  & \\ \hline
    3 & 0 & $t^8 s_1$ & $t^8(s_2 + s_{1,1})$ &  $2t^9 s_3$ &  & & \\ \hline
    4 & 0 & 0 & 0 & $2t^{12}s_{1,1,1}$ &  & & \\ \hline
    5 & 0 & $t^{12} s_1$ & $t^{12}(s_2+s_{1,1})+2t^{14}s_2$ &  & & & \\\hline
    6 & 0 & $2t^{15}s_1+t^{17}s_1$ &  &  & & & \\
    \hline
\end{tabular}
}
\]
\caption{\small{The $\bbS_n$-equivariant Poincar\'e polynomials of $H^\bullet(X_{g,n}) \cong \gr_2 H^\bullet_c(\MM_{g,n})$ for small $g$ and $n$.}} \label{Fig:Xgn}
\end{figure}

\begin{figure}[h]
\[
\scalebox{.7}{
\begin{tabular}{|g|M|M|M|M|M|M|M|}
    \hline
    \rowcolor{Gray}
    g,n & 1 & 2 & 3 & 4 & 5 & 6\\
    \hline
    0 & $t s_1$ & $t s_{1,1}$ & $t^2 s_3$ & $t^3s_{2,2}$ & $t^4s_{3,1,1}$ & $\scriptstyle{t^5 (s_{4,1,1} + s_{3,3}+s_{2,2,1,1})}$ \\ \hline
    1 & 0 & 0 & $t^5 s_{1,1,1}$& $\scriptstyle{t^5s_{1,1,1,1} +t^6s_{3,1}}$ & $\!\!\scriptstyle{t^6s_{3,1,1}+  t^7(s_5+s_{3,2}+s_{2,2,1}+}$ $\scriptstyle{s_{1,1,1,1})}$ & \\ \hline
    2 & 0 & $t^7 s_2$ & $t^7 s_{2,1}$ & $\scriptstyle{t^8(2 s_4+s_{3,1})+t^9s_{2,1,1}}$ &  & \\ \hline
    3 & $t^8 s_1$ & $t^8 s_{1,1}$ & $t^9 s_3 + t^{11} s_3$&  & & \\ \hline
    4 & 0 & 0 & $\scriptstyle{t^{12}s_{1,1,1}+t^{13}(s_3+s_{2,1})+}$ $\scriptstyle{t^{14}s_{1,1,1}}$&  & & \\ \hline
    5 & $ t^{12} s_1$ & $\scriptstyle{t^{12}s_{1,1}+t^{14}(s_2+s_{1,1})+t^{16}s_2}$&  & & & \\
    \hline
\end{tabular}
}
\]
\caption{\small{The $\bbS_n$-equivariant Poincar\'e polynomials of $H^\bullet(\Jast_{g,n})$ for small $g$ and $n$.}} \label{Fig:Jgn}
\end{figure}

\end{document}